\documentclass[11pt,reqno]{amsart}

\usepackage{amsmath, amsthm, amssymb}
\usepackage{amsfonts}

\usepackage{hyperref}

\usepackage[ansinew]{inputenc}
\usepackage[dvips]{epsfig}
\usepackage{graphicx}
\usepackage[english]{babel}

\usepackage{thmtools}
\theoremstyle{plain}
\declaretheorem[title=Theorem, parent=section]{theorem}
\declaretheorem[title=Lemma,sibling=theorem]{lemma}
\declaretheorem[title=Proposition,sibling=theorem]{proposition}
\declaretheorem[title=Corollary,sibling=theorem]{corollary}

\theoremstyle{definition}
\declaretheorem[title=Definition,sibling=theorem]{definition}
\declaretheorem[title=Remark,sibling=theorem]{remark}
\declaretheorem[title=Remark, numbered=no]{remark*}

\declaretheorem[title=Assumption, numbered=no]{assumption*}

\numberwithin{equation}{section}

\usepackage[backgroundcolor=white, bordercolor=blue,
linecolor=blue]{todonotes}

\usepackage{dsfont}
\usepackage{bbm}

\newcommand{\N}{\mathbb{N}}
\newcommand{\R}{\mathbb{R}}

\newcommand{\cD}{\mathcal{D}}
\newcommand{\cN}{\mathcal{N}}

\newcommand{\cE}{\mathcal{E}}

\newcommand{\cM}{\mathcal{M}}
\newcommand{\cI}{\mathcal{I}}

\newcommand{\eps}{\varepsilon}

\newcommand{\1}{\mathbbm{1}}

\DeclareMathOperator{\dist}{dist}

\DeclareMathOperator{\diam}{diam}
\DeclareMathOperator{\supp}{supp}

\DeclareMathOperator{\tail}{Tail}

\renewcommand{\d}{\textnormal{\,d}}

\newcommand{\average}{{\mathchoice {\kern1ex\vcenter{\hrule height.4pt
width 6pt depth0pt} \kern-9.7pt} {\kern1ex\vcenter{\hrule
height.4pt width 4.3pt depth0pt} \kern-7pt} {} {} }}
\newcommand{\dashint}{\average\int}

\begin{document}
\allowdisplaybreaks
 \title[Optimal regularity for nonlocal elliptic equations and free boundary problems]{Optimal regularity for nonlocal elliptic equations \\ and free boundary problems}

\author{Xavier Ros-Oton}
\author{Marvin Weidner}

\address{ICREA, Pg. Llu\'is Companys 23, 08010 Barcelona, Spain \& Universitat de Barcelona, Departament de Matem\`atiques i Inform\`atica, Gran Via de les Corts Catalanes 585, 08007 Barcelona, Spain \& Centre de Recerca Matem\`atica, Barcelona, Spain}
\email{xros@icrea.cat}

\address{Departament de Matem\`atiques i Inform\`atica, Universitat de Barcelona, Gran Via de les Corts Catalanes 585, 08007 Barcelona, Spain}
\email{mweidner@ub.edu}



\allowdisplaybreaks

\begin{abstract}
In this article we establish for the first time the $C^s$ boundary regularity of solutions to nonlocal elliptic equations with kernels $K(y)\asymp |y|^{-n-2s}$.
This was known to hold only when $K$ is {homogeneous}, and it is quite surprising that it holds for general inhomogeneous kernels, too.

As an application of our results, we also establish the optimal $C^{1+s}$ regularity of solutions to obstacle problems for general nonlocal operators with kernels $K(y)\asymp |y|^{-n-2s}$.
Again, this was only known when $K$ is {homogeneous}, and it solves a long-standing open question in the field.

A new key idea is to construct a 1D solution as a minimizer of an appropriate nonlocal one-phase free boundary problem, for which we establish optimal $C^s$ regularity and non-degeneracy estimates.
\end{abstract}

\allowdisplaybreaks

\maketitle

\section{Introduction}  

\addtocontents{toc}{\protect\setcounter{tocdepth}{1}}

Nonlocal elliptic operators of the type
\begin{align}
\label{eq:L}
L u(x) = 2~\text{p.v} \int_{\R^n} \big(u(x) - u(x+h) \big) K(h) \d h
\end{align}
with kernels $K : \R^n \to [0,\infty]$ satisfying
\begin{align}\label{eq:Kcomp}
0<\frac{\lambda}{|h|^{n+2s}} \leq K(h) \leq \frac{\Lambda}{|h|^{n+2s}},\qquad K(h)=K(-h),\qquad s\in(0,1),
\end{align}
have been widely studied in the last decades, with important works of Bass and Levin \cite{BaLe02,BaLe02b,Bas09}, Kassmann \cite{Kas09}, as well as the famous series of papers by Caffarelli and Silvestre \cite{CaSi09,CaSi11a,CaSi11b}; see also \cite{Sil06}, \cite{BFV14}, \cite{CCV11}, \cite{DKP14}, \cite{FKV15}, \cite{DKP16}, \cite{Coz17},  \cite{DyKa20}, \cite{CKW20}, \cite{DRSV22}, \cite{FeRo23}, \cite{KaWe23}.

These works study mostly the interior regularity of solutions to nonlocal elliptic equations of the type
\begin{equation}
\label{eq:PDE-g}
\left\{\begin{array}{rcl}
L v &=& f ~~ \text{ in }\  \Omega,\\
v &=& g ~~ \text{ in }\ \R^n \setminus \Omega,
\end{array}\right.
\end{equation}
as well as the corresponding generalizations to nonlinear and/or $x$-dependent operators.

The fine boundary regularity for problems of the type \eqref{eq:PDE-g} was first established for the fractional Laplacian, in which case one can use some explicit sub- and supersolutions (see \cite{Get61}, \cite{Lan72}, \cite{ChSo98}, \cite{CKS10}, \cite{Dyd12}, \cite{RoSe14}) to prove that solutions of \eqref{eq:PDE-g} with $L=(-\Delta)^s$ satisfy 
\begin{align}
\label{Cs-intro}
u\in C^s(\overline\Omega),
\end{align}
provided that $\Omega$ is $C^{1,1}$, $f\in L^\infty(\Omega)$, and $g=0$ (see \cite{RoSe14}).

These results were later extended and improved in \cite{Gru15}, \cite{RoSe16a}, \cite{RoSe16b}, \cite{RoSe17}, \cite{AbRo20} to more general operators $L$ of the type \eqref{eq:L}-\eqref{eq:Kcomp}, provided that $K$ is \emph{homogeneous}, i.e.,
\begin{align}
\label{homogeneous}
K(h)=\frac{K(h/|h|)}{|h|^{n+2s}}.
\end{align}
For operators with homogeneous kernels, it is still possible to construct explicit barriers, since
\begin{align}
\label{1d-hom}
L(x_n)_+^s=0\quad \textrm{in}\quad \{x_n>0\}.
\end{align}
This property completely fails for general (inhomogeneous) kernels of the type \eqref{eq:Kcomp}.
Furthermore, as noticed in \cite[Section~2]{RoSe16a}, there are fully nonlinear operators $I$ with kernels $K$ satisfying \eqref{eq:Kcomp} for which the corresponding solutions fail to be $C^s(\overline\Omega)$, even in dimension $n=1$.

This means that, for \emph{fully nonlinear} operators, $C^s$ boundary regularity only holds true for homogeneous kernels, and fails for the general class \eqref{eq:Kcomp}.
For this reason, it was expected that the same could happen for linear equations, but no counterexamples had been constructed.

The goal of this paper is to prove, quite surprisingly, that in case of \emph{linear} equations \eqref{eq:L}-\eqref{eq:PDE-g} the $C^s$ boundary regularity does hold true without \emph{any} additional assumptions on the kernels.
Namely, our first main result is the following.

\begin{theorem}\label{thm-Cs-intro}
Let $\Omega\subset \R^n$ be a $C^{1,\alpha}$ domain, and $L$, $K$, $s$, $\lambda$, and $\Lambda$ be as in \eqref{eq:L}-\eqref{eq:Kcomp}.
Let $f\in L^\infty(B_1 \cap \Omega)$, and $u$ be any weak solution of
\begin{equation}
\label{eq:PDE}
\left\{\begin{array}{rcl}
L u &=& f ~~ \text{ in }\  B_1 \cap \Omega,\\
u &=& 0 ~~ \text{ in }\ B_1 \setminus \Omega.
\end{array}\right.
\end{equation}
Then, we have
\begin{align*}
\|u\|_{C^s(B_{1/2} \cap \overline\Omega)} \leq C \left( \|u\|_{L^\infty(\R^n)} + \|f\|_{L^\infty(B_1 \cap \Omega)} \right),
\end{align*}
where $C$ depends only on $\Omega$, $s$, $\lambda$, $\Lambda$.

Moreover, if $u\geq0$ in $\R^n$ and $f\geq0$, then either $u\equiv0$ or
\begin{align*}
u \geq c d^s \quad \textrm{in}\quad B_{1/2} \cap \Omega
\end{align*}
with $c>0$, and where $d(x):={\rm dist}(x,\Omega^c)$.
\end{theorem}

We believe this result to be quite surprising, in view of the above-mentioned counterexamples for nonlinear operators.
Recall that, prior to our result, the best known boundary regularity result for general operators \eqref{eq:L}-\eqref{eq:Kcomp} was a $C^\eps(\overline\Omega)$ estimate, for some $\eps>0$ small enough (see \cite{CaSi11b}, \cite{KKP16}, \cite{FeRo24}).

\begin{remark} Notice that, if $g$ is regular enough, we can always reduce \eqref{eq:PDE-g} to the case $g=0$, by extending $g$ to $\overline\Omega$ and subtract it from $u$ (see \cite{AuRo20}, \cite{FeRo24}, or \autoref{lemma:Lipschitz-Holder-prelim}).
Moreover, for applications to the obstacle problem (explained below), it is important to have a localized version of the result in $B_1$, instead of the global problem  \eqref{eq:PDE-g}.
This is why we always consider solutions of \eqref{eq:PDE}.
\end{remark}

As explained later on, we will also establish $C^{s-\eps}$ estimates in flat-Lipschitz domains (see \autoref{thm:C-s-eps}), as well as a finer description of solutions near~$\partial\Omega$ (see \autoref{thm-expansion-intro} or \autoref{thm:inhom-Cs}) for general operators \eqref{eq:L}-\eqref{eq:Kcomp}.

\subsection{Obstacle problems for nonlocal operators}

Obstacle problems for nonlocal operators \eqref{eq:L}-\eqref{eq:Kcomp} arise naturally in the optimal stopping problem for stable-like L\'evy processes, as well as in the study of interacting particle systems (see \cite{CoTa04,CaFi13}, \cite{Ser15, CDM16}, \cite{FeRo24} and references therein).

Given a smooth obstacle $\phi$, the problem is given by 
\begin{equation}\label{obst-pb}
\min\{Lu,\, u-\phi\}=0\quad \textrm{in}\quad \R^n,
\end{equation}
typically with $u\to0$ at infinity.
The first regularity results in this context were obtained for the fractional Laplacian $L=(-\Delta)^s$ in \cite{Sil07}, \cite{CSS08}, where Caffarelli, Salsa, and Silvestre established the optimal $C^{1+s}$ regularity of solutions.

A well known open question after the results in \cite{CSS08} was to establish the optimal regularity of solutions to the obstacle problem \eqref{obst-pb} for general operators of the form \eqref{eq:L}-\eqref{eq:Kcomp}.
The question was explicitly raised for example in \cite{mwiki}:

\vspace{2mm}

\noindent \textbf{Open problem (\cite{mwiki})} \textit{If $u$ is a solution of \eqref{obst-pb},  with a kernel $K$ satisfying the usual ellipticity conditions \eqref{eq:Kcomp}, is the optimal regularity going to be $C^{1+s}$ as in the fractional Laplacian case? Most probably some extra assumption on the kernel will be needed.}

\vspace{2mm}

This problem was partially solved in \cite{CRS17}, \cite{FRS23}, namely, the $C^{1+s}$ regularity was established under the extra assumption that the kernel $K$ is \emph{homogeneous} --- i.e. \eqref{homogeneous}.

In this paper we prove, to our own surprise, that the $C^{1+s}$ regularity holds for \textit{all} kernels \eqref{eq:Kcomp}, with no extra assumption on $K$.
This finally solves the long-standing open problem about the optimal regularity of solutions to the obstacle problem \eqref{obst-pb}.

\begin{theorem}\label{thm-obst-intro}
Let $L$, $K$, $s$, $\lambda$, and $\Lambda$ be as in \eqref{eq:L}-\eqref{eq:Kcomp}.
Let $\phi\in C^3_c(\R^n)$, and $u$ be any solution of \eqref{obst-pb} with $u\to0$ at $\infty$.
Then, we have
\[\|u\|_{C^{1+s}(\R^n)} \leq C\|\phi\|_{C^3(\R^n)}, \]
with $C$ depending only on $n$, $s$, $\lambda$, $\Lambda$.

Moreover, for every free boundary point $x_0\in\{u>\phi\}$ the following dichotomy holds:
\begin{itemize}
\item[(i)] either 
\[0<cr^{1+s} \leq \sup_{B_r(x_0)}(u-\phi) \leq Cr^{1+s} \qquad \textrm{for all}\quad r\in(0,1),\]

\item[(ii)] or 
\[\qquad \qquad \ 0 \leq \sup_{B_r(x_0)}(u-\phi) \leq Cr^{1+s+\gamma} \qquad \textrm{for all}\quad r\in(0,1).\]
\end{itemize}
Furthermore, the set of points satisfying (i) is an open subset of the free boundary, and it is $C^{1,\gamma}$.
The constant $\gamma>0$ depends only on $n$, $s$, $\lambda$, and $\Lambda$.
\end{theorem}

As in \cite{CSS08}, \cite{CRS17}, points of type (i) are called regular points, while those of type (ii) are degenerate points. We refer the reader to \autoref{thm:obstacle-problem} for a finer description of the behavior at regular points.

As we will see, the main missing ingredient for \autoref{thm-obst-intro} to hold was the boundary regularity for {linear} equations (\autoref{thm-Cs-intro}), and in particular the understanding of 1D solutions.
This is one of the key contributions of this paper.

\subsection{1D solutions}

In case of the fractional Laplacian, an explicit computation shows that \eqref{1d-hom} holds. Since operators with homogeneous kernels are basically superpositions of one-dimensional fractional Laplacians on the sphere, \eqref{1d-hom} carries over to this more general class of operators (see \cite[Lemma~2.6.2]{FeRo24}). Once \eqref{1d-hom} is known, the
proof of $C^s$ regularity for \eqref{eq:PDE} is then relatively simple, and requires the construction of some appropriate explicit barriers, which are essentially perturbations of such 1D solution (see \cite{RoSe14}, \cite{RoSe16b} or \cite[Appendix B.2]{FeRo24}).

For general (inhomogeneous) kernels satisfying the ellipticity condition \eqref{eq:Kcomp}, there is no possible explicit computation, not even in dimension $n=1$.
What follows from standard arguments is that for every given $K$ there exists a unique positive solution $b\in C(\R)$ of
\begin{equation}
\label{eq:1D}
\left\{\begin{array}{rcl}
L b &=& 0 ~~ \text{ in }\  \{x>0\},\\
b &=& 0 ~~ \text{ in }\ \{x\leq0\}.
\end{array}\right.
\end{equation}
However, until now, there was no reason to believe that such solution is $C^s$.

Actually, for non-symmetric (homogeneous) kernels, such solution $b$ can be homogeneous of any degree $\alpha\in (0,2s)\cap (2s-1,1)$ and thus the $C^s$ regularity fails (see \cite{DRSV22}).
Moreover, when $L$ is replaced by a fully nonlinear operator $I$ (with kernels satisfying only \eqref{eq:Kcomp}), then the $C^s$ regularity also fails (see \cite[Section 2]{RoSe16a}).

Notice that the problem \eqref{eq:1D} has no natural homogeneity, in the sense that if $b$ is a solution, then so is $b(rx)/r^\beta$ for any $\beta$ (with respect to a rescaled operator, still satisfying \eqref{eq:Kcomp}).

Thus, it is quite surprising that we can actually establish the following result.

\begin{theorem}\label{thm-1D}
Let $L$, $K$, $s$, $\lambda$, and $\Lambda$ be as in \eqref{eq:L}-\eqref{eq:Kcomp}. Let $e \in \mathbb{S}^{n-1}$.
Then, there exists a positive solution $b\in C(\R^n)\cap L^1_{2s}(\R^n)$ of
\begin{equation}
\label{eq:half-space}
\left\{\begin{array}{rcl}
L b &=& 0 ~~ \text{ in }\  \{x \cdot e >0\}\\
b &=& 0 ~~ \text{ in }\ \{x \cdot e \leq0\},
\end{array}\right.
\end{equation}
which is unique up to a multiplicative constant, and satisfies $b(x)=b(x \cdot e)$, $b\in C^s_{loc}(\R^n)\cap H^s_{loc}(\R^n)$, and 
\[0<c(x \cdot e)_+^s \leq b(x \cdot e) \leq C(x \cdot e)_+^s \quad \textrm{in}\quad \R^n.\]
The constants $c$ and $C$ depend only on $s$, $\lambda$, $\Lambda$.
\end{theorem}

For the definition of the space $L^1_{2s}(\R^n)$, we refer to Section \ref{sec:prelim}.

Notice that the function $b$ depends strongly on the direction $e \in \mathbb{S}^{n-1}$, and also on the kernel $K$. Highly oscillating kernels could lead to highly oscillating functions $b$. In fact, it turns out that in general the quotient $b(x \cdot e)/(x \cdot e)_+^s$ is not even continuous as $x \cdot e \searrow 0$ (see \autoref{remark:optimality}). This phenomenon rules out higher order boundary regularity for solutions to \eqref{eq:PDE} with kernels merely satisfying \eqref{eq:Kcomp}.

To prove \autoref{thm-1D}, the first thing we notice is that it suffices to establish it in dimension $n=1$, thanks to \autoref{lemma:1dnd}. Then, a key new idea we introduce in this paper is to construct the 1D function $b$ as a minimizer of a one-phase \emph{free boundary problem}.
The reason for this approach to work is that such free boundary problem has a natural scaling, and therefore solutions are expected to be $C^s$. Moreover, these solutions solve \eqref{eq:1D}, at least locally in a small interval.

More precisely, we consider minimizers of 
\begin{equation}
\label{onephase-intro}
\cI(u):= \iint\limits_{(B_1^c \times B_1^c)^c} \big| u(x)-u(y)\big|^2 K(x-y) \d x \d y + M\big|\{u>0\}\cap B_1\big|
\end{equation}
for some $M>0$, and with prescribed conditions $u=g \ge 0$ outside $B_1$. Then, our proof of \autoref{thm-1D} essentially goes as follows:
\begin{itemize}

\item[(i)] Consider, in dimension $n=1$, a minimizer of $\cI(u)$ in $B_1 = (-1,1)$, with exterior data $u=0$ in $(-\infty,-1]$ and $u=1$ in $[1,\infty)$.
Such a minimizer $u$ will solve $Lu=0$ in $\{u>0\}\cap (-1,1)$.

\item[(ii)] Show that, if we choose $M$ appropriately, $u$ will have a free boundary point $x_0 \in(-1,1)$, with $u=0$ in $[x_0-\eps,x_0]$ and $u>0$ in $(x_0,x_0+\eps)$, and with $\eps>0$ depending only on $s$, $\lambda$, $\Lambda$.

\item[(iii)] Prove optimal $C^s$ regularity and non-degeneracy estimates for minimizers of $\cI(u)$.

\item[(iv)] The solution $b$ of \eqref{eq:1D} will be the limit $u(x_0+rx)/r^s$ as $r\to0$, and by construction will satisfy $b(x)\asymp (x_+)^s$, as desired.

\end{itemize}

Again, we want to emphasize that the proof of \autoref{thm-1D} cannot come from a cheap scaling argument, since it fails both for non-symmetric operators, and for fully nonlinear operators.
Hence, the structure, symmetry, and linearity of the problem must all play a role in the proof somehow, which is precisely what happens in our approach through the one-phase free boundary problem.

It is crucial to notice that, while the equation \eqref{eq:1D} does not have a natural homogeneity, the free boundary problem we consider does.
Indeed, if $u$ is a minimizer of $\cI(u)$, then so is $u(rx)/r^s$ (for a rescaled kernel, still satisfying \eqref{eq:Kcomp}), but not $u(rx)/r^\beta$ for any other $\beta\neq s$.
This is why one expects to have $C^s$ regularity for minimizers of $\cI(u)$.

\subsection{One-phase nonlocal free boundary problems}

Another important contribution of our paper is to establish the optimal regularity and non-degeneracy of minimizers to the one-phase free boundary problem \eqref{onephase-intro} for general nonlocal operators \eqref{eq:L} with kernels \eqref{eq:Kcomp}.
This was only known for the fractional Laplacian (see \cite{CRS10}), where one can rely on the corresponding extension problem to prove such result.

\begin{theorem}\label{thm-onephase-intro}
Let $K$, $s$, $\lambda$, and $\Lambda$ be as in \eqref{eq:L}-\eqref{eq:Kcomp}.
Let $u$ be any local minimizer of \eqref{onephase-intro} in $B_1 \subset \R^n$, for some $M > 0$, with $0\in \partial\{u>0\}$.
Then,
\[\|u\|_{C^s(B_{1/2})} \leq C\]
and 
\[\qquad\qquad \qquad \qquad \quad \sup_{B_r} u \geq c r^s \qquad \textrm{for}\quad r\in(0,1).\]
The constants $C \ge c >0$ depend only on $n$, $s$, $\lambda$, $\Lambda$, and $M$.
\end{theorem}

Notice that the constants $C$ and $c$ are completely independent of $u$.
This is interesting and non-trivial, since we are dealing with a nonlocal problem.

As explained above, in case $n=1$, this result is crucial in our proofs of \autoref{thm-Cs-intro} and \autoref{thm-obst-intro}.
Moreover, the theorem has its own interest, as it extends the results of \cite{CRS10} to a much more general setting, and it gives a new proof even for the fractional Laplacian.

Our proof of \autoref{thm-onephase-intro} is given in Section \ref{sec:one-phase}. It roughly follows the strategy in \cite{Vel23} but several non-trivial adaptations are required due to the nonlocality of our setting. Moreover, in \autoref{thm:density-est} and \autoref{cor:blowups}, we establish density estimates for the free boundary and basic properties of blow-ups, which are of independent interest, but also enter in the proof of \autoref{thm-1D}.
Notice that all proofs would be essentially the same in dimension $n=1$, in the sense that no substantial simplification is expected in that case.

Finally, we mention that in the special case of the fractional Laplacian, also fine regularity results for the free boundary of minimizers to the one-phase problem \eqref{onephase-intro} are available in the literature (see \cite{DeRo12}, \cite{DeSa12}, \cite{DSS14}, \cite{DeSa15}, \cite{DeSa15b}, \cite{EKPSS21}, \cite{FeRo22}). Similar results for general kernels are not known, and we plan to investigate this question in the future.

\subsection{Strategy of the proof of $C^s$ regularity in bounded domains}

Once we have the 1D barrier given by \autoref{thm-1D} (whose proof is based on  \autoref{thm-onephase-intro}), we can proceed with the proof of the $C^s$ regularity in bounded domains (see \autoref{thm-Cs-intro} above).

When $\Omega$ is convex,  one can directly use $b$ as a 1D barrier and deduce the desired $C^s$ estimate.
However, for general (smooth) domains this proof does not work, and one needs either a barrier adapted to the geometry of $\partial\Omega$, or a different argument to get the $C^s$ estimate.

Given the fact that the 1D solutions $b$ are very much non-explicit, and would strongly depend on the direction of the normal vector $\nu$, it does not seem easy at all to construct barriers adapted to the geometry of the domain.
Instead, we deduce the $C^s$ estimate from the following higher order expansion.

\begin{theorem}\label{thm-expansion-intro}
Let $\Omega\subset \R^n$ be any $C^{1,\alpha}$ domain, $0\in\partial\Omega$, and $L$, $K$, $s$, $\lambda$, and $\Lambda$ be as in \eqref{eq:L}-\eqref{eq:Kcomp}.
Let $f\in L^\infty(\Omega)$, and $u$ be any weak solution of
\begin{equation*}
\left\{\begin{array}{rcl}
L u &=& f ~~ \text{ in }\  \Omega \cap B_1,\\
u &=& 0 ~~ \text{ in }\ B_1 \setminus \Omega.
\end{array}\right.
\end{equation*}
Denote the normal vector at $0\in\partial\Omega$ by $e \in \mathbb{S}^{n-1}$. Then, for any $\eps \in (0,\alpha s)$ we have for some $q_0 \in \R$
\begin{equation}\label{expansion-intro}
\big|u(x)- q_0 b(x \cdot e)\big| \leq C |x|^{s+\eps} \left( \|u\|_{L^\infty(\R^n)}  +  \|f\|_{L^\infty(\Omega\cap B_1)}\right),
\end{equation}
where $b$ is the solution of \eqref{eq:half-space} (with respect to $e$, $L$); and $C$ depends only on $\Omega$, $s$, $\lambda$, $\Lambda$, $\eps$, $\alpha$.
\end{theorem}

We prove \eqref{expansion-intro} by a blow-up and contradiction argument, adapting the ideas from \cite{RoSe16b} to this context (see Section \ref{sec:Cs}). A key result in this proof is a Liouville theorem in the half-space (see \autoref{thm:Liouville-half-space}), which is of independent interest.
Once we have the expansion \eqref{expansion-intro}, the $C^s$ estimate in  \autoref{thm-Cs-intro} follows easily.

The Hopf-type result $u\geq cd^s$ in \autoref{thm-Cs-intro} is basically equivalent to $q >0$ in \eqref{expansion-intro}, and it does not follow so easily. 
To prove it, we need an extra blow-up argument, combined with a ``soft'' Hopf-type lemma of the type $u\geq c_1 d^{2s-\gamma}$ for some $\gamma>0$.

Finally, after we have \autoref{thm-Cs-intro}, we can proceed to prove the optimal $C^{1+s}$ regularity for the obstacle problem (see \autoref{thm-obst-intro}).

It is important to emphasize that, in order to prove \autoref{thm-obst-intro}, we truly need \emph{all} the results of this paper.
Indeed, \autoref{thm-onephase-intro} is crucial in the proof of \autoref{thm-1D}, which in turn is crucial in the proof of \autoref{thm-expansion-intro}, which in turn is crucial in the proof of \autoref{thm-Cs-intro}, which in turn will allow us to prove \autoref{thm-obst-intro}.

\subsection{Acknowledgements}

The authors were supported by the European Research Council under the Grant Agreements No. 801867 (EllipticPDE) and No. 101123223 (SSNSD), and by AEI project PID2021-125021NA-I00 (Spain).
Moreover, X.R was supported by the grant RED2022-134784-T funded by AEI/10.13039/501100011033, by AGAUR Grant 2021 SGR 00087 (Catalunya), and by the Spanish State Research Agency through the Mar\'ia de Maeztu Program for Centers and Units of Excellence in R{\&}D (CEX2020-001084-M).

\subsection{Organization of the paper}

The paper is organized as follows.
In Section \ref{sec:prelim} we give some preliminary results, and in Section \ref{sec:suboptimal-reg} we prove a $C^{2s-1+\eps}$ boundary estimate for $s>\frac12$, which we need in the study of the one-phase nonlocal free boundary problem.
In Section \ref{sec:one-phase} we develop the basic theory for such a free boundary problem and prove \autoref{thm-onephase-intro}.
In Section \ref{sec:half-space} we establish \autoref{thm-1D}, and in Section \ref{sec:Cs} we prove \autoref{thm-Cs-intro} and \autoref{thm-expansion-intro}.
Finally, in Section \ref{sec:obstacle} we apply these results to the obstacle problem, establishing \autoref{thm-obst-intro}.

\section{Preliminaries}
\label{sec:prelim}

In this section we collect several definitions and auxiliary lemmas that will become important throughout the course of this article. Most importantly, we will give a rigorous definition of (local) minimizers to the nonlocal one-phase problem \eqref{onephase-intro}.

From now on, we denote by $\mathcal{L}^n_s(\lambda,\Lambda)$ the class of operators \eqref{eq:L} with kernels $K$ satisfying \eqref{eq:Kcomp} for some $0 < \lambda \le \Lambda < \infty$ and $s \in (0,1)$, acting on functions $u : \R^n \to \R$.

\subsection{Function spaces and solution concepts}

Given an open, bounded domain $\Omega \subset \R^n$, let us introduce the following function spaces, which are naturally associated with the energy $\cI$ from \eqref{onephase-intro}:
\begin{align*}
V^s(\Omega | \Omega') &:= \left \{ u \hspace{-0.1cm}\mid_{\Omega} \hspace{0.1cm} \in L^2(\Omega) : [u]^2_{V^s(\Omega | \Omega')} := \int_{\Omega} \int_{\Omega'} \frac{(u(x) - u(y))^2}{|x-y|^{n+2s}} \d y \d x < \infty \right \}, ~~ \Omega \Subset \Omega'\\
H^s_{\Omega}(\R^n) &:=  \left\{ u \in V^s(\R^n | \R^n) : u \equiv 0 \text{ in } \R^n \setminus \Omega \right\},\\
H^s(\Omega) &:=  \left\{ u \in L^2(\Omega) : [u]^2_{H^s(\Omega)}  := \int_{\Omega} \int_{\Omega} \frac{(u(x) - u(y))^2}{|x-y|^{n+2s}} \d y \d x < \infty \right\},\\
L^p_{2s}(\R^n) &:= \left\{ u : \R^n \to \R : \Vert u \Vert_{L^p_{2s}(\R^n)}^p := \int_{\R^n} |u(y)|^p (1 + |y|)^{-n-2s} \d y < \infty \right\}, ~~ p \ge 1.
\end{align*}
These spaces are equipped with the following norms:
\begin{align*}
\Vert u \Vert_{V^s(\Omega | \Omega')} := \Vert u \Vert_{L^2(\Omega)} + [u]_{V^s(\Omega | \Omega')}, \qquad \Vert u \Vert_{H^s(\Omega)} := \Vert u \Vert_{L^2(\Omega)} + [u]_{H^s(\Omega)}.
\end{align*}

The following quantity captures the long-range interactions caused by the nonlocality of the energy:
\begin{align*}
\tail(u;R,x_0) := R^{2s} \int_{\R^n \setminus B_R(x_0)} |u(y)||y - x_0|^{-n-2s} \d y, ~~ x_0 \in \R^n, ~~ R > 0.
\end{align*}
When $x_0 = 0$, we will often write $\tail(u;R,0) = \tail(u;R)$.\\
Given a kernel $K : \R^n \to [0,\infty]$ satisfying \eqref{eq:Kcomp} and a set $\cD \subset \R^n \times \R^n$, we introduce the notation
\begin{align*}
\cE^K_{\cD}(u,v) = \iint_{\cD} (u(x) - u(y))(v(x) - v(y)) K(x-y) \d y \d x.
\end{align*}
We will write $\cE^K_{\cD} := \cE_{\cD}$ when no confusion about the associated kernel $K$ can occur. Moreover, if $\cD = D \times D$ for some $D \subset \R^n$, we write $\cE_{D} := \cE_{\cD}$, and if $\cD = \R^n \times \R^n$, we write $\cE := \cE_{\cD}$.\\
Moreover, given $K$ satisfying \eqref{eq:Kcomp}, $\Omega \subset \R^n$, and $M > 0$, we denote
\begin{align}
\label{eq:OP}
\cI^{K,M}_{\Omega}(u) := \cE^K_{(\Omega^c \times \Omega^c)^c}(u,u) + M |\{ u > 0 \} \cap \Omega|,
\end{align}
whenever this expression is finite.
We will write $\cI^{K,M}_{\Omega} =: \cI_{\Omega} =: \cI$, when no confusion about the associated kernel $K$ and $M$ and $\Omega$ can occur.

Note that by \eqref{eq:Kcomp}, we have that $\cE_{(\Omega^c \times \Omega^c)^c}(u,u) < \infty$, whenever $u \in V^s(\Omega | \R^n)$. Moreover, it holds $V^s(\Omega | \R^n) \subset L^1_{2s}(\R^n)$, and therefore $\tail(u;R,x_0) < \infty$ whenever $B_R(x_0) \subset \Omega$ and $u \in V^s(B_R(x_0) | \R^n)$. However in some cases we need to consider functions that grow like $t \mapsto t^s$ at infinity. Such functions arise as blow-up limits of minimizers of $\cI_{\Omega}$. Since such functions do not belong to $V^s(\Omega | \R^n)$, we need to define what it means for a function, merely belonging to $V^s(\Omega | \Omega') \cap L^1_{2s}(\R^n)$ (with $\Omega \Subset \Omega'$) to minimize $\cI_{\Omega}$. A similar notion of minimization for nonlocal functionals has been introduced in \cite{CoLo21} in the context of (nonlocal) minimal surfaces.

\begin{definition}[minimizers]
\label{def:minimizer}
Let $K$ satisfy \eqref{eq:Kcomp}. Let $\Omega \Subset \Omega' \subset \R^n$ be an open, bounded domain, and $M > 0$. We say that $u \in V^s(\Omega | \Omega') \cap L^1_{2s}(\R^n)$ with $u \ge 0$ in $\R^n$ is a (local) minimizer of $\cI_{\Omega}$ (in $\Omega$) if for any $v \in V^s(\Omega | \Omega') \cap L^1_{2s}(\R^n)$ with $u = v$ in $\R^n \setminus \Omega$, it holds
\begin{align*}
\iint\limits_{(\Omega^c \times \Omega^c)^c} \Big[(u(x) - u(y))^2 - (v(x) - v(y))^2 \Big] K(x-y) \d y \d x + M \Big[ | \{u > 0 \} \cap \Omega | - | \{v > 0 \} \cap \Omega | \Big] \le 0.
\end{align*}
Moreover, when $g \in V^s(\Omega | \Omega') \cap L^1_{2s}(\R^n)$ with $g \ge 0$ is given, then we say that $u$ is a minimizer of $\cI_{\Omega}$ with exterior data $g$, if $u$ is a local minimizer of $\cI_{\Omega}$ and $u = g$ in $\R^n \setminus \Omega$.
\end{definition}

\begin{remark}
Let $u \in V^s(\Omega | \R^n)$. Then $u$ minimizes $\cI_{\Omega}$ in the sense of \autoref{def:minimizer} if and only if 
\begin{align}
\label{eq:simplified-minimizer}
\cI_{\Omega}(u) \le \cI_{\Omega}(v) ~~ \forall v \in V^s(\Omega | \R^d) ~~ \text{ with } u=v=g ~~ \text{ in } \R^n \setminus \Omega.
\end{align}
Moreover, if $g \in V^s(\Omega | \R^n)$ with $g \ge 0$ is given, then the existence of a minimizer $u$ of $\cI_{\Omega}$ in $\Omega$ with exterior data $g$ follows by standard arguments using \eqref{eq:simplified-minimizer} (see \autoref{lemma:existence}).
\end{remark}

In the following, in order to simplify notation, we will occasionally write \eqref{eq:simplified-minimizer} instead of using the expression in \autoref{def:minimizer}, even if $u \in H^s(\Omega) \cap L^1_{2s}(\R^n)$. However, whenever we write \eqref{eq:simplified-minimizer}, everything can be made rigorous in a straightforward way, by writing the energies under the same integral.

Clearly, given a jumping kernel $K$, the energy $\cE^K$ gives rise to an integro-differential operator $L$ given by \eqref{eq:L} via the relation
\begin{align*}
\cE_{(\Omega^c \times \Omega^c)^c}(u,\phi) = (L u , \phi) ~~ \forall \phi \in H^s_{\Omega}(\R^n).
\end{align*}

Throughout this article, we will work with the following notion of weak solutions:

\begin{definition}[weak solutions]
\label{def:weak-solution}
Let $\Omega \Subset \Omega' \subset \R^n$ be an open, bounded domain. Let $f \in L^{\infty}(\Omega)$. We say that $u \in V^s(\Omega | \Omega') \cap L^1_{2s}(\R^n)$ is a weak subsolution to $Lu \le f$ in $\Omega$ if 
\begin{align}
\label{eq:weak-sol}
\cE_{(\Omega^c \times \Omega^c)^c}(u,\phi) \le \int_{\Omega} f \phi \d x ~~ \forall \phi \in H^s_{\Omega}(\R^n) ~~ \text{ s.t. } \phi \ge 0.
\end{align}
We say that $u$ is a weak supersolution to $Lu \ge f$ in $\Omega$ if \eqref{eq:weak-sol} holds true with $-u$ instead of $u$. Moreover, $u$ is a weak solution to $Lu = f$ in $\Omega$, if it is a weak subsolution and a weak supersolution.
\end{definition}

\begin{remark}
Note that $u \in V^s(\Omega| \Omega') \cap L^1_{2s}(\R^n)$ guarantees that the energy in \eqref{eq:weak-sol} is finite.
\end{remark}

Since the nonlocal one-phase problem is variational, minimizers of $\cI$ will naturally be weak (sub)solutions. Therefore, energy methods are needed in order to verify their regularity properties.

\subsection{Harmonic replacement}
\label{subsec:harmonic-replacement}

An important tool in the analysis of the optimal regularity of minimizers of $\cI$ is played by the so called $L$-harmonic replacement.

\begin{definition}[harmonic replacement]
Let $L \in \mathcal{L}_s^n(\lambda,\Lambda)$. Let $B \Subset B' \subset \R^n$ be an open, bounded domain and $u \in V^s(B | B') \cap L^1_{2s}(\R^n)$. Then, we call $v \in V^s(B | B') \cap L^1_{2s}(\R^n)$ the $L$-harmonic replacement of $u$ in $B$ if $v$ is a weak solution to
\begin{align*}
\begin{cases}
L v &= 0 ~~ \text{ in } B,\\
v &= u ~~ \text{ in } \R^n \setminus B.
\end{cases}
\end{align*} 
\end{definition} 

\begin{remark}
Note that if $u \in V^s(B | B') \cap L^1_{2s}(\R^n)$, then the $L$-harmonic replacement $v$ of $u$ in $B$ always exists and is uniquely defined (see \cite[Theorem 4.3]{KiLe23}). Moreover, it holds $u-v \in H^s_{B}(\R^n)$.
\end{remark}

We collect two lemmas on harmonic replacements, that will be useful for us in the course of this article.

\begin{lemma}
\label{lemma:harmonic-replacement}
Let $L \in \mathcal{L}_s^n(\lambda,\Lambda)$. Let $\Omega \Subset \Omega' \subset \R^n$ and $R > 0$ and $x_0 \in \R^n$ be such that $B_R(x_0) \subset \Omega$. Let $u \in V^s(\Omega| \Omega')$ and let $v$ be the $L$-harmonic replacement of $u$ in $B_R(x_0)$. Then,
\begin{align*}
\cE_{(B_R(x_0)^c \times B_R(x_0)^c)^c}(u,u) - \cE_{(B_R(x_0)^c \times B_R(x_0)^c)^c}(v,v) &= \cE_{(B_R(x_0)^c \times B_R(x_0)^c)^c}(u-v,u-v),\\
\cE_{(\Omega^c \times \Omega^c)^c}(u,u) - \cE_{(\Omega^c \times \Omega^c)^c}(v,v) &= \cE_{(\Omega^c \times \Omega^c)^c}(u-v,u-v).
\end{align*}
\end{lemma}

Clearly, if $u \in V^s(\Omega | \Omega') \cap L^1_{2s}(\R^n)$, but not in $V^s(\Omega | \R^n)$, the expressions on the left hand sides above are not necessarily finite. In that case, one interprets the difference of the energies in the same way as in \autoref{def:minimizer} (see also the proof).

\begin{proof}
Note that 
\begin{align*}
(u_1 - u_2)^2 - (v_1 - v_2)^2 &= 2(v_1 - v_2)((u_1 - u_2) - (v_1 - v_2)) + ((u_1 - u_2) - (v_1 - v_2))^2\\
&= 2(v_1 - v_2)((u_1 - v_1) - (u_2 - v_2)) + ((u_1 - v_1) - (u_2 - v_2))^2.
\end{align*}
Thus, we have
\begin{align*}
\iint_{(B_R(x_0)^c \times B_R(x_0)^c)^c} & \big[ (u(x) - u(y))^2 - (v(x) - v(y))^2 \big] K(x-y) \d y \d x \\
&= 2\cE_{(B_R(x_0)^c \times B_R(x_0)^c)^c}(v,u-v) + \cE_{(B_R(x_0)^c \times B_R(x_0)^c)^c}(u-v,u-v)\\
&=\cE_{(B_R(x_0)^c \times B_R(x_0)^c)^c}(u-v,u-v),
\end{align*}
where we used that by the properties of $v$, we have that $u-v \in H^s_{B_R(x_0)}(\R^n)$ is a valid test function for $Lv = 0$ in $B_R(x_0)$.\\
The proof of the second property follows by the same argument, and using that 
\begin{align*}
0 = \cE_{(B_R(x_0)^c \times B_R(x_0)^c)^c}(v,u-v) &= \cE_{(\Omega^c \times \Omega^c)^c}(v,u-v) = \cE(v,u-v),\\
\cE_{(B_R(x_0)^c \times B_R(x_0)^c)^c}(u-v,u-v) &= \cE_{(\Omega^c \times \Omega^c)^c}(u-v,u-v) = \cE(u-v,u-v),
\end{align*}
since $(B_R(x_0)^c \times B_R(x_0)^c)^c \subset (\Omega^c \times \Omega^c)^c$ and $u = v$ in $(\R^n \setminus \Omega) \subset (\R^n \setminus B_R(x_0))$.
\end{proof}

In particular, we have

\begin{lemma}
\label{lemma:energy-comp-est}
Let $L \in \mathcal{L}_s^n(\lambda,\Lambda)$. Let $\Omega \subset \R^n$ and $M > 0$. Let $u$ be a minimizer of $\cI_{\Omega}$. Let $B_R(x_0) \subset \Omega$ and let $v$ be the $L$-harmonic replacement of $v$ in $B_R(x_0)$. Then,
\begin{align*}
\cE_{(B_R(x_0)^c \times B_R(x_0)^c)^c}(u-v,u-v) \le M |\{ u = 0\} \cap B_R(x_0)| \le c M R^n
\end{align*}
for some $c > 0$, depending only on $n$.
\end{lemma}

\begin{proof}
We compute, using that $u$ minimizes $\cI_{\Omega}$ and that $v$ is a competitor for $u$, i.e., it holds $\cI_{\Omega}(u) - \cI_{\Omega}(v) \le 0$, and \autoref{lemma:harmonic-replacement}
\begin{align*}
\cE_{(B_R(x_0)^c \times B_R(x_0)^c)^c}(u-v,u-v) &= \cE_{(\Omega^c \times\Omega^c)^c}(u-v,u-v)\\
&= \cE_{(\Omega^c \times \Omega^c)^c}(u,u) - \cE_{(\Omega^c \times \Omega^c)^c}(v,v)\\
&\le - M | \{ u > 0 \} \cap \Omega| + M |\{ v > 0 \} \cap \Omega| \\
&= - M | \{ u > 0 \} \cap B_R(x_0)| + M |\{ v > 0 \} \cap B_R(x_0)|\\
&\le M \left(|B_R(x_0)| - | \{ u > 0 \} \cap B_R(x_0)| \right)\\
&= M | \{ u = 0 \} \cap B_R(x_0)|,
\end{align*}
as desired.
\end{proof}

\section{Homogeneous barriers and $C^{2s-1+\eps}$ estimates}
\label{sec:suboptimal-reg}

In this section, we establish that solutions to the nonlocal Dirichlet problem in a $C^{1,\alpha}$ domain are in $C^{2s-1+\eps}$ up to the boundary for some $\eps > 0$, when $s > 1/2$. The novelty of this result comes from the fact that we consider general operators of the form \eqref{eq:L} with \eqref{eq:Kcomp} but without assuming the kernels to be homogeneous. So far, it was only known that solutions are $C^{\eps}(\overline{\Omega})$
This result will allow us to establish non-degeneracy for minimizers of $\cI$ without assuming homogeneity of the kernels (see \autoref{cor:annulus-regularity}).

\begin{proposition}
\label{prop:beta-regularity}
Let $L \in \mathcal{L}_s^n(\lambda,\Lambda)$. Let $\Omega \subset \R^n$ be a bounded domain with $\partial \Omega \in C^{1,\alpha}$ for some $\alpha \in (0,s)$. Let $f \in L^{\infty}(B_1 \cap \Omega)$ and let $u \in L^{\infty}(\R^n)$ be a weak solution to
\begin{align*}
\begin{cases}
L u &= f ~~ \text{ in } B_1 \cap \Omega,\\
u &= 0 ~~ \text{ in } B_1 \setminus \Omega.
\end{cases}
\end{align*}
Then, there exists $\beta \in (0, \min \{ 1 , 2s \})$ with $\beta > 2s - 1$, depending only on $n,s,\lambda,\Lambda$, such that $u \in C^{\beta}_{loc}(B_1)$ with
\begin{align*}
\Vert u \Vert_{C^{\beta}(B_{1/2})} \le C \left( \Vert u \Vert_{L^{\infty}(\R^n)} + \Vert f \Vert_{L^{\infty}(B_1 \cap \Omega)} \right)
\end{align*}
where $C > 0$ depends only on $n,s,\lambda,\Lambda,\alpha,\Omega$.
\end{proposition}

\begin{remark}
The proof of \autoref{prop:beta-regularity} reveals that it is also possible to take $f$ such that $d^{\beta - 2s} f \in L^{\infty}(B_1 \cap \Omega)$, and $u \in C(B_1) \cap L^{\infty}_{2s-\eps}(\R^n)$ for some $\eps > 0$ being a distributional solution.
\end{remark}

As a corollary, we get the following result, which will be used in the proof of the non-degeneracy:

\begin{corollary}
\label{cor:annulus-regularity}
Let $L \in \mathcal{L}_s^n(\lambda,\Lambda)$. Let $R > 0$, $u \in L^1_{2s}(\R^n)$. Let $v$ be a weak solution to 
\begin{align*}
\begin{cases}
L v &= 0 ~~ \text{ in } B_{2R} \setminus B_R,\\
v &= 0 ~~ \text{ in } B_R,\\
v &= u ~~ \text{ in } \R^n \setminus B_{2R}.
\end{cases}
\end{align*}
Then, there exist $\beta \in (0 , \min\{ 2s , 1 \})$ with $\beta > 2s -1$, and $C > 0$, depending only on $n,s,\lambda,\Lambda$, such that
\begin{align*}
\left\Vert \frac{v}{\dist(\cdot,\partial B_R)^{\beta}} \right\Vert_{L^{\infty}(B_{2R} \setminus B_R)} \le C R^{-\beta} \left(\Vert v \Vert_{L^{\infty}(B_{2R})} + \tail(u;2R) \right).
\end{align*}
\end{corollary}

Before we prove \autoref{prop:beta-regularity}, we need the following refinement of \cite[Lemma B.1.1]{FeRo24}.

\begin{lemma}
\label{lemma:one-dimensional-barrier}
Let $L \in \mathcal{L}_s^n(\lambda,\Lambda)$. Let $e \in \mathbb{S}^{n-1}$. Then, there exists $\beta \in (0 , \min \{2s , 1 \})$ with $\beta > 2s - 1$, and $c > 0$, depending only on $n,s,\lambda,\Lambda$, such that it holds in the viscosity sense
\begin{align*}
L (x \cdot e)_+^{\beta} \ge c (x \cdot e)_+^{\beta - 2s} ~~ \text{ in } \{ x \cdot e > 0 \}.
\end{align*}
\end{lemma}

\begin{proof}
Let us assume without loss of generality that $e = e_n$.
We define the extremal operator
\begin{align*}
\cM^{-} w = \inf_{L \in \mathcal{L}^n_s(\lambda,\Lambda)} Lw.
\end{align*}
By following the arguments in the proof of \cite[Lemma B.1.1]{FeRo24} we deduce that for any $\beta \in (0,2s)$ there is $c_{\beta} \in \R$ such that
\begin{align*}
\cM^{-} (x_n)_+^{\beta} = c_{\beta} (x_n)_+^{\beta - 2s} ~~ \text{ in } \{ x_n > 0 \}.
\end{align*}
Moreover, the proof of \cite[Lemma B.1.1]{FeRo24} reveals that $c_{\beta} \to -\infty$ as $\beta \to 2s$, $c_1 < 0$, and $c_{\beta} > 0$ as $\beta \to 0$. Thus, by continuity of $\beta \mapsto c_{\beta}$, there exists $\beta_0 \in (0,\min \{ 1 , 2s\})$ such that $c_{\beta_0} = 0$. Moreover, by a standard sliding argument, we can deduce that $\beta_0$ is unique with this property. Therefore,
\begin{align*}
\begin{cases}
L (x_n)_+^{\beta} \ge \cM^{-} (x_n)_+^{\beta} &= c_{\beta} (x_n)_+^{\beta-2s} \ge 0 ~~ \text{ in } \{ x_n > 0 \} ~~ \forall \beta \in (0, \beta_0],\\
\cM^{-} (x_n)_+^{\beta} &= c_{\beta} (x_n)_+^{\beta-2s} < 0 ~~ \text{ in } \{ x_n > 0 \} ~~ \forall \beta \in (\beta_0, 2s).
\end{cases}
\end{align*}
It remains to show that $\beta_0 > 2s-1$. By contradiction, we assume that $\beta_0 \le 2s-1$. We take $\eps \in (0,\beta_0)$. Then, we have $\beta_0 + 1 - \eps \in (\beta_0,2s)$, and $\beta_0 - \eps \in (0,\beta_0)$, and therefore for any $L \in \mathcal{L}_s^n(\lambda,\Lambda)$
\begin{align*}
L (x_n)_+^{\beta_0 + 1 - \eps} = (\beta_0 + 1 - \eps) \int_0^{x} L (x_n)_+^{\beta_0 - \eps} \d x \ge 0 ~~ \text{ in } \{ x_n > 0 \},
\end{align*} 
and thus, by recalling the definition of $\cM^-$
\begin{align*}
\cM^{-} (x_n)_+^{\beta_0 + 1 -\eps} \ge 0 ~~ \text{ in } \{ x_n > 0 \},
\end{align*}
a contradiction. This implies $\beta_0 > 2s-1$, as desired.
\end{proof}

Next, we construct a barrier supersolution in $C^{1,\alpha}$ domains for non-stable operators.

\begin{lemma}
\label{lemma:beta-supersolution}
Let $L \in \mathcal{L}_s^n(\lambda,\Lambda)$. Let $\Omega \subset \R^n$ be a bounded domain with $\partial \Omega \in C^{1,\alpha}$ for some $\alpha \in (0,s)$. Then, there exist $c, \delta > 0$, depending only on $n,s,\lambda,\Lambda,\alpha,\Omega$, and $\beta \in (0, \min \{ 1 , 2s \})$ with $\beta > 2s - 1$ such that
\begin{align*}
L d_{\Omega}^{\beta} \ge c d_{\Omega}^{\beta - 2s} ~~ \text{ in } \{ 0 < d_{\Omega} < \delta \}.
\end{align*}
Here, $d := d_{\Omega}$ denotes the regularized distance to $\partial \Omega$ (see \cite[Lemma B.0.1]{FeRo24}).
\end{lemma}

\begin{proof}
The proof follows along the lines of \cite[Proposition B.2.1]{FeRo24}. Let us take $x_0 \in \Omega$ and denote $\rho := d_{\Omega}(x_0)$ and assume that $\rho < \delta$, where we will choose $\delta > 0$ small enough, depending only on $n,s,\lambda,\Lambda,\alpha,\Omega$, later. We denote
\begin{align*}
l(x) = (d(x_0) + \nabla d(x_0) \cdot (x-x_0))_+ = a( e \cdot (x - x_0 + b))_+,
\end{align*}
where we write $a = |\nabla d(x_0)|$, $e = \nabla d(x_0)/a$, and $b = d(x_0) \nabla d(x_0)$. Then, by \autoref{lemma:one-dimensional-barrier} (and \cite[Lemma B.0.1]{FeRo24}), we deduce
\begin{align*}
L(l^{\beta})(x_0) \ge c_1 \rho^{\beta - 2s} > 0
\end{align*}
for some $c_1 > 0$, once $\rho < \delta$ for some $\delta > 0$ small enough, depending only on $\Omega$. Then, since $\beta < 1$, and by \cite[Lemma B.2.2]{FeRo24}
\begin{align*}
|d^{\beta} - l^{\beta}|(x_0 + y) \le |d - l|^{\beta}(x_0 + y) \le C |y|^{(1+\alpha)\beta} ~~ \forall y \in \R^n.
\end{align*}
Moreover, following the arguments in \cite[Proposition B.2.1]{FeRo24}, we deduce
\begin{align*}
|d^{\beta} - l^{\beta}|(x_0 + y) \le c|d-l|(x_0 + y) \rho^{\beta - 1} \le c \Vert D^2 d \Vert_{B_{\rho/2}(x_0)} |y|^2 \rho^{\beta - 1}\le c |y|^2 \rho^{\alpha + \beta - 2} ~~ \forall y \in B_{\rho/2}.
\end{align*}
Finally, we clearly have
\begin{align*}
|d^{\beta} - l^{\beta}|(x_0 + y)  \le c |y|^{\beta} ~~ \forall y \in \R^n \setminus B_1.
\end{align*}
Altogether, we have (using that $d(x_0) = l(x_0)$):
\begin{align*}
L(d^{\beta})(x_0) &\ge c_1 \rho^{\beta - 2s} + L(d^{\beta} - l^{\beta})(x_0) \\
&\ge c_1 \rho^{\beta - 2s} - \int_{B_{\rho/2}} (d^{\beta} - l^{\beta})(x_0 + y)  K(y)\d y \\
&\quad - \int_{B_1 \setminus B_{\rho/2}} (d^{\beta} - l^{\beta})(x_0 + y)  K(y)\d y - \int_{\R^n \setminus B_1} (d^{\beta} - l^{\beta})(x_0 + y)  K(y)\d y \\
&\ge c_1 \rho^{\beta - 2s} - c \rho^{\alpha + \beta - 2} \int_{B_{\rho/2}} \hspace{-0.2cm} |y|^{-n-2s+2} \d y - c \int_{B_1 \setminus B_{\rho/2}} \hspace{-0.2cm} |y|^{-n-2s+(1+\alpha)\beta} \d y - c \int_{\R^n \setminus B_1} \hspace{-0.2cm} |y|^{-n-2s+\beta} \\
&\ge c_1 \rho^{\beta - 2s} - c \rho^{\alpha + \beta - 2s} - c \rho^{\eps_0 + \beta - 2s} - c,
\end{align*}
where we estimated inside the second integral $|y|^{-n-2s + (1+\alpha)\beta} \le |y|^{-n-2s + \beta + \eps_0}$, where we fixed any $\eps_0 \in (0, \min \{ \alpha\beta , 2s-\beta \})$. Altogether, by taking $\delta > 0$ small enough, we have shown
\begin{align*}
L(d^{\beta})(x_0) \ge \frac{c_1}{2} \rho^{\beta - 2s},
\end{align*}
which concludes the proof.
\end{proof}

Now, we are in position to prove \autoref{prop:beta-regularity}.

\begin{proof}[Proof of \autoref{prop:beta-regularity}]
In case $s \le 1/2$, the result was already proved in \cite[Proposition 2.6.9]{FeRo24}. If $s > 1/2$, the proof follows along the lines of the proof of \cite[Proposition 2.6.4]{FeRo24}, replacing the barrier in there by $d_{\Omega}^{\beta}$ from \autoref{lemma:beta-supersolution}, and the exponent $s$ by $\beta$.
\end{proof}

We end this section by explaining how to prove \autoref{cor:annulus-regularity}.

\begin{proof}[Proof of \autoref{cor:annulus-regularity}]
Let us define $w := v \1_{B_{2R}}$. Then, $w$ solves
\begin{align*}
\begin{cases}
L w &= -L(v \1_{\R^n \setminus B_{2R}}) ~~ \text{ in } B_{3R/2} \setminus B_R\\
w &= 0 ~~ \qquad\qquad\qquad \text{ in } B_R.
\end{cases}
\end{align*}
Since for $x \in B_{3R/2}$ it holds
\begin{align*}
|L(v \1_{\R^n \setminus B_{2R}})| \le c R^{-2s} \tail(u;2R),
\end{align*}
an application of a rescaled and translated version of \autoref{prop:beta-regularity} with $\Omega = B_{3R/2} \setminus B_R$ and $B_1$ replaced by $B_{R/2}(y_0)$ for any $y_0 \in \partial B_R$ yields
\begin{align*}
\left\Vert \frac{v}{\dist(\cdot,\partial B_R)^{\beta}} \right\Vert_{L^{\infty}(B_{3R/2} \setminus B_R)} \le \Vert v \Vert_{C^{\beta}(B_{3R/2} \setminus B_R)} \le C R^{-\beta} \left(\Vert v \Vert_{L^{\infty}(B_{2R})} + \tail(u;2R) \right).
\end{align*}
The estimate for $x_0 \in B_{2R} \setminus B_{3R/2}$ is trivial, since there we have $\dist(x_0, \partial B_R) \ge R/2$.
\end{proof}

\section{One-phase nonlocal free boundary problem}
\label{sec:one-phase}

In this section we establish our main result on the one-phase nonlocal free boundary problem (see \autoref{thm-onephase-intro}). In fact, \autoref{thm-onephase-intro} follows immediately by combination of \autoref{lemma:contradiction} and \autoref{lemma:weak-nondeg}. This section is split into five subsections. In Subsection \ref{subsec:aux}, we establish the existence of minimizers and also several basic properties. Subsections \ref{sec:opt-reg} and \autoref{thm:nondeg} are devoted to the proofs of the optimal $C^s$ regularity and the non-degeneracy, the main results being \autoref{thm:or} and \autoref{thm:nondeg}, which slightly generalize \autoref{thm-onephase-intro}. In Subsections \ref{sec:density} and \ref{sec:blow-ups}, we show density estimates for the free boundary (see \autoref{thm:density-est}) and establish basic properties of blow-ups as an application of all the previous results (see \autoref{cor:blowups}).

For the rest of this section, let us fix an open, bounded domain $\Omega \Subset \Omega' \subset \R^n$, a kernel $K$ satisfying \eqref{eq:Kcomp} such that $L \in \mathcal{L}_s^n(\lambda,\Lambda)$, and $M > 0$. We will consider minimizers of $\cI_{\Omega} = \cI_{\Omega}^{K,M}$ in $\Omega$.

\subsection{Existence and basic properties of minimizers}
\label{subsec:aux}

The main goal of this subsection is to establish the following lemma. Moreover, we will prove the existence of minimizers to $\cI$ (see \autoref{lemma:existence}) and establish a useful estimate for the energy and the tail of minimizers (see \autoref{lemma:energy-bound}).

\begin{lemma}
\label{lemma:aux}
Let $u$ be a minimizer of $\cI_{\Omega}$. Then, the following properties hold true:
\begin{itemize}
\item[(i)] $L u \le 0$ in $\Omega$ in the weak sense.
\item[(ii)] $u \ge 0$ in $\Omega$.
\item[(iii)] $u \in L^{\infty}_{loc}(\Omega)$.
\item[(iv)]$L u = 0$ in $\Omega \cap \{ u > 0 \}$ in the weak sense.
\end{itemize}
\end{lemma}

The proof of (i), (ii), (iii) follows by relatively standard arguments. The proof of (iv) is slightly more involved because it requires continuity of $u$ (in order to deduce that $\{ u > 0 \}$ is open), which we establish using a Campanato-type strategy towards the end of this section (see \autoref{thm:aor}).

\begin{proof}[Proof of (i), (ii), (iii)]
First, we prove (i). Let $v \in V^s(\Omega | \Omega') \cap L^1_{2s}(\R^n)$ be such that $v = u$ in $\R^n \setminus \Omega$ and $v \le u$. Then $v_+ \le u_+$ and therefore
\begin{align*}
(\{ v > 0 \} \cap \Omega) = (\{ v_+ > 0 \} \cap \Omega) \subset  (\{ u_+ > 0 \} \cap \Omega) = (\{ u > 0 \} \cap \Omega).
\end{align*}
Since $u$ is a minimizer of $\cI_{\Omega}$, we deduce that
\begin{align*}
\iint_{(\Omega^c \times \Omega^c)^c} \big[ (u(x)-u(y))^2 - (v(x)-v(y))^2 \big] K(x-y) \d y \d x \le 0.
\end{align*}
Let us now take $\phi \in H^s_{\Omega}(\R^n)$ with $\phi \ge 0$ and $t > 0$, and define $v := u -t \phi$. Then, the above estimate implies that 
\begin{align*}
\iint_{(\Omega^c \times \Omega^c)^c} (u(x)-u(y))(\phi(x) - \phi(y)) K(x-y) \d y \d x &\le \frac{t}{2} \iint_{(\Omega^c \times \Omega^c)^c} (\phi(x)-\phi(y))^2 K(x-y) \d y \d x\\
&\to 0, ~~ \text{ as } t \searrow 0.
\end{align*}
This yields $L u \le 0$ in $\Omega$ in the weak sense, as desired.\\
To see (ii), we observe that
\begin{align*}
\cE_{(\Omega^c \times \Omega^c)^c}(u,u) = \cE_{(\Omega^c \times \Omega^c)^c}(u_+,u_+) + \cE_{(\Omega^c \times \Omega^c)^c}(u_-,u_-) - 2 \cE_{(\Omega^c \times \Omega^c)^c}(u_+,u_-) \ge \cE_{(\Omega^c \times \Omega^c)^c}(u_+,u_+),
\end{align*}
which implies that $\cI_{\Omega}(u_+) \le \cI_{\Omega}(u)$, and since $u \ge 0$ in $\R^n \setminus \Omega$, this means that $u_+$ is a competitor for $u$ with less energy, i.e., $u = u_+$.\\
To see (iii), we observe that as a consequence of (i), (ii), $u$ satisfies a local boundedness estimate, namely we have for any $B_R(x_0) \subset \Omega$ (see \cite[Theorem 1.1]{DKP16}, \cite[Theorem 6.2]{Coz17}):
\begin{align*}
\sup_{B_{R/2}(x_0)} |u| = \sup_{B_{R/2}(x_0)} u_+ \le c \left(\dashint_{B_R(x_0)} u_+^2 \d x \right)^{1/2} + c \tail(u_+;R,x_0) < \infty.
\end{align*}
The right hand side is finite since $u \in V^s(\Omega | \Omega') \cap L^1_{2s}(\R^n)$.
This concludes the proof.
\end{proof}

Let us now prove the existence of minimizers to $\cI_{\Omega}$ with given exterior data $g \in V^s(\Omega | \R^n)$. Note that we expect minimizers to exist also for general $g \in V^s(\Omega | \Omega') \cap L^1_{2s}(\R^n)$, however the proof is slightly more complicated since in that case we cannot use \eqref{eq:simplified-minimizer} (see \cite{CoLo21}).

\begin{lemma}
\label{lemma:existence}
Let $g \in V^s(\Omega | \R^n)$ with $g \ge 0$. Then, there exists a minimizer $u \in V^s(\Omega | \R^n)$ to $\cI_{\Omega}$ with exterior data $g$, i.e., $u-g \in H^s_{\Omega}(\R^n)$.
\end{lemma}

\begin{proof}
Let $(u_k)_k \subset V^s(\Omega | \R^n)$ be a minimizing sequence with $u_k - g \in H^s_{\Omega}(\R^n)$ satisfying $\cI_{\Omega}(u_k) \le \cI_{\Omega}(g)$. By \autoref{lemma:aux}, we can assume that $u_k \ge 0$. It holds
\begin{align*}
[u_k-g]_{V^s(\Omega | \R^n)}^2 \le 2\Lambda \cE_{(\Omega^c \times \Omega^c)^c}(u_k-g,u_k-g) &\le C(\cE_{(\Omega^c \times \Omega^c)^c}(u_k,u_k) + \cE_{(\Omega^c \times \Omega^c)^c}(g,g)) \\
&\le C(\cI_{\Omega}(u_k) + \cI_{\Omega}(g)) \le C \cI_{\Omega}(g).
\end{align*}
Moreover, we have by the Poincar\'e-Friedrich's inequality (see \cite[Lemma 4.7]{Coz17})
\begin{align*}
\Vert u_k - g \Vert_{L^2(\Omega)}^2 \le \Vert u_k - g \Vert_{L^{\frac{2n}{n-2s}}(\Omega)}^2 \le C |\diam(\Omega)|^{\frac{2s}{n}} [u_k-g]_{V^s(\Omega | \R^n)}^2 \le C\cI_{\Omega}(g) < \infty.
\end{align*}
Thus, $(u_k)_k$ is uniformly bounded in $V^s(\Omega | \R^n)$, and hence it converges weakly in $V^s(\Omega | \R^n)$, strongly in $L^2(\Omega)$, and pointwise almost everywhere in $\Omega$, up to a subsequence, to some $u \in V^s(\Omega | \R^n)$ with $u = g$ in $\R^n \setminus \Omega$. Note that by \eqref{eq:Kcomp}, $\cE_{(\Omega^c \times \Omega^c)^c}$ induces a seminorm on $V^s(\Omega | \R^n)$, so
\begin{align*}
\cE_{(\Omega^c \times \Omega^c)^c}(u,u) \le \liminf_{k \to \infty} \cE_{(\Omega^c \times \Omega^c)^c}(u_k,u_k), \qquad \1_{\{ u > 0 \} \cap \Omega} \le \liminf_{k \to \infty} \1_{\{ u_k > 0 \} \cap \Omega},
\end{align*}
which implies $\cI_{\Omega}(u) \le \liminf_{k \to \infty}\cI_{\Omega}(u_k)$. By the construction of $u_k$, this implies that $u$ minimizes $\cI_{\Omega}$, as desired.
\end{proof}

Moreover, we deduce the following energy estimate:

\begin{lemma}
\label{lemma:energy-bound}
Let $u$ be a minimizer of $\cI_{\Omega}$. Let $R > 0$ and $x_0 \in \R^n$ such that $B_{2R}(x_0) \subset \Omega$. Then,
\begin{align*}
\cE_{B_R(x_0) \times B_R(x_0)}(u,u) &\le cR^n \left( M + R^{-2s}  \dashint_{B_{2R}(x_0)} u^2 \d x \right), \\
\tail(u;R,x_0) &\le c \left( M^{1/2} R^{s} + \dashint_{B_R(x_0)} u \d x \right)
\end{align*}
for some $c > 0$, depending only on $n,s,\lambda,\Lambda$.
\end{lemma}

\begin{proof}
Let us denote by $v$ the $L$-harmonic replacement of $u$ in $B_R(x_0)$. Then, using the algebraic estimate
\begin{align*}
(u_1 - u_2)^2 \le 2((u_1 - u_2) - (v_1 - v_2))^2 + 2(v_1 - v_2)^2,
\end{align*}
we obtain by \autoref{lemma:energy-comp-est}, as well as the Caccioppoli estimate for $v = v_+$ (see \cite[(6.1)]{Coz17}),
\begin{align*}
\cE_{B_{R/2}(x_0) \times B_{R/2}(x_0)}(u,u) &\le 2 \cE_{(B_R(x_0)^c \times B_R(x_0)^c)^c}(u-v,u-v) + 2 \cE_{B_{R/2}(x_0) \times B_{R/2}(x_0)}(v,v)\\
&\le c M R^n + c R^{-2s} \left(\int_{B_{3R/4}(x_0)} v^2 \d x + \left[\int_{B_{3R/4}(x_0)} v \d x\right] \tail(v,3R/4,x_0) \right).
\end{align*}
Let us now continue by estimating the terms on the right hand side of the previous line. \\
First, by \eqref{eq:I2-est}:
\begin{align*}
\int_{B_{3R/4}(x_0)} v^2 \d x \le 2 \int_{B_{3R/4}(x_0)} |u-v|^2 \d x + 2 \int_{B_{3R/4}(x_0)} u^2 \d x \le c M R^{n+2s} + 2 \int_{B_{R}(x_0)} u^2 \d x,
\end{align*}
and by an analogous argument, also
\begin{align*}
\dashint_{B_{3R/4}(x_0)} v \d x \le \left( \dashint_{B_{3R/4}(x_0)} |u-v|^2 \d x \right)^{1/2} +  \dashint_{B_{3R/4}(x_0)} u \d x \le c \left( M^{1/2} R^s +  \dashint_{B_{R}(x_0)} u \d x \right).
\end{align*}
Thus, since $v$ solves $L v = 0$ in $B_R(x_0)$ and is nonnegative, we have (see \cite[Theorem 1.9]{KaWe23})
\begin{align}
\label{eq:tail-est-replacement}
\tail(v,3R/4,x_0) \le c \inf_{B_{R/2}(x_0)} v \le c \dashint_{B_{R/2}(x_0)} v \d x \le c M^{1/2} R^{s} + c \dashint_{B_R(x_0)} u \d x.
\end{align}
Altogether, we obtain
\begin{align*}
\cE_{B_{R/2}(x_0) \times B_{R/2}(x_0)}(u,u) \le cR^n \left( M + R^{-2s} \dashint_{B_{R}(x_0)} u^2 \d x \right),
\end{align*}
as desired. The estimate for the tail term follows from the estimate \eqref{eq:tail-est-replacement} and since by construction
\begin{align*}
\tail(u;R,x_0) = \tail(v;R,x_0) \le \tail(v;3R/4,x_0).
\end{align*}
\end{proof}

Next, we prove the following regularity result for minimizers of $\cI_{\Omega}$. Although the H\"older exponent in this result is suboptimal, the lemma will be sufficient to show compactness of a sequence of minimizers, which will allow us to show the optimal regularity (see \autoref{thm:or}). The proof follows by an application of Campanato's theory to a nonlocal framework.

\begin{lemma}
\label{thm:aor}
Let $u$ be a minimizer of $\cI_{\Omega}$. Then, $u \in C^{\frac{2s}{n+4s}s}_{loc}(\Omega)$, and for any $R > 0$ and $x_0 \in \R^n$ with $B_{4R}(x_0) \subset \Omega$, it holds
\begin{align*}
[u]_{C^{\frac{2s}{n+4s}s}(B_R(x_0))} \le c R^{-\frac{2s}{n+4s}s} \left(\dashint_{B_{2R}(x_0)} u \d x + \tail(u;2R,x_0) + M^{1/2} \right)
\end{align*}
for some $c > 0$, depending only on $n,s,\lambda,\Lambda$.
\end{lemma}

Before we prove \autoref{thm:aor}, we observe that as a consequence, we have that $\{ u > 0\}$ is an open set. This allows us to prove \autoref{lemma:aux}(iv):

\begin{proof}[Proof of \autoref{lemma:aux}(iv)]
Let $t \in \R$ and $\phi \in H^s_{\Omega}(\R^n)$ with $\supp(\phi) \subset \{ u > 0\}$. Then, since $u$ is a minimizer of $\cI_{\Omega}$, we get
\begin{align*}
\iint_{(\Omega^c \times \Omega^c)^c}  \big[ (u(x)-u(y))^2 - (u(x)+t\phi(x)-u(y)-t\phi(y))^2 \big] K(x-y) \d y \d x \le 0,
\end{align*}
and by taking the limit $t \to 0$, we get
\begin{align*}
\iint_{(\Omega^c \times \Omega^c)^c} (u(x)-u(y))(\phi(x) - \phi(y)) K(x-y) \d y \d x  = 0.
\end{align*}
Since $\{ u > 0 \}$ is open, this yields that $L u= 0$ in $\{ u > 0\} \cap \Omega$ in the weak sense, as desired.
\end{proof}

Now, we turn to the proof of \autoref{thm:aor}. 

\begin{proof}[Proof of \autoref{thm:aor}]
First, we claim that for any $\sigma \in (0,1/2)$, $R > 0$, and $x_0 \in \R^n$ such that $B_{2R}(x_0) \subset \Omega$, it holds
\begin{align}
\label{eq:aor-claim}
\left(\dashint_{B_{\sigma R}(x_0)} |u - (u)_{\sigma R,x_0}|^2\right)^{1/2} \le c  \sigma^s \left(\dashint_{B_R(x_0)} |u| \d x + \tail(u;R,x_0) \right) + c  M^{1/2} \sigma^{-\frac{n}{2}} R^{s},
\end{align}
where we denote $(w)_{r,x_0} = \dashint_{B_r(x_0)} w \d x$. Let $v$ be the $L$-harmonic replacement of $u$ in $B_R(x_0)$. We compute
\begin{align*}
\dashint_{B_{\sigma R}(x_0)} |u - (u)_{\sigma R,x_0}|^2 \le  \dashint_{B_{\sigma R}(x_0)} |v - (v)_{\sigma R,x_0}|^2 + 2\dashint_{B_{\sigma R}(x_0)} |u - v|^2 =: I_1 + I_2.
\end{align*}
For $I_2$, we simply estimate using the fractional Poincar\'e-Friedrich's inequality (see \cite[Lemma 4.7]{Coz17}), (which is applicable since $u-v = 0$ in $B_R(x_0)^c$), and \autoref{lemma:energy-comp-est}:
\begin{align}
\label{eq:I2-est}
I_2 &\le c \sigma^{-n} R^{2s-n} \cE_{(B_R(x_0)^c \times B_R(x_0)^c)^c}(u-v,u-v) \le c M \sigma^{-n} R^{2s}.
\end{align}
For $I_1$, we make use of the $C^s$-interior estimate (see \cite[Theorem 2.4.3]{FeRo24}) and the $L^{\infty}-L^1$ estimate (see \cite[Theorem 6.2]{Coz17}) for $L$-harmonic functions to obtain
\begin{align*}
I_1 \le [v]^2_{C^{s}(B_{R/2}(x_0))} (\sigma R)^{2s} &\le c\sigma^{2s} \left(\Vert v \Vert_{L^{\infty}(B_R(x_0))} + \tail(v;R,x_0) \right)^2\\
&\le c\sigma^{2s} \left(\dashint_{B_R(x_0)} |v| \d x + \tail(u;R,x_0) \right)^2,
\end{align*}
where we also used that $u$ and $v$ coincide in $B_R(x_0)^c$.
Finally, we observe that by \eqref{eq:I2-est}:
\begin{align*}
\dashint_{B_R(x_0)} |v| \d x &\le c \dashint_{B_R(x_0)} |u| \d x + c \dashint_{B_R(x_0)} |u-v| \d x\\
&\le c \dashint_{B_R(x_0)} |u| \d x + c M^{1/2} \sigma^{-\frac{n}{2}} R^{s}.
\end{align*}
Thus, altogether we have proved \eqref{eq:aor-claim}. From here, we intend to deduce the desired result from Campanato's characterization of H\"older spaces. In fact, taking any ball $B_{R_0}(z) \subset \Omega$ and setting $d = \dist(B_{R_0}(z) , \partial \Omega)$ and $R_1 = \frac{d}{2} + R_0$, we obtain for any $x_0 \in B_{R_0}(z)$ and $r \le R < \min \{1,d/8\}$:
\begin{align*}
\dashint_{B_{r}(x_0)} |u - (u)_{r,x_0}| &\le c  \left( \frac{r}{R} \right)^s \left(\dashint_{B_{4R}(x_0)} |u| \d x + \tail(u;4R,x_0) \right) + c M^{1/2} \left( \frac{r}{R} \right)^{-\frac{n}{2}} R^{s}\\
&\le c  \left( \frac{r}{R} \right)^s \left( \Vert u\Vert_{L^{\infty}(B_{R_1}(z))} + \tail(u;R_1,z) \right) + c M^{1/2} \left( \frac{r}{R} \right)^{-\frac{n}{2}} R^{s},
\end{align*}
where we used that $B_{4R}(x_0) \subset B_{R_1}(z)$ by construction and 
\begin{align*}
\tail(u;4R,x_0) \le c \tail(u;R_1,z) + c \Vert u \Vert_{L^{\infty}(B_{R_1}(z))}.
\end{align*}
Therefore, choosing $R = 2r^{\theta}$ for some $\theta \in (0,1)$, we get 
\begin{align*}
\dashint_{B_{r}(x_0)} |u - (u)_{r,x_0}| &\le c r^{(1-\theta)s} \left( \Vert u\Vert_{L^{\infty}(B_{R_1}(z))} + \tail(u;R_1,z) \right) + c M^{1/2} r^{-\frac{(1-\theta)n}{2} + \theta s}.
\end{align*}
Let us choose $\theta = \frac{n+2s}{n+4s} \in (0,1)$. Then, we have $r^{-\frac{(1-\theta)n}{2} + \theta s - (1-\theta)s} \le 1$, which implies that after dividing by $r^{(1-\theta)s}$, we get
\begin{align*}
r^{-n-(1-\theta)s} \int_{B_{r}(x_0) \cap B_{R_0}(z)} |u - (u)_{r,x_0}| \le c \left( \Vert u\Vert_{L^{\infty}(B_{R_1}(z))} + \tail(u;R_1,z) + M^{1/2} \right).
\end{align*}
Thus, $u \in \mathcal{L}^{1,n+(1-\theta)s}(B_{R_0}(z))$, i.e., $u \in C^{(1-\theta)s}(B_{R_0}(z)) = C^{\frac{2s}{n+4s}s}(B_{R_0}(z))$, and we have the following estimate for any $x_0 \in \R^n$, $R > 0$ such that $B_{4R}(x_0) \subset \Omega$:
\begin{align*}
[u]_{C^{\frac{2s}{n+4s}s}(B_R(x_0))} \le c R^{-\frac{2s}{n+4s}s} \left(\Vert u \Vert_{L^{\infty}(B_{2R}(x_0))} + \tail(u;2R,x_0) + M^{1/2} \right).
\end{align*}
From here, the desired result follows by application of (ii), (iii) and the local boundedness estimate (see \cite[Theorem 6.2]{Coz17}), just as in the proof of \autoref{lemma:aux}(iii).
\end{proof}

\subsection{Optimal regularity}
\label{sec:opt-reg}

The goal of this section is to prove that minimizers of $\cI_{\Omega}$ are in $C^s_{loc}(\Omega)$, which is the first claim in \autoref{thm-onephase-intro}.

The following theorem is the main result of this section. The first part of \autoref{thm-onephase-intro} follows immediately from \autoref{lemma:contradiction} and \autoref{thm:or}.

\begin{theorem}
\label{thm:or}
Let $u$ be a minimizer of $\cI_{\Omega}$. Then, $u \in C^s_{loc}(\Omega)$, and for any $R > 0$ and $x_0 \in \R^n$ with $B_{2R}(x_0) \subset \Omega$, we have
\begin{align*}
\Vert u \Vert_{C^s(B_R(x_0))} \le C R^{-s} \left( 1 + \dashint_{B_{2R}(x_0)} u \ d x \right)
\end{align*}
for some constant $C > 0$, depending only on $n,s,\lambda,\Lambda,M$.
\end{theorem}

We need the following lemma on the convergence of almost minimizers to nonlocal Dirichlet energies.

\begin{lemma}
\label{lemma:convergence}
Let $R > 0$. Let $(u_k)_k \subset V^s(B_R | B_{2R}) \cap L^1_{2s}(\R^n)$ be a sequence of functions with $u_k \ge 0$ that is uniformly bounded in $H^s(B_R) \cap L^1_{2s}(\R^n)$, i.e.,
\begin{align*}
[ u_k ]_{H^s(B_R)} + \Vert u_k \Vert_{L^2(B_R)} + \Vert u_k \Vert_{L^1_{2s}(\R^n)} \le C.
\end{align*} 
Moreover, assume that $u_k \to u_{\infty} \in L^1_{2s}(\R^n)$ and that for every $M > 1$ there is $k_M \in \N$ such that
\begin{align}
\label{eq:unif-tail-decay-weak-sol}
|u_k| \le C(M) ~~ \text{ in } B_M, \qquad \int_{\R^n \setminus B_{M}} |u_k(x)| |x|^{-n-2s} \d x \le c(M) \qquad \forall k \ge k_m
\end{align}
for some $C(M) > 0$ and $c(M) \searrow 0$, as $M \to \infty$.
Moreover, let $(f_k)_k \subset L^2(B_R)$ with $f_k \to f_{\infty} \in L^2(B_R)$ weakly in $L^2(B_R)$, and $(\eps_k)_k$ such that $\eps_k \searrow 0$, as $k \to \infty$, let $(K_k)_k$ be a sequence of kernels satisfying \eqref{eq:Kcomp} for some $0 < \lambda \le \Lambda$, and assume that for any $\phi \in H_{B_r}^s(\R^n)$ and $r < R$ it holds:
\begin{align}
\label{eq:almost-min}
\cE^{K_k}_{(B_R^c \times B_R^c)^c}(u_k,u_k) - \int_{B_R} f_k u_k \d x \le \cE^{K_k}_{(B_R^c \times B_R^c)^c}(u_k + \phi, u_k + \phi) - \int_{B_R} f_k(u_k + \phi_k) \d x + \eps_k.
\end{align}
Then, there exists $L_{\infty} \in \mathcal{L}_s^n(\lambda,\Lambda)$ with  kernel $K_{\infty}$ such that $\min\{1,|h|^2\}K_k(h) \d h \to \min\{1,|h|^2\} K_{\infty}(h)\d h$ weakly in the sense of measures, and $u_{\infty}$ is a weak solution to $L_{K_{\infty}} u_{\infty} = f_{\infty}$ in $B_R$, and up to a subsequence, it holds $u_k \to u_{\infty}$ weakly in $H^s(B_R)$.
\end{lemma}

\begin{proof}
Since $(u_k)_k$ is uniformly bounded in $H^s(B_R)$, by weak compactness, we get that up to a subsequence, $u_k \rightharpoonup u_{\infty}$ weakly in $ H^s(B_R)$.  Moreover, the existence of $K_{\infty}$ such that $\min\{1,|h|^2\} K_k(h) \d h \to \min\{1,|h|^2\} K_{\infty}(h) \d h$ weakly in the sense of measures follows from \cite[Proof of Proposition 2.2.36]{FeRo24}. 
Next, we claim that
\begin{align}
\label{eq:energy-convergence-weak}
\cE^{K_k}_{(B_r^c \times B_r^c)^c}(u_k,\phi) \to \cE^{K_{\infty}}_{(B_r^c \times B_r^c)^c}(u_{\infty},\phi).
\end{align}
To see this, we first use the uniform boundedness of $(u_k)_k$ in $H^s(B_R)$ to obtain:
\begin{align*}
\cE^{|K_k - K_{\infty}|}_{B_R \times B_R}(u_k,\phi) \le c \left(\int_{\R^n} \left[ \int_{\R^n} (\phi(x) - \phi(x+h))^2 |K_k(h) - K_{\infty}(h)| \d h \right] \d x \right)^{1/2} \to 0,
\end{align*}
by the dominated convergence theorem. Moreover, since $\supp(\phi) \subset B_{r}$, we have
\begin{align*}
\cE^{|K_k - K_{\infty}|}_{B_r \times (\R^n \setminus B_R)}(u_k,\phi) &\le c \left( \int_{B_{r}} |u_k(x)| |\phi(x)| \d x \right) \sup_{x \in B_{r}}  \left( \int_{\R^n \setminus B_R} |K_k(x-y) - K_{\infty}(x-y) | \d y \right) \\
&\quad +  c \int_{B_{r}} |\phi(x)| \left( \int_{\R^n \setminus B_R} |u_k(y)|  |K_k(x-y) - K_{\infty}(x-y) | \d y  \right)  \d x\\
&= K_1 + K_2.
\end{align*}
Here, $K_1 \to 0$, as $k \to \infty$, by the uniform boundedness of $(u_k)_k$ in $L^2(B_R)$ and the convergence $\min\{1 , |y|^2\} K_k(y) \d y \to \min\{1 , |y|^2\}  K_{\infty} \d y$. For $K_2$, we use that by \eqref{eq:unif-tail-decay-weak-sol}, for any $\delta \in (0,1)$ we can find $k(\delta) \in \N$ and $M > 1$ such that for any $k \ge k(\delta)$ it holds
\begin{align*}
\sup_{x \in B_{r}} & \left( \int_{\R^n \setminus B_R} |u_k(y)|  |K_k(x-y) - K_{\infty}(x-y) | \d y  \right) \\
&\le C(M) \sup_{x \in B_{r}} \int_{B_M} |K_k(x-y) - K_{\infty}(x-y) | \d y + C \int_{\R^n \setminus B_M} |u_k(y)| |y|^{-n-2s} \d y \\
&\le  C(M) \sup_{x \in B_{r}} \int_{B_M} |K_k(x-y) - K_{\infty}(x-y) | \d y + C \delta. 
\end{align*}
Thus, $K_2 \to 0$, by using again the convergence $\min\{1 , |y|^2\} K_k(y) \d y \to \min\{1 , |y|^2\}  K_{\infty} \d y$ on the first summand. This yields $\cE^{|K_k - K_{\infty}|}_{(B_r^c \times B_r^c)^c}(u_k,\phi) \to 0$. Moreover, note that by the weak convergence $u_k \to u$ in $H^s(B_R)$, we get $\cE^{K_{\infty}}_{B_R \times B_R}(u_k - u_{\infty},\phi) \to 0$. Also, by the convergence in $L^1_{2s}(\R^n)$ and in $L^2(B_R)$, which follows from the boundedness in $H^s(B_R)$ by a compact embedding (see \cite{FoKa24}), we have
\begin{align*}
\cE^{K_{\infty}}_{B_r \times (\R^n \setminus B_R)}(u_k - u_{\infty},\phi) & \le \int_{B_{r}} |\phi(x)| |u_k(x) - u_{\infty}(x)| \left(\int_{\R^n \setminus B_R} K_{\infty}(x-y) \d y \right) \d x \\
&\quad + \int_{B_{r}} |\phi(x)| \left(\int_{\R^n \setminus B_R} |u_k(y) - u_{\infty}(y)| K_{\infty}(x-y) \d y \right) \d x \to 0.
\end{align*}
This shows $\cE^{K_{\infty}}_{(B_r^c \times B_r^c)^c}(u_k - u_{\infty},\phi) \to 0$, and a combination of the previous findings, together with the triangle inequality yields
\begin{align*}
\big|\cE^{K_k}_{(B_r^c \times B_r^c)^c}(u_k,\phi) - \cE^{K_{\infty}}_{(B_r^c \times B_r^c)^c}(u_{\infty},\phi) 
\big| \le \cE^{|K_k - K_{\infty}|}_{(B_r^c \times B_r^c)^c}(u_k,\phi) + \cE^{K_{\infty}}_{(B_r^c \times B_r^c)^c}(u_k - u_{\infty},\phi) \to 0,
\end{align*}
and therefore \eqref{eq:energy-convergence-weak}.\\
In order to prove that $L_{\infty} u_{\infty} = f_{\infty}$ in $B_R$, let us take any $\phi \in H^s_{B_r}(\R^n)$ for some $r \in (0,R)$ and $t_k \in \R$ and deduce from \eqref{eq:almost-min} that the following holds true:
\begin{align*}
0 \le t_k \cE^k_{(B_r^c \times B_r^c)^c}(u_k,\phi) + t_k^2 \cE^k_{(B_r^c \times B_r^c)^c}(\phi,\phi) - t_k \int_{B_r} f_k \phi \d x + \eps_k.
\end{align*}
Choosing $t_k = \eps_k^{1/2} > 0$, we deduce that
\begin{align*}
\cE^{\infty}_{(B_r^c \times B_r^c)^c}(u_{\infty},\phi) &= \lim_{k \to \infty} \cE^k_{(B_r^c \times B_r^c)^c}(u_k,\phi) \\
&\ge -\lim_{k \to \infty} \eps_k^{1/2}\left(1 + \cE^k_{(B_r^c \times B_r^c)^c}(\phi,\phi) \right) + \lim_{k \to \infty} \int_{B_r} f_k \phi \d x = \int_{B_r} f_{\infty} \phi \d x,
\end{align*}
where we used \eqref{eq:energy-convergence-weak} and the weak convergence $f_k \to f_{\infty}$ in $L^2(B_R)$ together with a density argument. Since $r < R$ was arbitrary, this implies that $L_{\infty} u_{\infty} \ge f_{\infty}$ in the weak sense in $B_R$. Choosing $t_k = - \eps_k^{1/2}$, the same argument yields that $L_{\infty} u_{\infty} \le f_{\infty}$ in the weak sense in $B_R$. This concludes the proof.
\end{proof}

The following lemma is the key ingredient in the proof of the optimal regularity (see \autoref{thm:or}). It shows that minimizers of $\cI$ in $B_2$ that vanish at zero are uniformly bounded in a ball of radius $1/8$, with a bound that does not depend on the exterior data. Moreover, it implies the first part of \autoref{thm-onephase-intro}.

\begin{lemma}
\label{lemma:contradiction}
Let $u$ be a minimizer of $\cI_{\Omega}$ with $B_2 \subset \Omega$, and $u(0) = 0$.
Then, there is a constant $C > 0$, depending only on $n,s,\lambda,\Lambda,M$, such that 
\begin{align*}
\Vert u \Vert_{L^{\infty}(B_{1/8})} \le C.
\end{align*}
\end{lemma}

\begin{proof}
Let us assume by contradiction that there exist sequences of operators $(L_k)_k \subset \mathcal{L}_s^n(\lambda,\Lambda)$ with kernels $K_k$, and $(u_k)_k \subset V^s(B_2 | B_3) \cap L^1_{2s}(\R^n)$ minimizers of $\cI_{\Omega}$, but such that $\Vert u_k \Vert_{L^{\infty}(B_{1/8})} \to \infty$. For any $k \in \N$, we define
\begin{align*}
\mathcal{W}_k := \left\{ x \in B_1 : \dist(x,\{ u_k = 0 \}) \le \frac{\dist(x,\partial B_1)}{3} \right\}.
\end{align*} 
We immediately have that $B_{1/4} \subset \mathcal{W}_k$ and that the $u_k$ are continuous in $\overline{B_1}$, which yields that the $u_k$ achieve their respective maxima in the closed set $\mathcal{W}_k$ at a point $x_k \in \mathcal{W}_k$ and
\begin{align*}
\Vert u_k \Vert_{L^{\infty}(B_{1/8})} \le \max_{\mathcal{W}_k} u_k \le u_k(x_k) =: S_k.
\end{align*} 
Next, we let $y_k \in \partial \{ u_k > 0\}$ be the projection of $x_k$ onto $\partial \{ u_k > 0\}$ and we set $r_k := |x_k - y_k|$. An easy computation (see \cite[Proposition 3.15]{Vel23}) yields that $B_{r_k/2}(y_k) \subset \mathcal{W}_k$, which implies
\begin{align}
\label{eq:Mk-control}
\sup_{B_{r_k/2}(y_k)} u_k \le S_k.
\end{align}
Note that by construction it holds $B_{r_k}(x_k) \subset \{ u_k > 0\}$, which means that $u_k$ satisfies $L_k u_k = 0$ in the weak sense in $B_{r_k}(x_k)$ by \autoref{lemma:aux}(iv). Since $u_k \ge 0$, the Harnack inequality (see \cite[Theorem 6.9]{Coz17}) implies
\begin{align*}
u_k(z_k) \ge c u_k(x_k) = c S_k
\end{align*}
for some $c > 0$, depending only on $n,s,\lambda,\Lambda$, where we set $z_k = \frac{x_k}{8} + \frac{7y_k}{8}$. Altogether, we obtain
\begin{align*}
c S_k \le u_k(z_k) \le S_k.
\end{align*}
Now, we set $\rho_k = r_k/4$ and $\zeta_k = \frac{z_k - y_k}{\rho_k}$. Moreover, we introduce the functions
\begin{align*}
v_k(x) = \frac{u_k(y_k + \rho_k x)}{u_k(z_k)}, \qquad w_k(x) = v_k(x)\1_{B_2}(x).
\end{align*}
The goal is to apply \autoref{lemma:convergence} to $w_k$ in $B_1$. The uniform boundedness of the sequence $(w_k)_k$ in $H^s(B_1)$ follows from
\begin{align}
\label{eq:verify-L2-bound-1}
\Vert w_k \Vert_{L^{\infty}(\R^n)}^2 = \Vert v_k \Vert_{L^2(B_2)}^2 \le c S_k^{-2} \rho_k^{-n} \Vert u_k \Vert_{L^{\infty}(B_{2\rho_k}(y_k))}^2 \le cS_k^{-2} \Vert u_k \Vert_{L^{\infty}(B_{r_k/2}(y_k))}^2 \le c,
\end{align}
In particular, this implies
\begin{align}
\label{eq:verify-L2-bound-2}
\Vert w_k \Vert_{L^2(B_2)}^2 \le c, \qquad \Vert w_k \Vert_{L^1_{2s}(\R^n)} \le c,
\end{align}
and therefore also \eqref{eq:unif-tail-decay-weak-sol}. Moreover, we have the following computation based on \autoref{lemma:energy-bound} and a scaling argument
\begin{align*}
[w_k]^2_{H^s(B_1)} = [v_k]^2_{H^s(B_1)} &= (u_k(z_k))^{-2} \rho_k^{-n+2s} [u_k]^2_{H^s(B_{\rho_k}(y_k))} \\
&\le c S_k^{-2} \rho_k^{2s} \left( M + \rho_k^{-2s} \dashint_{B_{2\rho_k}(y_k)} u_k^2 \d x \right) \\
&\le c M S_k^{-2} \rho_k^{2s} + c \frac{\Vert u_k \Vert_{L^{\infty}(B_{r_k/2}(y_k))}^2}{S_k^2} \le c,
\end{align*}
where we used \eqref{eq:Mk-control}, as well as $r_k \le 2$ and $S_k^{-2} \le \Vert u_k \Vert_{L^{\infty}(B_{1/8})}^{-2} \to 0$, as $k \to \infty$. The convergence $w_k \to w_{\infty} \in L^1_{2s}(\R^n)$ strongly in $L^1_{2s}(\R^n)$ up to a subsequence is immediate from \eqref{eq:verify-L2-bound-1} and Vitali's convergence theorem.

Thus, in order to apply \autoref{lemma:convergence}, it suffices to prove that also \eqref{eq:almost-min} is satisfied. To see this, let us first observe that by \autoref{lemma:aux}(iv) $v_k$ solves $L_{\tilde{K}_k} v_k = 0$ (where $\tilde{K}_k(h) = \rho_k^{n+2s}K_k(\rho_k h)$ still satisfies \eqref{eq:Kcomp} with $0 < \lambda \le \Lambda$, i.e., $L_{\tilde{K}_k} \in \mathcal{L}_s^n(\lambda,\Lambda)$) in the weak sense in $\{ v_k > 0 \} \cap B_2$. Hence, $w_k$ solves $L_{\tilde{K}_k} w_k = f_k$ in the weak sense in $\{ w_k > 0 \} \cap B_2$, where for any $x \in B_{1}$
\begin{align*}
0 \le f_k(x) := 2\int_{\R^n \setminus B_2} v_k(y) \tilde{K}_k(x-y) \d y &\le c (u_k(z_k))^{-1}  \tail(u_k ; 2\rho_k, y_k) \\
&\le c (u_k(z_k))^{-1} \dashint_{B_{\rho_k}(y_k)} u_k \d x \le c  \frac{\Vert u_k \Vert_{L^{\infty}(B_{r_k/2}(y_k))}}{S_k} \le c,
\end{align*}
where we used the tail estimate (see \cite[Theorem 1.9]{KaWe23}). In particular, we have that $(f_k)_k$ is uniformly bounded in $L^2(B_1)$, and in particular there exists a weakly converging subsequence.\\
Consequently, for any $\phi_k \in V^s(B_1 | \R^n)$ such that $w_k - \phi_k \in H^s_{B_1}(\R^n)$, we have
\begin{align}
\label{eq:energy-with-minimum}
\tilde{\cE}^k_{(B_1^c \times B_1^c)^c}(w_k,w_k) - \int_{B_1} f_k w_k \d x \le \tilde{\cE}^k_{(B_1^c \times B_1^c)^c}(w_k \wedge \phi_k,w_k \wedge \phi_k) - \int_{B_1} f_k (w_k \wedge \phi_k) \d x,
\end{align}
where we denote $\tilde{\cE}^k := \cE^{\tilde{K}_k}$. On the other hand, since $u_k$ is a minimizer of $\cI_{\Omega}$, it holds
\begin{align}
\label{eq:energy-with-maximum}
\tilde{\cE}^k_{(B_1^c \times B_1^c)^c}(w_k,w_k) - \int_{B_1} f_k w_k \d x \le \tilde{\cE}^k_{(B_1^c \times B_1^c)^c}(w_k \vee \phi_k,w_k \vee \phi_k) - \int_{B_1} f_k (w_k \vee \phi_k) \d x + c M S_k^{-2}.
\end{align}
To see this, set $\tilde{\phi}_k = \phi_k(-y_k + \cdot/\rho_k) u_k(z_k)$, $\tilde{\Omega} = -y_k + \rho_k^{-1}\Omega$ and observe that $u_k = \tilde{\phi}_k$ in $B_{\rho_k}(y_k)^c$ and therefore,
\begin{align*}
\cE^k_{(\Omega^c \times \Omega^c)^c}(u_k \vee \tilde{\phi}_k,u_k \vee \tilde{\phi}_k) &+ M |\{ u_k \vee \tilde{\phi_k} > 0 \} \cap \Omega| \\
&\ge \cE^k_{(\Omega^c \times \Omega^c)^c}(u_k,u_k) + M |\{ u_k > 0 \} \cap \Omega| \\
&= \cE^k_{(B_{\rho_k}(y_k)^c \times B_{\rho_k}(y_k)^c)^c}(u_k,u_k) \\
&\quad +\cE^k_{(\Omega^c \times \Omega^c)^c \setminus (B_{\rho_k}(y_k)^c \times B_{\rho_k}(y_k)^c)^c}(u_k \vee \tilde{\phi}_k,u_k \vee \tilde{\phi}_k) +  M |\{ u_k > 0 \} \cap \Omega|,
\end{align*}
where we denote $\cE^k := \cE^{K_k}$. Note that one should understand the estimates from above in the sense of \autoref{def:minimizer} in case the energies are not finite. However, this does not change any of the arguments. The previous estimate yields 
\begin{align*}
\cE^k_{(B_{\rho_k}(y_k)^c \times B_{\rho_k}(y_k)^c)^c}(u_k,u_k) &\le \cE^k_{(B_{\rho_k}(y_k)^c \times B_{\rho_k}(y_k)^c)^c}(u_k \vee \tilde{\phi}_k,u_k \vee \tilde{\phi}_k) \\
&\quad + M |\{ \tilde{\phi}_k > 0 \} \cap \{u_k = 0\} \cap B_{\rho_k}(y_k)|,
\end{align*}
and therefore, by scaling:
\begin{align}
\label{eq:energy-with-maximum-help-1}
\begin{split}
\tilde{\cE}^k_{(B_1^c \times B_1^c)^c}(v_k,v_k) &= (u_k(z_k))^{-2} \rho_k^{-n+2s} \cE^k_{(B_{\rho_k}(y_k)^c \times B_{\rho_k}(y_k)^c)^c}(u_k,u_k)\\
&\le (u_k(z_k))^{-2} \rho_k^{-n+2s} \big[ \cE^k_{(B_{\rho_k}(y_k)^c \times B_{\rho_k}(y_k)^c)^c}(u_k \vee \tilde{\phi}_k,u_k \vee \tilde{\phi}_k)\\ 
&\quad + M |\{ \phi_k > 0 \} \cap \{v_k = 0\} \cap B_{\rho_k}(y_k)| \big]\\
&\le \tilde{\cE}^k_{(B_1^c \times B_1^c)^c}(v_k \vee \phi_k,v_k \vee \phi_k) + c M (u_k(z_k))^{-2} \rho_k^{2s}.
\end{split}
\end{align}
Finally, recalling that $w_k = v_k \1_{B_2}$, we have (at least formally)
\begin{align*}
\tilde{\cE}^k_{(B_1^c \times B_1^c)^c}(w_k,w_k) &=  \tilde{\cE}^k_{(B_2 \times B_2) \setminus (B_1^c \times B_1^c)}(v_k,v_k) + 2 \tilde{\cE}^k_{B_1 \times B_2^c}(w_k,w_k)\\
&= \tilde{\cE}^k_{(B_1^c \times B_1^c)^c}(v_k,v_k) + 2 \tilde{\cE}^k_{B_1 \times B_2^c}(w_k,w_k) - 2 \tilde{\cE}^k_{B_1 \times B_2^c}(v_k,v_k) \\
&= \tilde{\cE}^k_{(B_1^c \times B_1^c)^c}(v_k,v_k) - 2\int_{B_1} v_k(x) \left( \int_{\R^n \setminus B_2} v_k(y) \tilde{K}_k(x-y) \d y \right) \d x \\
&\quad + \int_{B_1} \left( \int_{\R^n \setminus B_2} v_k^2(y) \tilde{K}_k(x-y) \d y \right) \d x \\
&= \tilde{\cE}^k_{(B_1^c \times B_1^c)^c}(v_k,v_k) - \int_{B_1} w_k f_k \d x + \int_{B_1} \left( \int_{\R^n \setminus B_2} v_k^2(y) \tilde{K}_k(x-y) \d y \right) \d x.
\end{align*}
Thus, carrying out an analogous computation for $\tilde{\cE}^k_{(B_1^c \times B_1^c)^c}(w_k \vee \phi_k,w_k \vee \phi_k)$ and recalling that $w_k \vee \phi_k = w_k$ in $\R^n \setminus B_1$, it holds
\begin{align*}
&\tilde{\cE}^k_{(B_1^c \times B_1^c)^c}(w_k,w_k) - \tilde{\cE}^k_{(B_1^c \times B_1^c)^c}(w_k \vee \phi_k,w_k \vee \phi_k) \\
&\qquad\qquad = \tilde{\cE}^k_{(B_1^c \times B_1^c)^c}(v_k,v_k) -  \tilde{\cE}^k_{(B_1^c \times B_1^c)^c}(v_k \vee \phi_k,v_k \vee \phi_k) - \int_{B_1} w_k f_k \d x + \int_{B_1} (w_k \vee \phi_k) f_k \d x.
\end{align*}
Combining this estimate with \eqref{eq:energy-with-maximum-help-1}, we immediately deduce \eqref{eq:energy-with-maximum}, as desired.\\
Next, we observe the following algebraic identity
\begin{align}
\label{eq:min-max-alg-identity}
\begin{split}
\big((w_1 \vee \phi_1) - (w_2 \vee \phi_2) \big)^2 &+ \big((w_1 \wedge \phi_1) - (w_2 \wedge \phi_2) \big)^2 \\
&= (w_1 - w_2)^2 + (\phi_1 - \phi_2)^2 -2(w_1 - \phi_1)_+(w_2 - \phi_2)_- \\
&\le (w_1 - w_2)^2 + (\phi_1 - \phi_2)^2
\end{split}
\end{align}
which implies that
\begin{align*}
\tilde{\cE}^k_{(B_1^c \times B_1^c)^c}(w_k \vee \phi_k,w_k \vee \phi_k) & + \tilde{\cE}^k_{(B_1^c \times B_1^c)^c}(w_k \wedge \phi_k,w_k \wedge \phi_k) \\
&\le \tilde{\cE}^k_{(B_1^c \times B_1^c)^c}(w_k,w_k) + \tilde{\cE}^k_{(B_1^c \times B_1^c)^c}(\phi_k , \phi_k).
\end{align*}
Moreover, clearly 
\begin{align*}
\int_{B_1} f_k (w_k \wedge \phi_k) \d x + \int_{B_1} f_k (w_k \vee \phi_k) \d x = \int_{B_1} f_k w_k \d x + \int_{B_1} f_k \phi_k.
\end{align*}
Thus, by adding the previous two observations and combining them with \eqref{eq:energy-with-minimum} and \eqref{eq:energy-with-maximum}, we obtain
\begin{align*}
\tilde{\cE}^k_{(B_1^c \times B_1^c)^c}(w_k,w_k) - \int_{B_1} f_k w_k \d x \le \tilde{\cE}^k_{(B_1^c \times B_1^c)^c}(\phi_k,\phi_k) - \int_{B_1} f_k \phi_k \d x + c M S_k^{-2},
\end{align*}
i.e., \eqref{eq:almost-min} with $\eps_k := c M S_k^{-2} \to 0$, as $k \to \infty$, as desired. Therefore, we can apply \autoref{lemma:convergence}, and obtain that $w_k \to w_{\infty} \in H^s(B_r) \cap L^1_{2s}(\R^n)$ weakly in $H^s(B_r) \cap L^1_{2s}(\R^n)$, and strongly in $L^2(B_r)$ for any $r \in (0,1)$, up to a subsequence. Moreover, $f_k \to f_{\infty} \in L^2(B_1)$ weakly in $L^1(B_1)$ and $f_{\infty} \ge 0$. Moreover, \autoref{lemma:convergence} implies $w_{\infty} \ge 0$ and that there exists $L_{\infty} \in \mathcal{L}_s^n(\lambda,\Lambda)$ such that $L_{\infty} v_{\infty} = f_{\infty} \ge 0$ in the weak sense $B_1$. By \autoref{thm:aor}, we have for any $R \in (0,2)$
\begin{align*}
\Vert v_k \Vert_{C^{\alpha}(B_R)} \le c(R) \left( \dashint_{B_2} v_k \d x + \tail(v_k;2)+ M^{1/2} \right) \le c(R),
\end{align*}
where $\alpha = \frac{2s}{n+4s}s$, and we also used \eqref{eq:verify-L2-bound-2}. 
Thus, by Arzela-Ascoli's theorem, we get that $w_k \to w_{\infty}$ locally uniformly in $B_2$, up to extraction of a further subsequence. Let us also observe that $\zeta_{k} \to \zeta_{\infty} \in \partial B_{1/2}$ up to a further subsequence. Therefore, we have
\begin{align*}
w_{\infty}(0) = \lim_{k \to \infty} w_k(0) = 0, \qquad w_{\infty}(\zeta_{\infty}) = \lim_{k \to \infty} w_k(\zeta_k) = 1.
\end{align*}
Since $w_{\infty} \ge 0$ by the locally uniform convergence in $B_2$, the above two lines contradict the fact that $L_{\infty} w_{\infty} \ge 0$ in $B_1$ due to the weak Harnack inequality. This proves the desired result.
\end{proof}

By combination of the previous results, we are now in position to prove \autoref{thm:or}.

\begin{proof}[Proof of \autoref{thm:or}]
Note that if $y_0 \in \partial \{ u > 0 \} \cap B_R(x_0)$, then for any $r \in (0,R)$, $u_r(x) = u(y_0 + rx)/r^s$ satisfies the assumptions of \autoref{lemma:contradiction} (with $K$ replaced by $r^{n+2s} K(r \cdot)$, which still satisfies \eqref{eq:Kcomp} with $0 < \lambda \le \Lambda$). As a consequence,
\begin{align}
\label{eq:boundary-Cs}
\Vert u \Vert_{L^{\infty}(B_r(y_0))} = r^s \Vert u_r \Vert_{L^{\infty}(B_1)} \le C r^s.
\end{align}
Moreover, since $Lu = 0$ in $\{ u > 0 \} \cap B_R(x_0)$ by \autoref{lemma:aux}, we have for any $r \in (0,R)$ and $y_0 \in \{u > 0 \} \cap B_R(x_0)$ with $B_{2r}(y_0) \subset (\{ u > 0 \} \cap B_R(x_0))$
\begin{align}
\label{eq:interior-Cs}
\begin{split}
\Vert u \Vert_{C^s(B_r(y_0))} &\le C r^{-s} \left( \Vert u \Vert_{L^{\infty}(B_{3r/2}(y_0))} + \tail(u;r,y_0) \right) \\
&\le C r^{-s} \left(\dashint_{B_r(y_0)} u \d x  + \tail(u;r,y_0) \right) \le C r^{-s} \dashint_{B_r(y_0)} u \d x,
\end{split}
\end{align}
by \cite[Theorem 2.4.3]{FeRo24}, where we used also that $u \ge 0$ by \autoref{lemma:aux} and the local boundedness and tail estimate (see \cite[Theorem 6.2]{Coz17}, \cite[Theorem 1.9]{KaWe23}). From here, the proof follows by a standard combination of \eqref{eq:boundary-Cs} and \eqref{eq:interior-Cs} (see for instance \cite[Proof of Proposition 2.6.4]{FeRo24}).
\end{proof}

\begin{proof}[Proof of the first part of \autoref{thm-onephase-intro}]
By rescaling \autoref{lemma:contradiction}, we know that $u(x) \le c |x|^s$ in $B_{3/4}$ since $0 \in \partial \{ u > 0 \}$. Hence, an application of \autoref{thm:or} yields
\begin{align*}
\Vert u \Vert_{C^s(B_{1/2})} \le C + C \int_{B_{3/4}} u(x) \d x \le C,
\end{align*}
as desired.
\end{proof}

\subsection{Non-degeneracy}
\label{sec:non-deg}

The goal of this section is to establish non-degeneracy for minimizers of $\cI$ near the free boundary. It directly implies the second claim in \autoref{thm-onephase-intro}.

\begin{theorem}[Non-degeneracy]
\label{thm:nondeg}
Let $u$ be a minimizer of $\cI_{\Omega}$. Then, there is a constant $\kappa > 0$, depending only on $n,s,\lambda,\Lambda,M$, such that for any $R > 0$ and $x_0 \in \overline{ \{ u > 0 \} } $ with $B_{2R}(x_0) \subset \Omega$ it holds:
\begin{align*}
\Vert u \Vert_{L^{\infty}(B_R(x_0))} \ge \kappa R^s.
\end{align*}
\end{theorem}

The proof of \autoref{thm:nondeg} will be a consequence of the following weak non-degeneracy property and a well-known improving scheme (see \cite[Proposition 3.3]{CRS10}):

\begin{lemma}[Weak non-degeneracy]
\label{lemma:weak-nondeg}
Let $u$ be a minimizer of $\cI_{\Omega}$, $R > 0$, and $x_0 \in \R^n$ such that $B_{2R}(x_0) \subset \Omega$. Then, there exists $c > 0$, depending only on $n,s,\lambda,\Lambda$, such that for any $x \in B_R(x_0)$:
\begin{align}
\label{eq:non-deg-dist-est}
u(x) \ge c M \dist( x , \partial\{ u > 0\})^s.
\end{align}
\end{lemma}

\begin{proof}[Proof of \autoref{lemma:weak-nondeg}]
The proof is split into two steps.\\
Step 1: First, we claim that there is $c > 0$, depending only on $n,s,\lambda,\Lambda$, such that 
\begin{align}
\label{eq:non-deg-preprelim}
u > 0 ~~ \text{ in } B_{2R}(x_0) ~~ \Rightarrow ~~ \Vert u \Vert_{L^{\infty}(B_{2R}(x_0))} \ge c M R^s.
\end{align}

We will prove it, without loss of generality (up to translation) only for $x_0 = 0$. \\
Let us take $v$ to be the weak solution to
\begin{align*}
\begin{cases}
L v &= 0 ~~ \text{ in } B_{2R} \setminus B_R,\\
v &= u  ~~ \text{ in } \R^n \setminus B_{2R},\\
v &= 0 ~~ \text{ in } B_R.
\end{cases}
\end{align*}
For later purpose, let us compute $Lv$ for $x \in B_R$. It holds
\begin{align*}
Lv(x) = - \int_{\R^n \setminus B_R} v(y) K(x-y) \d y.
\end{align*}
To estimate this quantity, we split the integral into two parts. Note that as a consequence of \autoref{cor:annulus-regularity}, there exists $\beta \in (0, \min\{1 ,2s \})$ with $\beta > 2s - 1$ such that
\begin{align*}
\left\Vert \frac{v}{\dist(\cdot,B_R)^{\beta}} \right\Vert_{L^{\infty}(B_{2R} \setminus B_R)} \le c R^{-\beta} \left(\Vert v \Vert_{L^{\infty}(B_{2R})}  + \tail(u,2R) \right) \le c R^{-s} \Vert u \Vert_{L^{\infty}(B_{2R})},
\end{align*}
where we also used the tail estimate for $u$ (see \cite[Theorem 1.9]{KaWe23}) in the last step. This allows us to estimate for any $x \in B_R$
\begin{align}
\label{eq:Lv-estimate}
\begin{split}
|Lv(x)| &\le c R^{-s} \Vert u \Vert_{L^{\infty}(B_{2R})}\int_{B_{2R} \setminus B_R} \dist(y, B_R)^{\beta} K(x-y) \d y + c \int_{\R^n \setminus B_{2R}} |u(y)| K(x-y) \d y\\
&\le c R^{-s} \Vert u \Vert_{L^{\infty}(B_{2R})}\int_{B_{2R} \setminus B_R} |x-y|^{-n-2s+\beta} \d y + c R^{-2s} \tail(u;2R)\\
&\le c R^{-s} \Vert u \Vert_{L^{\infty}(B_{2R})} \dist(x,\partial B_R)^{-2s+\beta} + c R^{-2s} \Vert u \Vert_{L^{\infty}(B_{2R})}.
\end{split}
\end{align}
Thus, upon defining $f := Lv$ in $B_R$ and $f := 0$ in $B_{2R} \setminus B_R$, we observe that by the previous computation $f d_{B_R}^{\beta-2s} \in L^{\infty}(B_{R})$, so $f \in L^1(B_{2R})$. Therefore, due to \autoref{lemma:aux}(iv), and \cite[Lemma 2.2.32]{FeRo24}, $v$ and $u$ are distributional solutions to
\begin{align}
\label{eq:v-PDE}
\begin{cases}
Lv &= f ~~ \text{ in } B_{2R},\\
v &= u ~~ \text{ in } \R^n \setminus B_{2R},
\end{cases}
\qquad 
\begin{cases}
Lu &= 0 ~~ \text{ in } B_{2R},\\
u &= u ~~ \text{ in } \R^n \setminus B_{2R}.
\end{cases}
\end{align}
Keeping this information in mind, let us now start with the actual proof of \eqref{eq:non-deg-preprelim}.
Since $v$ is a competitor for $u$ in $B_{2R}$, we have
\begin{align*}
\cE_{B_R \times B_R}(u,u) + M |B_R| &\le \cE_{B_R \times B_R}(u,u) + M |\{ u > 0\} \cap B_{R}| \\
&\le \cE_{(B_{2R}^c \times B_{2R}^c)^c}(v,v) - \cE_{(B_{2R}^c \times B_{2R}^c)^c \setminus (B_R \times B_R)}(u,u) \\
&\quad + M \Big[|\{ v > 0\} \cap B_{2R}| - |\{ u > 0\} \cap (B_{2R} \setminus B_R)| \Big] \\
&\le \cE_{(B_{2R}^c \times B_{2R}^c)^c \setminus (B_R \times B_R)}(v,v) - \cE_{(B_{2R}^c \times B_{2R}^c)^c \setminus (B_R \times B_R)}(u,u) \\
&\le 2\cE_{(B_{2R}^c \times B_{2R}^c)^c \setminus (B_R \times B_R)} (v,v-u) = 2\cE_{(B_{2R}^c \times B_{2R}^c)^c} (v,v-u),
\end{align*}
where we used in the first estimate that $u > 0$ in $B_{R}$, in the third estimate that by assumption and construction:
\begin{align*}
\{v > 0\} \cap B_{2R} = \{v > 0\} \cap (B_{2R} \setminus B_R) \subset B_{2R} \setminus B_R = \{u > 0\} \cap (B_{2R} \setminus B_R),
\end{align*}
and then the algebraic estimate
\begin{align*}
(v(x) - v(y))^2 - (u(x) - u(y))^2 \le 2(v(x) - v(y))((v(x) - u(x))-(v(y) - u(y))).
\end{align*}
Next, we claim that the following identity holds true
\begin{align}
\label{eq:ibp-weak-sol-v}
2\cE_{(B_{2R}^c \times B_{2R}^c)^c} (v,v-u) = 4 \int_{B_{R}} \left(\int_{\R^n \setminus B_R} v(y) K(x-y) \d y \right) u(x) \d x.
\end{align}
To see it, note that due to the equation satisfied by $v$ and by $u$ (see \eqref{eq:v-PDE}) and the nonlocal integration by parts formula, for any $\phi \in C_c^{\infty}(B_{2R})$, it holds
\begin{align*}
2\cE_{(B_{2R}^c \times B_{2R}^c)^c} (v,\phi) &= 2\cE_{(B_{2R}^c \times B_{2R}^c)^c} (v-u,\phi) + 2\cE_{(B_{2R}^c \times B_{2R}^c)^c} (u,\phi) \\
&= 2\int_{B_{2R}} (v(x) - u(x)) L \phi(x) \d x = 2\int_{B_{2R}} f(x) \phi(x) \d x = 2\int_{B_{R}} Lv(x) \phi(x) \d x.
\end{align*} 
Since $Lv \in L^1(B_R)$ and $v-u \in L^{\infty}(B_R)$, an approximation argument allows us to replace $\phi$ by $v-u$ and thus to obtain
\begin{align*}
2\cE_{(B_{2R}^c \times B_{2R}^c)^c} (v,v-u) &= 2\int_{B_{R}} Lv(x) (v(x) - u(x) ) \d x \\
&= -2 \int_{B_R} Lv(x) u(x) \d x = 4 \int_{B_{R}} \left(\int_{\R^n \setminus B_R} v(y) K(x-y) \d y \right) u(x) \d x,
\end{align*}
as claimed in \eqref{eq:ibp-weak-sol-v}. Altogether, we have proved, recalling also \eqref{eq:Lv-estimate}
\begin{align*}
\cE_{B_R \times B_R}(u,u)& + M |B_R| \le 4 \int_{B_{R}} \left(\int_{\R^n \setminus B_R} v(y) K(x-y) \d y \right) u(x) \d x\\
&= 4\int_{B_R} \left(\int_{B_{2R} \setminus B_R} \hspace{-0.2cm} v(y) K(x-y) \d y \right) u(x) \d x + 4 \int_{B_{R}} \left(\int_{\R^n \setminus B_{2R}} \hspace{-0.2cm} v(y) K(x-y) \d y \right) u(x) \d x\\
&\le c R^{-\beta} \Vert u \Vert_{L^{\infty}(B_{2R})} \int_{B_R}  u(x) \dist(x,B_R)^{-2s+\beta} \d x + c \Vert u \Vert_{L^{\infty}(B_{2R})}^2 R^{n-2s} \\
&\le c \Vert u \Vert_{L^{\infty}(B_{2R})}^2 R^{n-2s}.
\end{align*}
Clearly, this estimate implies $c M R^s \le \Vert u \Vert_{L^{\infty}(B_{2R})}$, as desired. We have thus proved \eqref{eq:non-deg-preprelim}.

Step 2: Now, we prove that \eqref{eq:non-deg-preprelim} implies that for some $c > 0$, depending only on $n,s,\lambda,\Lambda$,
\begin{align}
\label{eq:non-deg-prelim}
u > 0 ~~ \text{ in } B_{2R}(x_0) ~~ \Rightarrow ~~ u(x_0) \ge c M R^s.
\end{align}
In fact, by \autoref{lemma:aux} it holds $L u = 0$ in $B_{2R}(x_0)$ and $Lu \le 0$ in $\Omega$, and therefore using the local boundedness estimate (see \cite[Theorem 6.2]{Coz17}) and the weak Harnack inequality with tail estimate (see \cite[Theorem 1.9]{KaWe23}),
\begin{align*}
\Vert u \Vert_{L^{\infty}(B_{2R}(x_0))} &\le c \dashint_{B_{4R}(x_0)} u \d x + c\tail(u;4R,x_0)\\
&\le c \dashint_{B_{R}(x_0)} u \d x + c\tail(u;R,x_0) \le c \inf_{B_R(x_0)} u \le c u(x_0).
\end{align*}
Combining this estimate with \eqref{eq:non-deg-preprelim} yields \eqref{eq:non-deg-prelim}.
Finally, note that \eqref{eq:non-deg-prelim} implies the desired result upon suitable rescaling.
\end{proof}

The following lemma allows to deduce some geometric growth of the solution away from the free boundary.

\begin{lemma}
\label{lemma:geometric-growth}
Let $u$ be a minimizer of $\cI_{\Omega}$ and assume that $y_0 \in \partial \{ u > 0\} \cap \Omega$ and $B_{|x_0 - y_0|}(x_0) \subset \{ u > 0\}$. Then, there exist $M,\eps >0$, depending only on $n,s,\lambda,\Lambda,M$, such that
\begin{align*}
\sup_{B_{M|x_0 - y_0|}(y_0)} u \ge (1+\eps) u(x_0).
\end{align*}
\end{lemma}

\begin{proof}
Without loss of generality, we prove the result for $y_0 = 0$, $x_0 = e_1$.
Assume by contradiction that there exist sequences $(K_k)_k$, $(u_k)_k$, $k \in \N$, such that $u_k$ is a minimizer of $\cI_{\Omega}$ with respect to $L_{K_k}$ in $B_k$ such that $0 \in \partial \{ u_k > 0\}$ and $B_1(e_1) \subset \{ u_k > 0\}$, but 
\begin{align*}
\sup_{B_k} u_k = u_k(e_1).
\end{align*}
By the optimal regularity (see \autoref{thm:or}), the sequence $v_k = u_k / u_k(e_1)$ is uniformly bounded in $C^s(B_{k/2})$ and therefore, by the Arzel\`a-Ascoli theorem, it converges locally uniformly to a function $v_{\infty} \in C(\R^n) \cap L^{\infty}(\R^n)$. In particular, it holds $0 \in \partial \{ v_{\infty} > 0\}$ and 
\begin{align*}
\sup_{\R^n} v_{\infty} = v_{\infty}(e_1) = 1.
\end{align*} 
However, by application of \cite[Theorem 2.2.36, Lemma 2.2.32]{FeRo24}, there exists a kernel $K_{\infty}$ satisfying \eqref{eq:Kcomp} such that $L_{K_{\infty}} v_{\infty} = 0$ in $B_1(e_1)$. Thus, by the strong maximum principle (see \cite[Theorem 2.4.15(b)]{FeRo24}), it must be $v_{\infty} \equiv 1$ in $B_1(e_1)$, which contradicts $0 \in \partial \{ v_{\infty} > 0\}$. Thus, the proof is complete.
\end{proof}

Now, we are in the position to conclude the proof of the non-degeneracy.

\begin{proof}[Proof of \autoref{thm:nondeg}]
Let us first assume that $x_0 \in \partial \{ u > 0 \}$.
We will prove the desired result only for $x_0 = 0$ and $R=M+1$, where $M > 1$ is the constant from \autoref{lemma:geometric-growth}. The full result then follows by scaling and translation.\\
Then, we construct a sequence of points $(x_k)_{k}$ for $k \in \{0,1,2,\dots,k_0 \}$, where $x_0 = 0$, and $x_{k+1} \in B_{M \dist(x_k,\{ u = 0\})}(y_k)$, for $y_k \in \partial \{ u > 0 \}$ being such that $|x_k - y_k| = \dist(x_k,\{ u = 0\})$, i.e., $y_k$ realizes the distance from $x_k$ to the free boundary. Thanks to \autoref{lemma:geometric-growth}, we can choose $x_{k+1}$ such that
\begin{align}
\label{eq:geom-growth}
u(x_{k+1}) \ge (1 + \eps) u(x_k)
\end{align}
for some $\eps > 0$.
Note that the construction implies
\begin{align}
\label{eq:sequence-dist}
|x_{k+1} - x_k| \le (M+1)\dist(x_k,\{ u = 0 \}).
\end{align}
Finally, we denote $k_0 \in \N$ to be the first integer for which $x_{k_0 + 1} \not\in B_1$. Clearly, such $k_0$ exists, since due to \eqref{eq:geom-growth}, the value $u(x_{k+1})$ has to become larger than one eventually, while due to the optimal regularity (see \autoref{lemma:contradiction}), we also have $u(x_{k+1}) \le c|x_{k+1}|^s$.\\
Having constructed the sequence $(x_k)$, we can now estimate using $u(x_0) = 0$, as well as \eqref{eq:geom-growth}, \eqref{eq:non-deg-dist-est}, and \eqref{eq:sequence-dist}:
\begin{align*}
u(x_{k_0+1}) &= \sum_{k = 0}^{k_0} (u(x_{k+1}) - u(x_k)) \ge \eps \sum_{k = 0}^{k_0} u(x_k)\\
&\ge c \sum_{k = 0}^{k_0} \dist(x_k , \{ u = 0 \})^s \ge c \sum_{k = 0}^{k_0} |x_{k+1} - x_k|^s\\
&\ge c \sum_{k = 0}^{k_0} \min\{1,|x_{k+1} - x_k|\} \ge c \min\{1,|x_{k_0+1}|\} \ge c
\end{align*}
for some $c > 0$ depending only on $n,s,\lambda,\Lambda,M, M$. Since $x_{k_0+1} \in B_{M+1}$ by construction, we have
\begin{align*}
\Vert u \Vert_{L^{\infty}(B_{M+1})} \ge c.
\end{align*}
Upon performing a suitable rescaling, this implies the desired result in case $x_0 \in \partial \{ u > 0 \}$. In case $x_0 \in \{ u > 0 \}$, we proceed as follows. If $\dist(x_0 , \partial \{ u > 0 \}) \ge R/2$, then the desired result follows from \autoref{lemma:weak-nondeg} applied with $x_0$. In case $\dist(x_0 , \partial \{ u > 0 \} ) < R/2$, we denote $y_0 \in \partial \{ u > 0\}$ to be the projection of $x_0$ onto the free boundary, i.e., it holds $|x_0 - y_0| = \dist(x_0 , \partial \{ u > 0 \} )$. We observe that $B_{R/2}(y_0) \subset B_R(x_0)$, and hence, the desired result follows by application of the first part of the proof to $y_0$ and the ball $B_{R/2}(y_0)$.
\end{proof}

\begin{proof}[Proof of the second part of \autoref{thm-onephase-intro}]
This is a direct consequence of \autoref{thm:nondeg}.
\end{proof}

\subsection{Density estimates for the free boundary}
\label{sec:density}

As a consequence of the $C^s$ regularity (see \autoref{thm:or}) and the non-degeneracy (see \autoref{thm:nondeg}), we deduce that the contact set is Ahlfors regular.

\begin{theorem}
\label{thm:density-est}
Let $u$ be a minimizer of $\cI_{\Omega}$. Then, for any $0 < r < 1$ and $x_0 \in \partial \{ u > 0 \}$ with $B_{2}(x_0) \subset \Omega$ it holds
\begin{align*}
0 < c_1 \le \frac{|\{ u > 0 \} \cap B_r(x_0)|}{|B_r(x_0)|} \le 1 - c_2 < 1,
\end{align*}
where $c_1 > 0$ and $c_2 > 0$ depend only on $n,s,\lambda,\Lambda,M$.
\end{theorem}

\begin{proof}
Without loss of generality, let $x_0 = 0$. Let us first prove the lower estimate. Indeed, for any $r \in (0,1)$ by \autoref{thm:nondeg}, there is $x \in B_r$ such that $u(x) \ge \kappa r^s$. Moreover, by \autoref{thm-onephase-intro}, it holds 
\begin{align*}
|u(x) - u(y)| \le c |x-y|^s ~~ \forall y \in B_1(x).
\end{align*}
Thus, $u(y) > 0$ as long as $|x-y| \le r \left(\frac{\kappa }{c} \right)^{1/s} =: rK,$
and therefore, 
\begin{align*}
\frac{|\{ u > 0 \} \cap B_r|}{|B_r|} \ge \frac{|B_{rK}(x)|}{|B_r|} = K^n \ge c_1.
\end{align*}

Let us now prove the upper estimate. Let $v$ be the $L$-harmonic replacement of $u$ in $B_r$ and deduce from \autoref{lemma:energy-comp-est} and the fractional Poincar\'e-Friedrich's inequality (see \cite[Lemma 4.7]{Coz17}):
\begin{align}
\label{eq:upper-est-prelim}
\begin{split}
\int_{B_r} |v-u| \d x &\le r^{\frac{n}{2}} \left(\int_{B_r} |v-u|^2 \d x\right)^{1/2} \\
&\le c r^{\frac{n}{2} + s} \cE_{B_r \times B_r}(u-v,u-v)^{1/2} \le c r^{\frac{n}{2} + s} |\{ u = 0 \} \cap B_r|^{1/2}.
\end{split}
\end{align}
Moreover, note since $Lu \le 0$ in $B_r$, we have $u \le v$ in $B_r$ by the weak maximum principle (see \cite[Corollary 2.3.4]{FeRo24}). Using this information, together with the non-degeneracy (see \autoref{thm:nondeg}) and the Harnack inequality for $v$ (see \cite[Theorem 6.9]{Coz17}), we obtain
\begin{align}
\label{eq:v-nondeg}
\kappa r^s \le \Vert u \Vert_{L^{\infty}(B_{r/2})} \le \Vert v \Vert_{L^{\infty}(B_{r/2})} \le c_1 \inf_{B_{r/2}} v.
\end{align}
Let us denote by $c_2>0$ the constant in \autoref{lemma:contradiction} and choose $\eps > 0$ so small that $\eps^s < \frac{\kappa}{2c_1c_2}$. Then, by application of the optimal regularity (see \autoref{lemma:contradiction}) in the ball $B_{\eps r}$ to $u$, we obtain from \eqref{eq:v-nondeg}:
\begin{align}
\label{eq:u-v-est}
\Vert u \Vert_{L^{\infty}(B_{\eps r})} \le c_2 (\eps r)^s \le \frac{\kappa}{2c_1} r^s \le \frac{1}{2}  \inf_{B_{r/2}} v.
\end{align}
Therefore, using \eqref{eq:upper-est-prelim}, \eqref{eq:v-nondeg}, and \eqref{eq:u-v-est}, we obtain
\begin{align*}
\frac{\kappa}{2c_1} r^s |B_{\eps r}| &\le \frac{|B_{\eps r}|}{2} \inf_{B_{r/2}} v \le |B_{\eps r}| \inf_{B_{\eps r}} \left( v - \frac{v}{2} \right) \le |B_{\eps r}| \inf_{B_{\eps r}} (v-u) \\
&\le \int_{B_{\eps r}} |v-u| \d x \le c r^{\frac{n}{2} + s} |\{ u = 0 \} \cap B_r|^{1/2},
\end{align*}
which implies the desired result.
\end{proof}

\subsection{Blow-ups}
\label{sec:blow-ups}

In this section we prove several basic properties of blow-ups, the main result being \autoref{cor:blowups}. In particular we prove that blow-ups are global minimizers of $\cI$. 

\begin{definition}[blow-ups]
Given a minimizer $u$ of $\cI_{\Omega}$, and a point $x_0 \in \partial \{ u > 0 \} \cap \Omega$ we define
\begin{align*}
u_{r,x_0}(x) = \frac{u(x_0 + r x)}{r^s} ~~ \forall x \in \R^n, \qquad \forall r > 0.
\end{align*}
The sequence of functions $(u_{r,x_0})_r$ is called blow-up sequence for $u$ at $x_0$. If there exists a sequence $r_k \searrow 0$ such that $u_{r_k,x_0} \to u_{x_0}$, as $k \to \infty$ for a function $u_{x_0} : \R^n \to \R$ locally uniformly, we say that $u_{x_0}$ is a blow-up (limit) of $u$ at $x_0$. If $x_0 = 0$, we sometimes write $u_{r} := u_{r,x_0}$.
\end{definition}

Note that by a scaling argument, if $u$ is a minimizer of $\cI_{\Omega}$ in $\Omega$, then $u_{r,x_0}$ is a minimizer of $\cI_{r^{-1}(\Omega-x_0)}$ in $r^{-1}(\Omega-x_0)$ with respect to the rescaled kernel $r^{n+2s} K(rh)$ which also satisfied \eqref{eq:Kcomp}.

The following three results establish the existence of (possibly non unique) blow-up limits for minimizers $u$ of $\cI$ at any free boundary point $x_0 \in \partial \{ u > 0 \}$. 
The first step, which is contained in the following lemma, is to show that the blow-up sequence $(u_{r,x_0})_r$ is uniformly bounded in $C^s_{loc}(\R^n)$, which implies by the Arzel\`a-Ascoli theorem that there exists a locally uniformly converging subsequence. Moreover, we establish its convergence in $H^s_{loc}(\R^n)$ and in $L^1_{2s}(\R^n)$. The lemma is written in a slightly more general framework, allowing for sequences of minimizers, belonging to different energy functionals. This setup will be important for the proof of \autoref{thm-1D}. 

\begin{lemma}
\label{lemma:scaling-blow-up}
Let $\Omega \Subset \Omega' \subset \R^n$. Let $x_0 \in \R^n$ be such that $B_2(x_0) \subset \Omega$, and $R > 0$. Let $(L_k)_k \subset \mathcal{L}_s^n(\lambda,\Lambda)$ be a sequence of operators with kernels $K_k$. Let $(u^{(k)})_k \subset V^s(\Omega| \Omega') \cap L^1_{2s}(\R^n)$ be minimizers of $\cI_{\Omega}$ with respect to $L_k$ such that $u^{(k)}(x_0) = 0$. Then, $u^{(k)}_{r,x_0}(x) := \frac{u^{(k)}(x_0 + rx)}{r^s}$ is uniformly bounded in $H^s(B_R) \cap L^1_{2s}(\R^n) \cap C^{s}(B_R)$ for any $r \le R^{-1}$ and $k \in \N$. Moreover it holds
\begin{align}
\label{eq:tail-vanishing}
|u^{(k)}_{r,x_0}(x)| \le c |x|^s ~~ \forall x \in B_R, \qquad \int_{\R^n \setminus B_{R}} |u^{(k)}_{r,x_0}(x)| |x|^{-n-2s} \d x \le c_1 R^{-s}
\end{align}
for some $c_1 > 0$, depending only on $n,s,\lambda,\Lambda$. In particular, there exists a subsequence $(r_k)_k$ with $r_k \searrow 0$, such that $u^{(k)}_{r_k,x_0} \to u_{x_0}$ locally uniformly, in $L^1_{2s}(\R^n)$, and weakly in $H^s(B_R)$ to some $u_{x_0} \in H^s(B_R) \cap L^1_{2s}(\R^n)$ for any $R > 0$, satisfying for some $c_2 > 0$, depending only on $n,s,\lambda,\Lambda$:
\begin{align}
\label{eq:blow-up-limit-growth}
|u_{x_0}(x)| \le c_1 |x|^s ~~ \forall x \in \R^n.
\end{align}
\end{lemma}

\begin{proof}
First, by  \autoref{lemma:energy-bound}, \autoref{thm-onephase-intro}, and \autoref{lemma:contradiction}, we deduce that
\begin{align*}
\Vert u^{(k)} \Vert_{C^s(B_1(x_0))} + \Vert u^{(k)} \Vert_{L^1_{2s}(\R^n)} \le C
\end{align*}
for some $C > 0$, depending only on $n,s,\lambda,\Lambda$. Moreover, for any $k \in \N$, $r \le R^{-1}$ and $x \in B_R$
\begin{align}
\label{eq:blow-up-growth}
|u^{(k)}_{r,x_0}(x)| \le \frac{|u^{(k)}(x_0 + rx) - u^{(k)}(x_0)|}{r^s} \le c |x|^s \Vert u^{(k)} \Vert_{C^s(B_{1}(x_0))} \le  c R^s  < \infty.
\end{align}
Moreover, by scaling and \autoref{lemma:energy-bound} and \autoref{thm:or} applied to $u^{(k)}$:
\begin{align*}
[u^{(k)}_{r,x_0}]^2_{H^s(B_R)} &= r^{-2s} \int_{B_R(x_0)} \int_{B_R(x_0)} \hspace{-0.4cm} \frac{(u^{(k)}(rx) - u^{(k)}(ry))^2}{|x-y|^{n+2s}} \d y \d x = r^{-n} [u^{(k)}]^2_{H^s(B_{rR}(x_0))} \\
&\le c r^{-n}(rR)^n \left(1 + (rR)^{-2s} \dashint_{B_{rR}(x_0)} |u^{(k)}(x)|^2 \d x \right) \le R^{n} (1 + \Vert u^{(k)} \Vert_{C^s(B_1(x_0))}) \le C R^n.
\end{align*}
Together, this proves the uniform boundedness of $(u^{(k)}_{r,x_0})_{r,k}$ in $H^s(B_R)$. Since $u^{(k)}_{r,x_0}$ is a minimizer of $\cI_{r^{-1}(\Omega - x_0)}$ with respect to the kernel $r^{n+2s}K(rh)$, which still satisfies \eqref{eq:Kcomp} with $\lambda,\Lambda$, and since $B_R \subset r^{-1}(\Omega - x_0)$, we can apply \autoref{thm:or} to $u^{(k)}_{r,x_0}$, and deduce from \eqref{eq:blow-up-growth} that $(u^{(k)}_{r,x_0})_r$ is also uniformly bounded in $C^s(B_{R})$. The existence of a locally uniformly converging subsequence $(u^{(k)}_{r_k,x_0})_k$ follows by the Arzel\`a-Ascoli theorem. Moreover, by weak compactness, up to a subsequence, $u^{(k)}_{r_k,x_0} \to u_{x_0} \in H^s(B_R)$.\\
Next, we prove boundedness of the tails. First, we observe that, by \eqref{eq:blow-up-growth} and the locally uniform convergence, it holds \eqref{eq:blow-up-limit-growth}, and thus, $u_{x_0} \in L^1_{2s}(\R^n)$. 
Next, we compute, using again \eqref{eq:blow-up-growth}
\begin{align*}
\int_{\R^n \setminus B_1} |u^{(k)}_{r,x_0}(x)| |x|^{-n-2s} \d x &\le  \int_{\R^n \setminus B_{r^{-1}}} |u^{(k)}_{r,x_0}(x)| |x|^{-n-2s} \d x  + c \Vert u^{(k)} \Vert_{C^s(B_{1}(x_0))} \int_{B_{r^{-1}} \setminus B_1} |x|^{-n-s} \d x \\
&\le \int_{\R^n \setminus B_1(x_0)} |u^{(k)}(x)||x-x_0|^{-n-2s} \d x + C \le \Vert u^{(k)} \Vert_{L^1_{2s}(\R^n)} + C \le C < \infty.
\end{align*}
Thus, $(u^{(k)}_{r,x_0})_{r,k}$ is uniformly bounded in $L^1_{2s}(\R^n)$.
Finally, by the tail estimate in \autoref{lemma:energy-bound} applied to $u^{(k)}_{r,x_0}$, which is applicable since $u^{(k)}_{r,x_0}$ is a minimizer in $B_R$, as was mentioned before:
\begin{align*}
\int_{\R^n \setminus B_{R}} |u^{(k)}_{r,x_0}(x)| |x|^{-n-2s} \d x &\le R^{-2s} \tail(u^{(k)}_{r,x_0};R) \le c R^{-s} + c R^{-2s} \left( \dashint_{B_R} |u^{(k)}_{r,x_0}(x)|^2 \d x \right)^{1/2} \\
&\le C R^{-s} (1 + \Vert u^{(k)} \Vert_{C^{s}(B_1(x_0))}) \le C R^{-2s} < \infty,
\end{align*}
where we used \eqref{eq:blow-up-growth}. This proves \eqref{eq:tail-vanishing}. Thus, given any $\eps > 0$, upon splitting
\begin{align*}
\int_{\R^n} & |u^{(k)}_{r_k,x_0}(x) - u_{x_0}|(1+|x|)^{-n-2s} \d x  \le  \int_{B_R} |u^{(k)}_{r_k,x_0}(x) - u_{x_0}|(1+|x|)^{-n-2s} \d x \\
& + 2\int_{\R^n \setminus B_R} |u^{(k)}_{r_k,x_0}(x)| |x|^{-n-2s} \d x + 2\int_{\R^n \setminus B_R} |u_{x_0}(x)| |x|^{-n-2s} \d x,
\end{align*}
we see that by the locally uniform convergence, the first summand becomes smaller than $\eps/3$ for any $R > 0$ as $k \to 0$. Moreover, by \eqref{eq:tail-vanishing}, the second term can be made smaller than $\eps/3$ by choosing $R > 1$ large enough, and the third term can be made smaller than $\eps/3$ since $u_{x_0} \in L^1_{2s}(\R^n)$. This proves the strong convergence $u^{(k)}_{r_k,x_0} \to u_{x_0}$ in $L^1_{2s}(\R^n)$. The proof is complete.
\end{proof}

The following lemma proves that for any sequence of minimizers to $\cI$ that is bounded in $H^s(B_R)$ and converges in $L^1_{2s}(\R^n)$, also the limit will be a minimizer to $\cI$. The main work consists in proving that blow-ups converge strongly in $H^s_{loc}(\R^n)$. To our surprise, we were not able to find such result in the literature, not even for linear nonlocal equations. Also the proof of the minimizing property for the limit is a lot more complicated than in the proof of \autoref{lemma:convergence}, due to the additional measure term.

\begin{lemma}
\label{lemma:convergence-minimizer}
Let $R > 0$. Let $(u_k)_k \subset V^s(B_R | B_{2R} ) \cap L^1_{2s}(\R^n)$ be a sequence of functions with $u_k \ge 0$ in $\R^n$ that is uniformly bounded in $H^s(B_R) \cap L^1_{2s}(\R^n)$, i.e.,
\begin{align}
\label{eq:unif-Vs-bd}
[u_k]_{H^s(B_R)} + \Vert u_k \Vert_{L^2(B_R)} + \Vert u_k \Vert_{L^1_{2s}(\R^n)} \le C
\end{align} 
for some $C > 0$. Moreover, assume that $u_k \to u_{\infty}$ in $L^1_{2s}(\R^n)$, and that for every $M > R$ there is $k_M \in \N$ such that
\begin{align}
\label{eq:unif-tail-decay}
|u_k| \le C M^s ~~ \text{ in } B_M, \qquad \int_{\R^n \setminus B_{M}} |u_k(x)| |x|^{-n-2s} \d x \le C M^{-s} \qquad \forall k \ge k_m.
\end{align}
Moreover, let $(L_k)_k \subset \mathcal{L}_s^n(\lambda,\Lambda)$ be a sequence of operators with kernels $K_k$. Assume that $u_k$ is a minimizer of $\cI_{B_R}$ with respect to $K_k$,
\begin{align}
\label{eq:min-functional-approx}
\cI^{K_k}_{B_R}(u_k) \le \cI^{K_k}_{B_R}(u_k + \phi) ~~ \forall \phi \in H_{B_R}^s(\R^n).
\end{align}
Then, there exists $L_{\infty} \in \mathcal{L}_s^n(\lambda,\Lambda)$ with kernel $K_{\infty}$, such that weakly in the sense of measures \\
$\min\{1,|h|^2\} K_k(h) \d h \to \min\{1,|h|^2\} K_{\infty}(h) \d h$, and $u_{\infty} \in H^s(B_R) \cap L^1_{2s}(\R^n)$ is a minimizer of $\cI_{B_R}$, 
\begin{align}
\label{eq:min-functional}
\cI^{K_{\infty}}_{B_R}(u_{\infty}) \le \cI^{K_{\infty}}_{B_R}(u_{\infty} + \phi) ~~ \forall \phi \in H_{B_R}^s(\R^n).
\end{align}
Moreover, $u_k \to u_{\infty}$ strongly in $H^s(B_r)$, a.e. in $B_R$, and $\1_{\{ u_k > 0\} } \to \1_{\{ u_{\infty} > 0\} }$ strongly in $L^1(B_r)$, pointwise a.e. in $B_r$, up to a subsequence, for any $0 < r < R$.
\end{lemma}

\begin{remark}
In case the quantities in \eqref{eq:min-functional-approx} or \eqref{eq:min-functional} are not well-defined, one has to interpret them in the spirit of \autoref{def:minimizer}.
\end{remark}

\begin{proof}
It is well-known that up to a subsequence, $u_k \to u_{\infty}$ weakly in $ H^s(B_R)$ for some $u_{\infty} \in H^s(B_R)$, strongly in $L^2(B_R)$ (see  \cite[Theorem 3.10]{FoKa24}), and therefore also a.e. in $B_R$ for any $R > 0$.\\
To see $u_{\infty} \ge 0$ in $\R^n$, we use that by the strong convergence in $L^1_{2s}(\R^n)$, for any $\phi \in L^{\infty}(\R^n)$ with $\phi \ge 0$,
\begin{align*}
0 \le \int_{\R^n} u_k(y) \phi(y) (1 + |y|)^{-n-2s} \d y \to  \int_{\R^n} u(y) \phi(y) (1 + |y|)^{-n-2s} \d y,
\end{align*}
which implies that $u \ge 0$, as desired.

\textbf{Step 1:} We will show that up to a subsequence, $u_k \to u_{\infty}$ in $H^s(B_r)$ and $\1_{\{ u_k > 0\}} \to \1_{\{ u_{\infty} > 0\}}$ in $L^1(B_r)$ for any $0 < r < R$.\\
To do so, we first observe that by the weak convergence in $H^s(B_r)$, it holds for every $0 < r < R$:
\begin{align*}
\cE_{B_r \times B_r}(u_{\infty},u_{\infty}) \le \liminf_{k \to \infty} \cE_{B_r \times B_r}(u_k,u_k),
\end{align*}
where, throughout this step of the proof, we will denote by $\cE$ the energy associated with the kernel $K(x-y) = |x-y|^{-n-2s}$.
Moreover, since $u_k \to u_{\infty}$ a.e. in $B_R$, we have for almost every $x \in B_R \cap \{ u_{\infty} > 0 \}$ that $x \in \{ u_k > 0 \} \cap B_R$ for large enough $k$, and therefore by Fatou's lemma, for any $0 < r < R$:
\begin{align*}
|\{ u_{\infty} > 0\} \cap B_r| \le \liminf_{k \to \infty} |\{ u_k > 0 \} \cap B_r|.
\end{align*}
Thus, in order to prove both, $u_k \to u_{\infty}$ in $H^s(B_r)$,  $\1_{\{ u_k > 0\}} \to \1_{\{ u_{\infty} > 0\}}$ in $L^1(B_r)$, it remains to show
\begin{align}
\label{eq:uktouinfty-energy}
\liminf_{k \to \infty} \cE_{B_r \times B_r}(u_k,u_k) &= \cE_{B_r \times B_r}(u_{\infty},u_{\infty}),\\
\label{eq:uktouinfty-level-set}
\liminf_{k \to \infty} |\{ u_{k} > 0\} \cap B_r| &= |\{ u_{\infty} > 0\} \cap B_r|.
\end{align}
Let us now  fix $0 < r < R$. The goal of this step in proof will be to verify \eqref{eq:uktouinfty-energy} and \eqref{eq:uktouinfty-level-set} by using the minimizing property of $(u_k)$, namely \eqref{eq:min-functional}, to deduce
\begin{align}
\label{eq:uktouinfty-both}
\liminf_{k \to \infty} \Big[\cE_{B_r \times B_r}(u_k,u_k) + |\{ u_{k} > 0\} \cap B_r| \Big] = \cE_{B_r \times B_r}(u_{\infty},u_{\infty}) + |\{ u_{\infty} > 0\} \cap B_r|.
\end{align}
Note that this implies both, \eqref{eq:uktouinfty-energy} and \eqref{eq:uktouinfty-level-set} due to the aforementioned applications of Fatou's lemma.

We take $\eta \in C^{\infty}(B_R)$ such that $0 \le \eta \le 1$, $\eta \equiv 0$ in $B_r$, and $\eta \equiv 1$ in $\R^n \setminus B_{R-\eps}$ for some small, arbitrary $\eps > 0$. Moreover, we set $\tilde{u}_k = \eta u_k + (1-\eta) u_{\infty}$. First, we observe that $\tilde{u}_k - u_k \in H^s_{B_{R-\eps}}(\R^n)$, and that therefore by \eqref{eq:min-functional-approx}
\begin{align}
\label{eq:conv-help1}
0 \le \limsup_{k \to \infty} \left[ \cE_{(B_R^c \times B_R^c)^c}(\tilde{u}_k,\tilde{u}_k) - \cE_{(B_R^c \times B_R^c)^c}(u_k,u_k) + |\{ \tilde{u}_k > 0 \} \cap B_R| - | \{ u_k > 0 \} \cap B_R| \right].
\end{align}

Note that by the same arguments as in \cite[Proof of (6.6)]{Vel23}, we have
\begin{align}
\label{eq:conv-help-level-set1}
\begin{split}
\limsup_{k \to \infty} \left( |\{ \tilde{u}_k > 0 \} \cap B_R| - | \{ u_k > 0 \} \cap B_R| \right) &\le \limsup_{k \to \infty} \left( |\{ u_{\infty} > 0 \} \cap B_r| - | \{ u_k > 0 \} \cap B_r| \right)\\
&\quad - |\{ \eta  > 0 \} \cap B_R|.
\end{split}
\end{align}

Next, we observe the following algebraic identity:
\begin{align*}
\tilde{u}_k(x) - \tilde{u}_k(y) &= \frac{1}{2} \left[(u_k-u_{\infty})(x) + (u_k-u_{\infty})(y) \right](\eta(x) - \eta(y)) \\
&\quad + \frac{1}{2}(u_k(x) - u_k(y))(\eta(x) + \eta(y)) + (u_{\infty}(x) - u_{\infty}(y))\left\{ 1- \frac{\eta(x) + \eta(y)}{2}\right\}.
\end{align*}
By the strong $L^2$-convergence, and the weak convergence in $H^s(B_R)$, we deduce
\begin{align}
\label{eq:conv-help3}
\begin{split}
&\limsup_{k \to \infty} \left[ \cE_{B_R \times B_R}(\tilde{u}_k,\tilde{u}_k) - \cE_{B_R \times B_R}(u_k,u_k) \right] \\
&= \limsup_{k \to \infty} \iint_{B_R \times B_R } \Bigg[ \left(\frac{1}{2}(u_k(x) - u_k(y))(\eta(x) + \eta(y)) + (u_{\infty}(x) - u_{\infty}(y))\left\{ 1- \frac{\eta(x) + \eta(y)}{2}\right\} \right)^2 \\
&\qquad\qquad\qquad\qquad\qquad\qquad\qquad\qquad\qquad\qquad - (u_k(x) - u_k(y))^2\Bigg] K(x-y) \d y \d x \\
&= \limsup_{k \to \infty} \iint_{B_R \times B_R } \Bigg[
(u_k(x) - u_k(y))^2 \left\{\left(\frac{\eta(x) + \eta(y)}{2}\right)^2 - 1 \right\} \\
&\qquad\qquad\qquad + (u_{\infty}(x) - u_{\infty}(y))^2\left\{ 1 - \frac{\eta(x) + \eta(y)}{2} \right\}^2 \\
 &\qquad\qquad\qquad + (u_k(x) - u_k(y))(u_{\infty}(x) - u_{\infty}(y)) \left\{(\eta(x) +\eta(y))\left(1 - \frac{\eta(x) + \eta(y)}{2} \right)\right\}
\Bigg] K(x-y) \d y \d x \\
&= \limsup_{k \to \infty} \iint_{B_R \times B_R }
 \left\{1 - \left(\frac{\eta(x) + \eta(y)}{2}\right)^2 \right\} \left((u_{\infty}(x) - u_{\infty}(y))^2 - (u_k(x) - u_k(y))^2 \right) K(x-y) \d y \d x \\
 &\le  \limsup_{k \to \infty}  \iint_{\{ \eta = 0\} \times \{ \eta = 0 \}} \left((u_{\infty}(x) - u_{\infty}(y))^2 - (u_k(x) - u_k(y))^2 \right) K(x-y) \d y \d x \\
 &\quad + \iint_{(B_R \times B_R ) \setminus (\{ \eta = 0\} \times \{ \eta = 0 \})} (u_{\infty}(x) - u_{\infty}(y))^2 K(x-y) \d y \d x.
\end{split}
\end{align}
Moreover, note that by \eqref{eq:unif-Vs-bd}, we also have $u_k \to u_{\infty}$ weakly in $V^s( (\{ \eta = 0 \} \setminus B_r) |\{\eta = 0\})$, so it must also hold 
\begin{align*}
\cE_{ ( \{\eta = 0\} \times \{ \eta = 0\} ) \setminus (B_r \times B_r) }(u_{\infty},u_{\infty}) \le \liminf_{k \to \infty} \cE_{ ( \{\eta = 0\} \times \{ \eta = 0\} ) \setminus (B_r \times B_r)}(u_k,u_k).
\end{align*}

Therefore, 
\begin{align}
\label{eq:conv-help4}
\begin{split}
\limsup_{k \to \infty} & \iint_{\{ \eta = 0\} \times \{ \eta = 0 \}} \left((u_{\infty}(x) - u_{\infty}(y))^2 - (u_k(x) - u_k(y))^2 \right) K(x-y) \d y \d x \\
&\le \limsup_{k \to \infty}  \iint_{ (\{\eta = 0 \} \times \{ \eta = 0\}) \setminus (B_r \times B_r) } \left((u_{\infty}(x) - u_{\infty}(y))^2 - (u_k(x) - u_k(y))^2 \right) K(x-y) \d y \d x \\
&\quad + \limsup_{k \to \infty}  \iint_{ B_r \times B_r } \left((u_{\infty}(x) - u_{\infty}(y))^2 - (u_k(x) - u_k(y))^2 \right) K(x-y) \d y \d x \\
&\le \limsup_{k \to \infty}  \iint_{ B_r \times B_r } \left((u_{\infty}(x) - u_{\infty}(y))^2 - (u_k(x) - u_k(y))^2 \right) K(x-y) \d y \d x
\end{split}
\end{align}
Altogether, by combination of \eqref{eq:conv-help1}, 
\eqref{eq:conv-help3}, \eqref{eq:conv-help4}, and \eqref{eq:conv-help-level-set1} we obtain:
\begin{align}
\label{eq:conv-help2}
\begin{split}
0 &\le \limsup_{k \to \infty}  \iint_{ B_r \times B_r } \left((u_{\infty}(x) - u_{\infty}(y))^2 - (u_k(x) - u_k(y))^2 \right) K(x-y) \d y \d x \\
&\quad + \iint_{(B_R \times B_R ) \setminus (\{ \eta = 0\} \times \{ \eta = 0 \})} (u_{\infty}(x) - u_{\infty}(y))^2 K(x-y) \d y \d x \\
&\quad + \limsup_{k \to \infty} 2 \int_{B_R} \int_{\R^n \setminus B_R} \left( ( \tilde{u}_k(x) - u_k(y) )^2 - ( u_k(x) - u_k(y) )^2 \right) K(x-y) \d y \d x\\
&\quad + \limsup_{k \to \infty} \left( |\{ u_{\infty} > 0 \} \cap B_r| - | \{ u_k > 0 \} \cap B_r| \right) - |\{ \eta  > 0 \} \cap B_R|,
\end{split}
\end{align}

Moreover, note that due to the fact that $\tilde{u}_k - u_k \in H^s_{B_{R-\eps}}(\R^n)$, we have
\begin{align*}
\int_{B_R } & \int_{\R^n \setminus B_R} \left( ( \tilde{u}_k(x) - u_k(y) )^2 - ( u_k(x) - u_k(y) )^2 \right) K(x-y) \d y \d x \\
&= \int_{B_R } \int_{\R^n \setminus B_R} \left( \tilde{u}_k^2(x) - u_k^2(x) - 2 u_k(y)(\tilde{u}_k(x) - u_k(x)) \right) K(x-y) \d y \d x \\
&= \int_{B_{R-\eps}}  \left( \tilde{u}_k^2(x) - u_k^2(x) \right)\left(\int_{\R^n \setminus B_R} K(x-y) \d y \right) \d x \\
&\quad -2 \int_{B_{R-\eps} }  \left(\tilde{u}_k(x) - u_k(x) \right) \left( \int_{\R^n \setminus B_R}  u_k(y) K(x-y) \d y \right) \d x  \\
&= I_1 + I_2.
\end{align*}
Note that $I_1 + I_2 \to 0$, as $k \to \infty$, since 
\begin{align*}
\sup_{x \in B_{R-\eps}}  \left(\int_{\R^n \setminus B_R} K(x-y) \d y \right) + \sup_{x \in B_{R-\eps}}\left( \int_{\R^n \setminus B_R}  u_k(y) K(x-y) \d y \right) \le C(\eps),
\end{align*}
due to the boundedness of $(u_k)$ in $L^1_{2s}(\R^n)$, and since $u_k \to u_{\infty}$ and $\tilde{u}_k \to u_{\infty}$ in $L^2(B_R)$.
This proves
\begin{align*}
\limsup_{k \to \infty} \int_{B_R } \int_{\R^n \setminus B_R(x_0)} \left( ( \tilde{u}_k(x) - u_k(y) )^2 - ( u_k(x) - u_k(y) )^2 \right) K(x-y) \d y \d x = 0.
\end{align*}

Thus, by combination with \eqref{eq:conv-help2}:
\begin{align*}
\liminf_{k \to \infty} & \left[ \cE_{ B_r \times B_r }(u_k,u_k) + | \{ u_k > 0 \} \cap B_r| \right]\\
&\le \cE_{ B_r \times B_r }(u_{\infty},u_{\infty}) + \cE_{(B_R \times B_R) \setminus (\{\eta = 0 \} \times \{\eta = 0\} )}(u_{\infty},u_{\infty}) \\
&\quad + |\{ u_{\infty} > 0 \} \cap B_r| - |\{ \eta  > 0 \} \cap B_R|.
\end{align*}
Since $\eta$ was not fixed, we can choose a sequence of functions $\eta$ such that $\{ \eta \equiv 0 \} \to B_R$, which causes the previous estimate to remain true without the second and fourth term on the right hand side, namely it holds
\begin{align*}
\liminf_{k \to \infty} \left[ \cE_{ B_r \times B_r }(u_k,u_k)  + | \{ u_k > 0 \} \cap B_r| \right] &\le \cE_{ B_r \times B_r }(u_{\infty},u_{\infty}) + | \{ u_{\infty} > 0 \} \cap B_r|.
\end{align*}
This implies \eqref{eq:uktouinfty-both}, and, as discussed before,
also yields $u_k \to u_{\infty}$ in $H^s(B_r)$ and $\1_{\{ u_k > 0\}} \to \1_{\{ u_{\infty} > 0\}}$ in $L^1(B_r)$, concluding Step 1.\\

\textbf{Step 2:} Throughout this part of the proof, we use the notation $\cI^k := \cI^{K_k}$ and $\cE^k := \cE^{K_k}$. By

Let us now show that $u_{\infty}$ is a local minimizer of $\cI^{\infty}$ in $B_R$. To do do, note that by density it suffices to consider $\phi \in H^s_{B_r}(\R^n)$ in \eqref{eq:min-functional}. Let $\eta$ be as before, i.e., satisfying $0 \le \eta \le 1$, $\eta \equiv 1$ in $\R^n \setminus B_{R-\eps}$, and $\eta \equiv 0$ in $B_r$. Moreover, we set $\cN = \{ \eta < 1 \}$, $v_k = \eta u_k + (1-\eta)u_{\infty} + \phi$, and $v_{\infty} = u_{\infty} + \phi$.\\
Note that since $\phi = 0$ in $(\cN \setminus B_R) \cup \cN^c$, it holds
\begin{align*}
|\{ u_{\infty} + \phi > 0 \} \cap (B_R \setminus \cN)| - |\{ u_{\infty} > 0 \} \cap (B_R \setminus \cN)| &= 0,\\
\cE^{\infty}_{(\cN^c \times \cN^c)^c \setminus (B_R^c \times B_R^c)^c}(u_{\infty} + \phi , u_{\infty} + \phi) - \cE^{\infty}_{(\cN^c \times \cN^c)^c \setminus (B_R^c \times B_R^c)^c}(u_{\infty} , u_{\infty}) &= 0,
\end{align*}
where we used for the second identity: $(\cN^c \times \cN^c)^c \setminus (B_R^c \times B_R^c)^c = (\cN \setminus B_R) \times (\cN \setminus B_R) \cup (\cN \setminus B_R) \times \cN^c \cup \cN^c \times (\cN \setminus B_R)$.
As a consequence, the desired result \eqref{eq:min-functional} is equivalent to 
\begin{align}
\label{eq:min-functional-equiv}
\cI^{\infty}_{\cN}(u_{\infty}) \le \cI^{\infty}_{\cN}(u_{\infty} + \phi) ~~ \forall \phi \in H^s_{B_r}(\R^n).
\end{align}
Consequently, it suffices to prove \eqref{eq:min-functional-equiv}. Let us fix $\phi \in H^s_{B_r}(\R^n)$, and denote $\phi_k = (u_{\infty} - u_k)(1-\eta) + \phi \in H_{\cN}^s(\R^n)$. Note that \eqref{eq:min-functional-approx} implies 
\begin{align}
\label{eq:min-functional-approx-equiv-equiv}
|\{ u_{k} > 0 \} \cap \cN| \le 2 \cE^k_{(\cN^c \times \cN^c)^c}(u_{k},\phi_k) + \cE^k_{(\cN^c \times \cN^c)^c}(\phi_k,\phi_k) + |\{ u_{k} + \phi_k > 0 \} \cap \cN |.
\end{align}
We will use this line to verify the following property,
\begin{align}
\label{eq:min-functional-equiv-equiv}
|\{ u_{\infty} > 0 \} \cap \cN| \le 2 \cE^{\infty}_{(\cN^c \times \cN^c)^c}(u_{\infty},\phi) + \cE^{\infty}_{(\cN^c \times \cN^c)^c}(\phi,\phi) + |\{ u_{\infty} + \phi > 0 \} \cap \cN |.
\end{align}
Note that since $\phi \in H^s_{B_r}(\R^n)$ is arbitrary, verifying \eqref{eq:min-functional-equiv-equiv} is sufficient to conclude the proof.\\
Choosing $\eta$ to be radial, indeed \eqref{eq:uktouinfty-level-set} implies that up to a subsequence
\begin{align}
\label{eq:min-conv-min-1}
|\{ u_{k} > 0 \} \cap \cN| \to |\{ u_{\infty} > 0 \} \cap \cN|.
\end{align}
Moreover, since $v_k = v_{\infty}$ on $\{ \eta = 0 \}$ by construction, we have
\begin{align}
\label{eq:min-conv-min-2}
\begin{split}
|\{ u_k + \phi_k > 0 \} \cap \cN| &= |\{ v_k > 0 \} \cap \cN| \\
&= |\{ v_k > 0 \} \cap \{ \eta = 0\}| + |\{ v_k > 0 \} \cap \{ 0 < \eta < 1\}| \\
&\le |\{ v_{\infty} > 0 \} \cap \cN| + |\{ 0 < \eta < 1\}|. 
\end{split}
\end{align}
Next, we show
\begin{align}
\label{eq:min-conv-min-3}
\cE^k_{(\cN^c \times \cN^c)^c}(u_{k},\phi_k) &\to \cE^{\infty}_{(\cN^c \times \cN^c)^c}(u_{\infty},\phi),\\
\label{eq:min-conv-min-4}
\cE^k_{(\cN^c \times \cN^c)^c}(\phi_k,\phi_k) &\to \cE^{\infty}_{(\cN^c \times \cN^c)^c}(\phi,\phi).
\end{align}
To see \eqref{eq:min-conv-min-3}, we first observe that 
\begin{align}
\label{eq:min-conv-min-3-split}
\cE^k_{(\cN^c \times \cN^c)^c}(u_{k},\phi_k) = \cE^k_{(\cN^c \times \cN^c)^c}(u_{k},(u_{\infty} - u_k)(1-\eta)) + \cE^k_{(\cN^c \times \cN^c)^c}(u_{k},\phi),
\end{align}
We claim that the first summand in \eqref{eq:min-conv-min-3-split} vanishes in the limit. To prove it, we observe that by the strong convergence in $H^s_{loc}(B_R)$, and the uniform bound \eqref{eq:unif-Vs-bd}, we have
\begin{align}
\label{eq:strong-conv-N}
\begin{split}
\cE^k_{\cN \times \cN}&(u_{k},(u_{\infty} - u_k)(1-\eta)) \\
&\le \Lambda [u_{k}]_{H^s(\cN)} [(u_{\infty} - u_k)(1-\eta)]_{H^s(\cN)}\\
&\le C \Lambda \Vert 1- \eta \Vert_{L^{\infty}(\cN)} [u_{\infty} - u_k]_{H^s(\cN)} + c \Vert u_{\infty} - u_k \Vert_{L^2(\cN)} \Vert \eta \Vert_{C^{1}(\cN)} \to 0.
\end{split}
\end{align}
Moreover, we have
\begin{align*}
\int_{\cN} \int_{\R^n \setminus \cN} & (u_{\infty}(x) - u_k(x))(1-\eta(x)) (u_k(x) - u_k(y)) K_k(x-y) \d y \d x\\
&\le \int_{\cN} \int_{B_R}(u_{\infty}(x) - u_k(x))(1-\eta(x)) (u_k(x) - u_k(y)) K_k(x-y) \d y \d x \\
&\quad + \int_{\cN} \int_{\R^n \setminus B_R} (u_{\infty}(x) - u_k(x))(1-\eta(x)) (u_k(x) - u_k(y)) K_k(x-y) \d y \d x \\
&= J_1 + J_2.
\end{align*}
For $J_1$, using again \eqref{eq:unif-Vs-bd}, Cauchy-Schwarz, and also the strong $L^2(B_R)$-convergence $u_k \to u_{\infty}$, we obtain
\begin{align*}
J_1 &\le \Lambda [u_k]_{H^s(B_R)} \left(\int_{\cN}(u_{\infty}(x) - u_k(x))^2(1-\eta(x))^2 \int_{B_R} |x-y|^{-n-2s} \d y \d x \right)^{1/2} \\
&\le C \left(\int_{\cN}(u_{\infty}(x) - u_k(x))^2 \left[ \int_{B_R} (\eta(y)-\eta(x))^2 |x-y|^{-n-2s} \d y \right] \d x \right)^{1/2}\\
&\le C(\eta) \left(\int_{\cN}(u_{\infty}(x) - u_k(x))^2 \d x \right)^{1/2} \to 0.
\end{align*}
For $J_2$, we use that $\cN \subset B_{R-\eps}$, and therefore by the strong $L^2(B_R)$-convergence $u_k \to u_{\infty}$ and the boundedness of the sequence $(u_k)$ in $L^1_{2s}(\R^n)$, we get
\begin{align*}
J_2 &\le \Vert u_{\infty} - u_k \Vert_{L^2(\cN)} \left(\int_{B_{R-\eps}} \left( \int_{\R^n \setminus B_R} (\eta(y)-\eta(x)) (u_k(x) - u_k(y)) K_k(x-y) \d y \right)^2 \d x \right)^{1/2} \\
&\le C\Vert u_{\infty} - u_k \Vert_{L^2(\cN)} \left( \int_{B_{R-\eps}} u_k(x) \left[ \int_{\R^n \setminus B_R} K(x-y) \d y \right] + \left[  \int_{\R^n \setminus B_R} u_k(y) K(x-y) \d y \right] \d x \right)^{1/2} \\
&\le C(\eps) \Vert u_{\infty} - u_k \Vert_{L^2(\cN)} \Vert u_k \Vert_{L^1_{2s}(\R^n)}^{1/2} \to 0.
\end{align*}
Altogether, we have shown that the first summand in \eqref{eq:min-conv-min-3-split} vanishes in the limit. For the second summand, we have that 
\begin{align}
\label{eq:weak-conv-energy-2}
\cE^k_{(\cN^c \times \cN^c)^c}(u_{k},\phi) \to \cE^{\infty}_{(\cN^c \times \cN^c)^c}(u_{\infty},\phi).
\end{align}
The proof of this result is the exact same as the proof of \eqref{eq:energy-convergence-weak} in the proof of \autoref{lemma:convergence}, so we skip the details.
Altogether, this proves \eqref{eq:min-conv-min-3}.
Next, we check \eqref{eq:min-conv-min-4}. To do so, note that
\begin{align}
\label{eq:min-conv-min-4-split}
\begin{split}
\cE^k_{(\cN^c \times \cN^c)^c}(\phi_k,\phi_k) &= \cE^k_{(\cN^c \times \cN^c)^c}((u_{\infty} - u_k)(1-\eta),(u_{\infty} - u_k)(1-\eta))\\
&\quad + 2\cE^k_{(\cN^c \times \cN^c)^c}((u_{\infty} - u_k)(1-\eta),\phi)\\
&\quad + \cE^k_{(\cN^c \times \cN^c)^c}(\phi,\phi).
\end{split}
\end{align}
For the first summand in \eqref{eq:min-conv-min-4-split}, we proceed in a similar way as for \eqref{eq:min-conv-min-3}. Indeed, it holds
\begin{align*}
\cE^k_{\cN \times \cN}((u_{\infty} - u_k)(1-\eta),(u_{\infty} - u_k)(1-\eta)) \to 0,
\end{align*}
due to the strong $H^s_{loc}(B_R)$ convergence, mimicking the computation in \eqref{eq:strong-conv-N}. Moreover, by \eqref{eq:unif-Vs-bd} and the strong $L^2(B_R)$-convergence, we deduce
\begin{align}
\label{eq:strong-conv-Nc}
\begin{split}
\int_{\cN} \int_{\R^n \setminus \cN} & [(u_{\infty}(x) - u_k(x))(1-\eta(x))]^2 K_k(x-y) \d y \d x \\
&= \int_{\cN}(u_{\infty}(x) - u_k(x))^2 \left[ \int_{\R^n \setminus \cN} (\eta(y)-\eta(x))^2 K_k(x-y) \d y \right] \d x \\
&\le C(\eta) \int_{\cN}(u_{\infty}(x) - u_k(x))^2 \d x \to 0.
\end{split}
\end{align}
For the second summand in \eqref{eq:min-conv-min-4-split}, we observe that $\cE^k_{(\cN^c \times \cN^c)^c}((u_{\infty} - u_k)(1-\eta),\phi) \to 0$ since 
\begin{align*}
[(u_{\infty} - u_k)(1-\eta)]_{V^s(\cN|\R^n)} \to 0,
\end{align*}
as a consequence of the computations in \eqref{eq:strong-conv-N} and \eqref{eq:strong-conv-Nc}. Since the convergence of the third summand in \eqref{eq:min-conv-min-4-split} follows immediately by the convergence of the kernels, we have shown \eqref{eq:min-conv-min-4}.

Combining \eqref{eq:min-conv-min-1}, \eqref{eq:min-conv-min-2}, \eqref{eq:min-conv-min-3}, \eqref{eq:min-conv-min-4}, and \eqref{eq:min-functional-approx-equiv-equiv}, we obtain
\begin{align*}
|\{ u_{\infty} > 0 \} \cap \cN| \le 2\cE^{\infty}_{(\cN^c \times \cN^c)^c}(u_{\infty},\phi) + \cE^{\infty}_{(\cN^c \times \cN^c)^c}(\phi,\phi) +  |\{ v_{\infty} > 0 \} \cap \cN| + |\{ 0 < \eta < 1\}|,
\end{align*}
and since $\eta$ was not fixed, we can approximate $\1_{B_r^c}$ by functions $\eta$ satisfying all the aforementioned properties, making the last term on the right hand side vanish. This proves \eqref{eq:min-functional-equiv-equiv}, and therefore also verifies \eqref{eq:min-functional}, as desired. This concludes the proof.
\end{proof}

As an application of the previous compactness result, we have the following properties of blow-up sequences and their limits.

\begin{corollary}
\label{cor:blowups}
Let $\Omega \subset \R^n$. Let $u$ be a minimizer of $\cI_{\Omega}$ with $B_2 \subset \Omega$ and $x_0 \in \partial \{ u > 0 \} \cap B_1$. Then, there exists a subsequence $(r_k)_k$ with $r_k \searrow 0$ such that $u_{r_k,x_0} \to u_{x_0}$, as $k \to \infty$, locally uniformly. Moreover, for any such $(r_k)_k$, we have:
\begin{itemize}
\item[(i)] $u_{x_0}$ is a non-trivial minimizer of $\cI$ in $\R^n$, i.e., $u_{x_0}$ is a minimizer of $\cI_{B_R}$ for any $R > 0$.
\item[(ii)] Up to a subsequence, $u_{r_k,x_0} \to u_{x_0}$ in $H^s(B_R)$, in $L^1_{2s}(\R^n)$, and pointwise a.e. in $B_R$ for any $R > 0$.
\item[(iii)] Up to a subsequence, $\1_{\{ u_{r_k,x_0} > 0 \}} \to \1_{u_{x_0} > 0 }$ strongly in $L^1(B_R)$ and pointwise a.e. in $B_R$ for any $R > 0$.
\item[(iv)] Up to a subsequence, $\overline{\{ u_{r_k,x_0} > 0 \}} \to \overline{\{ u_{x_0} > 0 \}}$ locally in $B_R$ for any $R > 0$ in the Hausdorff-sense.
\end{itemize}

\end{corollary}

\begin{proof}
Note that $B_1(x_0) \subset \Omega$. We observe that since $u$ is a minimizer of $\cI_{B_1(x_0)}$ in $B_1(x_0)$, we also have that $u_{r_k,x_0}$ is a minimizer of $\cI_{B_{r_k^{-1}}}$ in $B_{r_k^{-1}}$ with kernel $K_k(h) := r_k^{n+2s} K(r_k h)$.
Then, by \autoref{lemma:scaling-blow-up} we have that $(u_{r_k,x_0})_r$ is uniformly bounded in $H^s(B_R) \cap C^s(B_R) \cap L^1_{2s}(\R^n)$ for any $R > 0$, and converges (up to a subsequence) locally uniformly and in $L^1_{2s}(\R^n)$ to some $u_{x_0} \in H^s(B_R) \cap L^1_{2s}(\R^n)$. Thus, we can apply \autoref{lemma:convergence-minimizer} to $(u_{r_k,x_0})_k$, which yields that $u_{\infty}$ minimizes $\cI_{B_R}$ for any $R > 0$ for some kernel $K_{\infty}$ and also implies strong convergence in $H^s(B_R)$. Since we already know by assumption that $u_{r_k,x_0} \to u_{x_0}$ locally uniformly, it must be $u_{x_0} = u_{\infty}$. Moreover, by the non-degeneracy (see \autoref{thm:nondeg}), $u_{\infty}$ must be nontrivial. Thus, (i), (ii), (iii) immediately follow from \autoref{lemma:convergence-minimizer}. To see (iv), we apply \cite[Lemma 6.5, Lemma 6.6, Lemma 6.7]{Vel23} to $u_{r_k,x_0}$, and $u_{x_0}$. The assumptions of \cite[Lemma 6.5(b)]{Vel23} and \cite[Lemma 6.6(b)]{Vel23} follow immediately from \autoref{thm:nondeg} and \autoref{lemma:weak-nondeg}.
\end{proof}

\section{Construction of half-space solutions}
\label{sec:half-space}

The goal of this section is to construct 1D solutions in the half-space, thereby proving \autoref{thm-1D}. The following lemma establishes this goal in one dimension and follows by using the results from Section \ref{sec:one-phase}.

\begin{lemma}
\label{lemma:1d-barrier}
Let $L \in \mathcal{L}^1_s(\lambda,\Lambda)$. There exists $b \in L^{1}_{2s}(\R) \cap H^s((0,R))$ for any $R > 0$ such that $b$ minimizes $\cI_{(0,R)}$ in $(0,R)$ for any $R > 0$ and some $M > 0$, depending only on $s,\lambda,\Lambda$, and it holds
\begin{align*}
\begin{cases}
L b &= 0 ~~ \text{ in } \{ x > 0 \},\\
b &= 0 ~~ \text{ in } \{ x \le 0\}.
\end{cases}
\end{align*}
Moreover, we have $b \in C^s((-R,R))$ with $\Vert b \Vert_{C^s((-R,R))} \le C(R)$ for some constant $C(R) > 0$, depending only on $s,\lambda,\Lambda,R$, for every $R > 0$, and it holds 
\begin{align*}
c_1(x_+)^s \le b(x) \le c_2 (x_+)^s ~~ \text{ in } \R,
\end{align*}
for some constants $0 < c_1 \le c_2 < \infty$ depending only on $s,\lambda,\Lambda$.
\end{lemma}

\begin{remark}[Uniqueness of the half-space solution]
\label{rem:unique-barrier}
By the uniqueness of positive solutions in cones (see \cite[Proposition 4.4.5]{FeRo24}), it follows that all functions satisfying the properties of \autoref{lemma:1d-barrier} coincide up to a constant. In this sense, $b$ is uniquely determined.
\end{remark}

\begin{proof}
The proof consists of several steps. The main idea is to first construct minimizers of $\cI_{(-1,0)}$ that have a free boundary point in $(-1,0)$ (see Step 1). Then, we prove by compactness that the positivity set cannot get arbitrarily small (see Step 2). Finally, by a scaling and blow-up argument, we construct the half-space solution (see Step 3), as desired. Then, the remaining properties follow from \autoref{thm:or} and \autoref{thm:nondeg}.

Step 1: Let $L \in \mathcal{L}^1_s(\lambda,\Lambda)$ and let $b$ be a minimizer of $\cI_{(-1,0)}$ in $(-1,0)$ with $b \equiv \1_{(0,\infty)}$ in $\R \setminus (-1,0)$. Note that such minimizer exists by \autoref{lemma:existence}. We claim that $b \not\equiv 0$ in $(-1,0)$, and that there exists $M_0 > 1$, depending only on $s,\lambda,\Lambda$, such that for any $M > M_0$ there is $x_0 \in (-1,0)$ with $b(x_0) = 0$.\\
First, let us rule out that $b > 0$ in $(-1,0)$. In fact, if this was true, then we would have $Lb = 0$ in $(-1,0)$ by \autoref{lemma:aux}(iv). Moreover, let us define $w$ as follows:
\begin{align*}
w(x) = \begin{cases}
0 ~~ &\text{ in } (-\infty, -\frac{1}{2}],\\
2(x+\frac{1}{2}), ~~ &\text{ in } (-\frac{1}{2} , 0),\\
1 ~~ &\text{ in } [0,\infty).
\end{cases}
\end{align*}
Clearly, $w$ has finite energy as a Lipschitz function. Therefore,
\begin{align*}
\cI_{(-1,0)}(b) - \cI_{(-1,0)}(w) = \cE_{((-1,0)^c \times (-1,0)^c)^c}(b,b) - \cE_{((-1,0)^c \times (-1,0)^c)^c}(w,w) + \frac{M}{2} \ge -C_1 + \frac{M}{2}
\end{align*}
for some $C_1 > 0$, depending only on $s,\lambda,\Lambda$. Thus, choosing $M > 2 C_1$, it follows that
\begin{align*}
\cI_{(-1,0)}(w) < \cI_{(-1,0)}(b),
\end{align*}
which implies that $b$ cannot be a minimizer. Thus, there exists $x_0 \in (-1,0)$ with $b(x_0) = 0$.\\
Next, let us rule out that $b \equiv 0$ in $(-1,0)$, i.e., that $b \equiv \1_{(0,\infty)}$. In case $s \ge 1/2$ this is trivial. Indeed,
\begin{align*}
\cE_{((-1,0)^c \times (-1,0)^c)^c}^K(\1_{(0,\infty)},\1_{(0,\infty)}) \ge \frac{\lambda}{2} [\1_{(0,\infty)}]_{V^s((-1,0)|\R)}^2 = \infty,
\end{align*}
but clearly, any minimizer $b$ of $\cI_{(-1,0)}$ in $(-1,0)$ with $b \equiv \1_{(0,\infty)}$ in $\R \setminus (-1,0)$ must have finite energy. In case $s < 1/2$, we need to work a little more, since $\1_{(0,\infty)} \in V^s((0,1) \ \R)$. For $t \in (0,1)$, we define $u_t$ to be the unique function satisfying
\begin{align*}
\begin{cases}
L u_t &= 0 ~~ \text{ in } (-t , 0) ,\\
u_t &= 0 ~~ \text{ in } (-\infty,-t],\\
u_t &= 1 ~~ \text{ in } [0,\infty),
\end{cases}
\end{align*}
and moreover, we set $u_0 = \1_{(0,\infty)}$
We claim that there exists $t_0 \in (0,1)$ such that $\cI_{(-1,0)}(u_{t_0}) < \cI_{(-1,0)}(u_0)$. In fact, note that by scaling
\begin{align*}
\cI_{(-1,0)}(u_{t}) = \cE_{((-1,0)^c \times (-1,0)^c)^c}^K(u_t,u_t) + t M =  t^{1-2s} \cE_{((-t^{-1},0)^c \times (-t^{-1},0)^c)^c}^{\tilde{K}_t}(\tilde{u}_t,\tilde{u}_t) + t M,
\end{align*}
where $\tilde{u}_t(x) = u_t(x/t)$ and $\tilde{K}_t(h) = t^{-(n+2s)} K(h/t)$, and it holds $\tilde{L}_t \in \mathcal{L}_s^1(\lambda,\Lambda)$, as well as
\begin{align*}
\begin{cases}
\tilde{L}_t \tilde{u}_t &= 0 ~~ \text{ in } (-1 , 0) ,\\
\tilde{u}_t &= 0 ~~ \text{ in } (-\infty,-1],\\
\tilde{u}_t &= 1 ~~ \text{ in } [0,\infty),
\end{cases}
\qquad
\tilde{u}_0 = u_0.
\end{align*}
Then, by \autoref{lemma:harmonic-replacement}, we compute
\begin{align*}
\cI_{(-1,0)}(u_{0}) - \cI_{(-1,0)}(u_{t}) = t^{1-2s} \cE_{((-t^{-1},0)^c \times (-t^{-1},0)^c)^c}^{\tilde{K}_t}(\tilde{u}_t - u_0,\tilde{u}_t - u_0) - t M.
\end{align*}
Clearly, $u_0 - \tilde{u}_t = -\1_{(-1,0)} \tilde{u}_t$, and since $\tilde{u}_t$ is harmonic in $(-1,0)$, by \autoref{prop:beta-regularity} we have that $u \ge 1 - c_0 |x|^{\delta}$  in $(-\frac{1}{2},0)$, where $c_0 > 0$ and $\delta > 0$ , depend only on $s,\lambda,\Lambda$. Thus, we estimate 
\begin{align*}
 \cE_{((-t^{-1},0)^c \times (-t^{-1},0)^c)^c}^{\tilde{K}_t}(\tilde{u}_t - u_0,\tilde{u}_t - u_0) & \ge \int_1^{\infty} \int_{-\frac{1}{2}}^0 u^2(y) K(x-y) \d y \d x \\
&\ge \int_{-\frac{1}{2}}^0 (1 - c_0 |y|^{\delta})^2 \left( \int_1^{\infty}  |x-y|^{-1-2s} \d x \right) \d y \ge c_1 > 0
\end{align*}
for some $c_1 > 0$, depending only on $s,\lambda,\Lambda$. Altogether, it  follows that for $t_0 \in (0,1)$ small enough,
\begin{align*}
\cI_{(-1,0)}(u_{0}) - \cI_{(-1,0)}(u_{t_0}) \ge t_0^{1-2s} c_1 - t_0 M > 0.
\end{align*}
Therefore, $u_0 = \1_{(0,\infty)}$ cannot be a minimizer, as we claimed. This finishes the proof of Step 1.

Step 2: Let now $M$ be so large that the claim in Step 1 holds true. Next, we claim that there exists $\eps > 0$, depending only on $s,\lambda,\Lambda,M$, such that for any $L \in \mathcal{L}^1_s(\lambda,\Lambda)$ with kernel $K$ there exist a minimizer $b_K =: b$ of $\cI_{(-1,0)}$ in $(-1,0)$, satisfying $b = \1_{(0,\infty)}$ in $\R \setminus (-1,0)$, and an interval $I \subset (-1,0)$ with $|I| \ge \eps$ such that $I \subset \{ b > 0 \}$.\\
Let us assume by contradiction that this property does not hold true. Then, we can find sequences $(\eps_k)_k$ with $\eps_k \searrow 0$, $(L_k)_k \subset \mathcal{L}^1_s(\lambda,\Lambda)$ with kernels $(K_k)_k$, and minimizers $(b_k)_k$ of $\cI_{(-1,0)}$ in $(-1,0)$ with $b_k = \1_{(0,\infty)}$ in $\R \setminus (-1,0)$ such that $\{ b_k > 0 \}$ does not contain any interval of length $\ge \eps_k$. Note that we have $L_k b_k \le 0$ in $(-1,0)$ by \autoref{lemma:aux}, and therefore, by the maximum principle it holds $0 \le b_k \le 1$. In particular, by the $C^s$ regularity (see \autoref{thm-onephase-intro}) and \autoref{lemma:energy-bound}, we have that $(b_k)_k$ is locally uniformly bounded in $C^s((-1,0)) \cap H^s((-1,0)) \cap L^1_{2s}(\R^n)$. Moreover, since the exterior data is uniformly bounded, the sequence $b_k$ also satisfies \eqref{eq:unif-tail-decay}. Hence, \autoref{lemma:convergence-minimizer} is applicable, and we obtain that there exist an operator $L_{\infty} \in \mathcal{L}^1_s(\lambda,\Lambda)$ and $b_{\infty} \in L^{1}_{2s}(\R)$ such that $b_{\infty}$ is a local minimizer of $\cI_{(-1,0)}$ in $(-1,0)$ with respect to $L_{\infty}$, and $b_k \to b_{\infty}$ locally uniformly in $(-1,0)$, up to a subsequence. In particular, it must be $b_{\infty} \equiv 0$ in $(-1,0)$, since $\{ b_k > 0 \} \cap (-1,0)$ does not contain any interval of length $\ge \eps_k$, and $\eps_k \searrow 0$. Thus, it must be $b_{\infty} \equiv \1_{(0,\infty)}$. However, it cannot hold $b_{\infty} = 0$ in $(-1,0)$ by Step 1. This is a contradiction, so the claim must be true.

Step 3: Next, we observe that up to a rescaling and translation, for any operator $L \in \mathcal{L}_s^1(\lambda,\Lambda)$ with kernel $K$ we have constructed in Steps 1 and 2 a minimizer $b_K$ of $\cI_{(-a,a)}$ in $(-a,a)$ for some (uniform) $a > 1$ such that $b_K(0) = 0$ and $(0,1) \subset \{b_K > 0 \}$. Then, by the density estimates for the free boundary (see \autoref{thm:density-est}), there is $d \in (0,a)$, depending only on $s,\lambda,\Lambda$, such that $\{ b_K = 0 \} \subset (-d,0)$.\\
Let us observe that given an operator $L$ with kernel $K$, the operators $L_R$ with kernels $K_R(h) = R^{n+2s} K(Rh)$ satisfy $L_R \in \mathcal{L}_s^1(\lambda,\Lambda)$, and the corresponding minimizers $b_{K_R}$ have the property that $\tilde{b}_R(x) := b_{K_R}(x/R)R^s$ are minimizers with respect to $K$ in $(-aR,aR)$ satisfying $(0,R) \subset \{ \tilde{b}_R > 0 \}$ and $\{\tilde{b}_R = 0\} \subset (-dR,0)$. Note that, we have $0 \le b_{K_R} \le 1$ by the maximum principle, and therefore, using \autoref{thm:or}, the family $(b_{K_R})$ is uniformly bounded in $C^s((-\frac{d}{2},\frac{d}{2}))$. Thus, by \autoref{lemma:scaling-blow-up} and \autoref{lemma:convergence-minimizer}, there exist $\tilde{b}_{\infty} \in H^s((-r,r)) \cap L^1_{2s}(\R)$, minimizing $\cI_{(-r,r)}$ for $K$ in $(-r,r)$ for every $r > 0$, and such that $\tilde{b}_R \to \tilde{b}_{\infty}$. In particular, we have $(0,\infty) \subset \{ \tilde{b}_{\infty} > 0 \}$ and $\{\tilde{b}_{\infty} = 0\} \subset (-\infty,0)$, and thus by \autoref{lemma:aux}
\begin{align*}
\begin{cases}
L \tilde{b}_{\infty} &= 0 ~~ \text{ in } \{ x > 0 \},\\
\tilde{b}_{\infty} &= 0 ~~ \text{ in } \{ x \le 0 \}.\\
\end{cases}
\end{align*}
Finally, by the optimal regularity and the non-degeneracy (see \autoref{thm-onephase-intro}), it must be $\tilde{b}_{\infty} \in C^s((-R,R))$ with $\Vert \tilde{b}_{\infty} \Vert_{C^s((-R,R))} \le C(R)$ for any $R > 0$, and moreover $c_1 (x_+)^s \le \tilde{b}_{\infty}(x) \le c_2(x_+)^s$. This concludes the proof.
\end{proof}

It is well-known that for homogeneous operators $L$, although solutions $u$ to $Lu = 0$ in $\Omega$ with $u = 0$ in $\R^n \setminus \Omega$ are also only $C^s$ up to the boundary, we have that $u/d^s \in C^{\infty}(\overline{\Omega})$, when the kernel and the domain $\Omega$ are smooth. For inhomogeneous kernels, such higher regularity results fail, as the following remark clarifies.

\begin{remark}
\label{remark:optimality}
Let us consider the half-space solution $b \in C^s_{loc}(\R)$ from \autoref{lemma:1d-barrier} associated to an operator $L \in \mathcal{L}_s^1(\lambda,\Lambda)$. Note that continuity of the quotient $b/(x_+)^s$ with a modulus of continuity $\omega(r)$ would be equivalent to the following expansion for some $c_0 \in \R$
\begin{align*}
|b(x) - c_0 x^s| \le \omega(x)x^s ~~ \forall x \in (0,1).
\end{align*}
This implies that for any $x \in (0,\infty)$
\begin{align}
\label{eq:hom-limit}
\left|\frac{b(rx)}{(rx)^s} - c_0 \right| \le \omega(rx) \to 0 ~~ \text{ as } r \to 0, \qquad \text{ i.e., } \qquad b(rx)r^{-s} \to c_0 x^s ~~ \text{ loc. unif. in } (0,\infty).
\end{align}
However, note that $b_r(x) := b(rx)r^{-s}$ coincides (up to a constant) with the half-space solution associated to $L_r \in \mathcal{L}_s^1(\lambda,\Lambda)$ with kernel $K_r(h) = r^{n+2s}K(rh)$. By the Arzel\`a-Ascoli theorem, and by stability, we have that $b_r \to b_0 = c_0(x_+)^s \in C^s_{loc}(\R) \cap L^1_{2s}(\R)$ locally uniformly and in $L^1_{2s}(\R)$ (up to a subsequence), such that $L_0 b_0 = 0$ in $\{ x > 0 \}$ and $c_1 (x_+)^s \le b_0 \le c_2(x_+)^s$ for some $0 < c_1 \le c_2 < \infty$. Here, $L_0 \in \mathcal{L}_s^1(\lambda,\Lambda)$ has a kernel $K_0$ and it holds $\min\{1,|h|^2\} K_r(h) \d h \to \min\{1,|h|^2\} K_{0}(h) \d h$ weakly in the sense of measures. However, by \autoref{rem:unique-barrier}, we have that $b_0(x) = c_0(x_+)^s$ must be the half-space solution with respect to $L_0$. Clearly, it is possible to construct an example of a kernel $K$ such that $K_0$ is inhomogeneous, in which case $L_0 (x_+)^s \not\equiv 0$ in $(0,\infty)$ by a similar reasoning as in \cite[Corollary 2.8]{RoSe16a}.
\end{remark}

\begin{lemma}
\label{lemma:1dnd}
Let $0 < \lambda \le \Lambda < \infty$, $e \in \mathbb{S}^{n-1}$, and $L \in \mathcal{L}^n_s(\lambda,\Lambda)$. Then, there exist, $0 < \tilde{\lambda} \le \tilde{\Lambda} < \infty$, depending only on $n,s,\lambda,\Lambda$, and $\tilde{L} \in \mathcal{L}^1_s(\tilde{\lambda},\tilde{\Lambda})$ such that for any $u : \R^n \to \R$ and $\tilde{u} : \R \to \R$ with $u(x) = \tilde{u}(x \cdot e)$, it holds
\begin{align*}
Lu(x) = \tilde{L} \tilde{u}(x \cdot e).
\end{align*}
\end{lemma}

\begin{proof}
Without loss of generality, we assume that $e = e_n$.
We have, after splitting $\R^n = \R^{n-1} \times \R$:
\begin{align*}
Lu(x) &= \int_{\R^n} (u(x) - u(y)) K(x-y) \d y \\
&= \int_{\R} \int_{\R^{n-1}} (\tilde{u}(x_n) - \tilde{u}(y_n)) K(x'-y',x_n - y_n) \d y' \d y_n \\
&=  \int_{\R} (\tilde{u}(x_n) - \tilde{u}(y_n)) \left(\int_{\R^{n-1}} K(x'-y',x_n - y_n) \d y' \right) \d y_n = \tilde{L} \tilde{u}(x_n),
\end{align*}
where we set $\tilde{K}(h_n) = \int_{\R^{n-1}} K(h',h_n) \d h'$. It remains to verify $\tilde{L} \in \mathcal{L}_s^1(\tilde{\lambda},\tilde{\Lambda})$. Indeed, it holds
\begin{align*}
\int_{\R^{n-1}} |h|^{-n-2s} \d h' &= \int_{\R^{n-1}} (|h'|^2 + |h_n|^2)^{-\frac{n+2s}{2}} \d h' \\
&= |h_n|^{-n-2s} \int_{\R^{n-1}} \left(\left(\frac{|h'|}{|h_n|}\right)^2 + 1 \right)^{-\frac{n+2s}{2}} \d h' \\
&= |h_n|^{-1-2s} \int_{\R^{n-1}} \left(|h'|^2 + 1\right)^{-\frac{n+2s}{2}} \d h' = c_0 |h_n|^{-1-2s},
\end{align*}
where $c_0 > 0$ depends only on $n,s$. Therefore, the desired result follows upon choosing $\tilde{\lambda} = c_0 \lambda$ and $\tilde{\Lambda} = c_0 \Lambda$.
\end{proof}

We are now in position to construct 1D barriers for operators belonging to the class $\mathcal{L}^n_s(\lambda,\Lambda)$, which proves \autoref{thm-1D}.


\begin{proof}[Proof of \autoref{thm-1D}]
Given $L \in \mathcal{L}^n_s(\lambda,\Lambda)$ with kernel $K$ and $e \in \mathbb{S}^{n-1}$, let $\tilde{L} \in \mathcal{L}^1_s(\tilde{\lambda},\tilde{\Lambda})$ and $\tilde{K}$ be the 1D kernel from \autoref{lemma:1dnd}, and denote by $\tilde{b}$ the barrier from \autoref{lemma:1d-barrier} with respect to $\tilde{K}$. Then, the result follows immediately by combination of \autoref{lemma:1dnd} and \autoref{lemma:1d-barrier}.
\end{proof}

%
%

\section{$C^s$ boundary regularity}
\label{sec:Cs}

In this section we prove the boundary regularity for solutions to nonlocal equations with inhomogeneous kernels. In particular, we prove our main results \autoref{thm-Cs-intro} and \autoref{thm-expansion-intro} establishing $C^s$ regularity in $C^{1,\alpha}$ domains, as well as a Hopf lemma (see \autoref{thm:inhom-Cs} and \autoref{thm:inhom-Hopf}). Moreover, we show a Liouville theorem with growth in the half-space (see \autoref{thm:Liouville-half-space}), and the $C^{s-\eps}$ regularity in flat Lipschitz domains (see \autoref{thm:C-s-eps}). Both results are at the same time crucial in the proof of our main results, and of independent interest.

We will need the following lemma, which proves H\"older regularity up to the boundary in flat Lipschitz domains for possibly inhomogeneous kernels. It was established in \cite[Proposition 2.6.9]{FeRo24} for zero exterior data $g \equiv 0$.

\begin{lemma}
\label{lemma:Lipschitz-Holder-prelim}
Let $L \in \mathcal{L}_s^n(\lambda,\Lambda)$. Let $\delta > 0$, $\Omega \subset \R^n$ be a Lipschitz domain with Lipschitz constant $\delta$. Let $f \in L^{\infty}(B_1 \cap \Omega)$, $g \in C^{\gamma}(\overline{B_1 \setminus \Omega})$, and $u \in C(\R) \cap L^1_{2s}(\R^n)$ be a distributional solution of 
\begin{align*}
\begin{cases}
L u &= f ~~ \text{ in } B_1 \cap \Omega,\\
u &= g ~~ \text{ in } B_1 \setminus \Omega.\\
\end{cases}
\end{align*}
Then, there are $C, \eps > 0$, depending only on $n,s,\lambda,\Lambda,\delta,\gamma$, such that
\begin{align*}
\Vert u \Vert_{C^{\eps}(B_{1/2} \cap \overline{\Omega})} \le C \left(\Vert u \Vert_{L^1_{2s}(\R^n)} + \Vert f \Vert_{L^{\infty}(B_1 \cap \Omega)} +  \Vert g \Vert_{C^{\gamma}(\overline{B_1 \setminus \Omega)}} \right).
\end{align*}
\end{lemma}

\begin{proof}
In case $g = 0$, the result can be found in \cite[Proposition 2.6.9]{FeRo24}. However, note that the barrier in \cite[Lemma B.4.2]{FeRo24} that is used in the proof, even allows for $f$ such that $d^{\alpha-2s}f \in L^{\infty}(\Omega \cap B_1)$ for some $\alpha \in (0,2s)$. Thus, in this case, we can even show
\begin{align}
\label{eq:Lipschitz-Holder-prelim-help-1}
\Vert u \Vert_{C^{\eps}(B_{1/2} \cap \overline{\Omega})} \le C \left(\Vert u \Vert_{L^1_{2s}(\R^n)} + \Vert d^{2s-\alpha} f \Vert_{L^{\infty}(B_1 \cap \Omega)} \right).
\end{align}
For general $g \in C^{\gamma}(\overline{B_1 \setminus \Omega})$, the result can be shown by following line by line the proof of \cite[Proposition 2.6.15]{FeRo24}, but applying in \cite[Step 3]{FeRo24} the already established version with zero exterior data and exploding source terms \eqref{eq:Lipschitz-Holder-prelim-help-1}, setting $\alpha := \gamma$. Note that although $g$ is a priori only given in $\overline{B_1 \setminus \Omega}$, we can easily extend it to the full space, so that it satisfies $g \in C^{\infty}(\Omega) \cap C^{\gamma}(\overline{\Omega})$, as in \cite[Step 1]{FeRo24}.
\end{proof}

\subsection{Liouville theorems in the half-space}

The goal of this section is to prove the following Liouville theorem in the half space for solutions that grow like than $x \mapsto |x|^{\beta}$ for $\beta \in (0,2s)$ at infinity. 

Since from now on, solutions might not belong to $V^s(B | \R^n)$ due to the growth at infinity, we now also consider distributional solutions, instead of weak solutions. Note that all the equations in the previous subsection can be interpreted in the distributional sense by \cite[Lemma 2.2.32]{FeRo24}.

\begin{theorem}
\label{thm:Liouville-half-space}
Let $L \in \mathcal{L}^n_s(\lambda,\Lambda)$ and $e \in \mathbb{S}^{n-1}$. Let $\beta \in (0,2s)$, and $u \in C(\R^n) \cap L^1_{2s}(\R^n)$ be a distributional solution of 
\begin{align*}
\begin{cases}
L u &= 0 ~~ \qquad\qquad \text{ in } \{ x \cdot e > 0 \},\\
u &= 0 ~~ \qquad \qquad \text{ in } \{ x \cdot e \le 0 \},\\
|u(x)| &\le C (1 + |x|^{\beta})  ~~ \forall x \in \R^n.
\end{cases}
\end{align*}
\begin{itemize}
\item[(i)] If $\beta < s$, then $u \equiv 0$.
\item[(ii)] If $\beta \ge s$, then $u \equiv \kappa b$ for some $\kappa \in \R$, and where $b(x) = \tilde{b}(x \cdot e)$ is the half-space solution from \autoref{thm-1D} corresponding to $L$ and $e$.
\end{itemize}
\end{theorem}

First, we establish the theorem in 1D. We will treat the cases (i) and (ii) separately. Let us first state another corollary of \autoref{lemma:1d-barrier}.

\begin{corollary}
\label{cor:1d-reg-half-space}
Let $L \in \mathcal{L}_s^1(\lambda,\Lambda)$, $R > 0$, $f \in L^{\infty}(\{ x > 0 \} \cap (0,2R))$, and $u \in H^s((-2R,2R)) \cap L^1_{2s}(\R^n)$ be a weak solution to
\begin{align*}
\begin{cases}
L u &= f ~~ \text{ in } \{ x > 0 \} \cap (0,2R),\\
u & = 0 ~~ \text{ in } \{ x \le 0 \}.
\end{cases}
\end{align*}
Then, there exists a constant $C > 0$, depending only on $s,\lambda,\Lambda,$ such that
\begin{align*}
[ u ]_{C^s((0,R))} \le C R^{-s} \left( \dashint_{0}^R |u(x)| \d x + R^{2s}\int_{R}^{\infty} |u(x)||x|^{-1-2s} \d x + \Vert f \Vert_{L^{\infty}(\{ x > 0 \} \cap (0,2R))} \right).
\end{align*}
\end{corollary}

\begin{proof}
By scaling, it suffices to prove the result in case $R = 1$ and under the assumption that $\Vert u \Vert_{L^1_{2s}(\R)} + \Vert f \Vert_{L^{\infty}(\{ x > 0\} \cap (0,2))} \le 1$. Moreover, by the interior regularity estimates (see \cite[Theorem 2.4.3]{FeRo24}), it suffices to prove that
\begin{align*}
|u(x)| \le C (x_+)^s ~~ \forall x \in (0,1).
\end{align*}
To prove it, we consider $v := \1_{(0,3/2)}u$ and $B := \1_{(0,3/2)}b$, where $b$ denotes the barrier from \autoref{lemma:1d-barrier}. A simple computation reveals that there are $c_0,c_1 > 0$ such that $L B \ge c_0$ and $L v \le c_1$ in $(0,1)$. Moreover, there are $c_3,c_4 > 0$ such that  $B \ge c_3(x_+)^s \ge c_3$ and $v = u \le c_4 \Vert u \Vert_{L^1_{2s}(\R)} \le c_4$ in $(1,3/2)$. Thus, by the comparison principle, it must be
\begin{align*}
u(x) = v(x) \le B \le c_5 (x_+)^s ~~ \forall x \in (0,1),
\end{align*}
as desired.
\end{proof}

As a consequence, we get a Liouville theorem in the half-space in 1D for solutions that grow slower that $x \mapsto |x|^s$. Note that, even this result was only known for nonlocal operators with homogeneous kernels.

\begin{lemma}
\label{lemma:Liouville-1d-1}
Let $L \in \mathcal{L}_s^1(\lambda,\Lambda)$, $\eps > 0$, and $u \in H^s(\R) \cap L^1_{2s}(\R^n)$ be a weak solution of 
\begin{align*}
\begin{cases}
L u &= 0 ~~~~ \qquad\qquad \text{ in } \{ x > 0 \},\\
u &= 0 ~~~~ \qquad\qquad \text{ in } \{ x \le 0 \},\\
|u(x)| &\le C (1 + |x|^{s-\eps})  ~~ \forall x \in \R.
\end{cases}
\end{align*}
Then, it holds $u \equiv 0$.
\end{lemma}

\begin{proof}
By application of \autoref{cor:1d-reg-half-space}, we deduce that for any $R > 1$:
\begin{align*}
[u]_{C^s((0,R))} &\le C R^{-s} \left( \dashint_{0}^R |u(x)| \d x + R^{2s}\int_{R}^{\infty} |u(x)||x|^{-1-2s} \d x \right) \\
&\le C R^{-s} \left( 1 + R^{s-\eps} + R^{2s} \int_R^{\infty} |x|^{-1-s-\eps} \d x \right) \le C R^{-\eps} \to 0, ~~ \text{ as } R \to 0.
\end{align*}
Thus, $u$ must be constant. Since $u = 0$ in $\{ x \le 0 \}$, it follows that $u \equiv 0$, as desired.
\end{proof}

To establish a Liouville theorem in the half-space for solutions that grow faster than $x \mapsto |x|^s$, we need the following two auxiliary higher order regularity results for 1D solutions. The first claim is a direct consequence of the boundary Harnack principle. The second result is an interior gradient estimate, which follows by an adaptation of \cite{RoVa16} to inhomogeneous kernels, using \autoref{cor:1d-reg-half-space}. Note that in case $s > 1/3$, the latter would follow immediately by standard Schauder estimates, as in \cite{RoSe16b}. 

\begin{lemma}
\label{lemma:1d-reg-half-space-higher}
Let $L \in \mathcal{L}_s^1(\lambda,\Lambda)$, $R > 0$ and $u \in C(\R) \cap L^1_{2s}(\R^n)$ be a distributional solution to
\begin{align*}
\begin{cases}
L u &= 0  ~~ \text{ in } \{ x > 0 \},\\
u & = 0 ~~ \text{ in } \{ x \le 0 \}.
\end{cases}
\end{align*}
Then, there exist a constant $C > 0$, and $\alpha \in (0,1)$, depending only on $s,\lambda,\Lambda,$ such that it holds
\begin{align}
\label{eq:bdry-Harnack}
\left[ \frac{u}{b} \right]_{C^{\alpha}((0,R))} \le C R^{-s-\alpha} \left(  \dashint_{0}^R |u(x)| \d x + R^{2s}\int_{R}^{\infty} |u(x)||x|^{-1-2s} \d x \right),
\end{align}
where $b$ is the half-space solution from \autoref{lemma:1d-barrier} corresponding to $L$. Moreover, it holds for any $\alpha \in (0,1+s)$ such that $2s+\alpha \not \in \N$:
\begin{align}
\label{eq:interior-higher-order-est}
[u]_{C^{2s+\alpha}(R,3R)} \le C R^{-2s-\alpha} \left( \dashint_{(0,R)} |u(y)| \d y + R^{2s}\int_{R}^{\infty} |u(y)||y|^{-1-2s} \d y \right) ~~ \forall x > 0,
\end{align}
and 
\begin{align}
\label{eq:interior-gradient-est}
|u'(x)| \le C (x_+)^{-1} \left( \dashint_{0}^{x_+} |u(y)| \d y + (x_+)^{2s}\int_{x_+}^{\infty} |u(y)||y|^{-1-2s} \d y \right) ~~ \forall x > 0.
\end{align}

\end{lemma}

\begin{proof}
For $R = 1$ the first claim follows immediately from the boundary Harnack principle (see \cite[Theorem 4.3.1]{FeRo24}), where we used that $Lb = 0$ in $\{ x > 0\}$, $b \ge 0$, and that there exist $c_2 \ge c_1 > 0$, depending only on $s,\lambda,\Lambda$ such that $c_1 \le \Vert b \Vert_{L^1_{2s}(\R)} \le c_2$ by \autoref{lemma:1d-barrier}. It is important to note that the proof can easily adapted in case $u$ is not assumed to be nonnegative by setting $m = -M$ in the proof of \cite[Proposition 4.3.6]{FeRo24}. Moreover, the result is applicable since $u$ is also a viscosity solution by \cite[Lemma 3.4.13]{FeRo24}. The result for general $R > 0$ follows by the observation that if $u$ satisfies the assumptions with $R$, then $u_R(x) = u(Rx)$ satisfies the assumption with $R = 1$ and operator $L_R$ having kernel $K_R(h) = R^{1+2s}K(Rh)$. Moreover, note that $b_R(x) = b(Rx)R^{-s}$ satisfies $L_R b_R = 0$ in $\{ x > 0 \}$ and $c_1\le \Vert b \Vert_{L^1_{2s}(\R)} \le c_2$. Thus, since
\begin{align*}
\Vert u_R \Vert_{L^1_{2s}(\R)} = \dashint_{0}^R |u(x)| \d x + R^{2s} \int_R^{\infty} |u(x)| |x|^{-1-2s} \d x,
\end{align*}
the result follows immediately from the fact that
\begin{align*}
\left[ \frac{u_R}{b_R} \right]_{C^{\alpha}((0,1))} = R^{\alpha} \left[ \frac{u}{b R^{-s}} \right]_{C^{\alpha}((0,R))} = R^{s+\alpha} \left[ \frac{u}{b} \right]_{C^{\alpha}((0,R))},
\end{align*}
and application of the result in case $R = 1$ to $u_R$ and $b_R$.

The proof of the second claim follows exactly in the same way as in \cite[Theorem 1.1]{RoVa16}. Indeed, although in \cite{RoVa16} it is assumed that kernels are homogeneous, all the arguments in \cite[Section 4,5,6]{RoVa16} go through in the exact same way for inhomogeneous kernels. The only point where homogeneity is used in their paper is in order to have the estimate \cite[(6.52)]{RoVa16}, which in our case becomes
\begin{align*}
\Vert u \Vert_{C^s([0,1])} \le C \Vert u \Vert_{L^{1}_{2s}(\R)},
\end{align*}
and was proved in \autoref{cor:1d-reg-half-space}. Note that the proof in \cite{RoVa16} easily carries over to unbounded domains. Thus, by scaling we immediately obtain for any $R > 0$
\begin{align*}
[u]_{C^{2s+\alpha}(R,3R)} \le C R^{-1} \left( \dashint_{(R,3R)} |u(y)| \d y + R^{2s}\int_{\R \setminus (R,3R)} |u(y)||y-2R|^{-1-2s} \d y \right) ~~ \forall x > 0,
\end{align*}
which immediately implies \eqref{eq:interior-higher-order-est} upon estimating the integrals in a straightforward way.
Finally, \eqref{eq:interior-gradient-est} follows from \eqref{eq:interior-higher-order-est} applied with $R = x_+$ by H\"older interpolation (for any small $\alpha > 0$) and the local boundedness estimate (see \cite[Theorem 6.2]{Coz17}):
\begin{align*}
|u'(x)| &\le \Vert u' \Vert_{L^{\infty}(x,3x)} \le C (x_+)^{-1} \Vert u \Vert_{L^{\infty}(x,3x)} + (x_+)^{\alpha} [u']_{C^{\alpha}(x,3x)} \\
&\le C (x_+)^{-1} \left( \dashint_{0}^{x_+} |u(y)| \d y + (x_+)^{2s}\int_{x_+}^{\infty} |u(y)||y|^{-1-2s} \d y \right).
\end{align*}
\end{proof}

By application of the previous lemma, we have the following control over the derivative of the half-space solutions from \autoref{lemma:1d-barrier}.

\begin{lemma}
\label{lemma:b-derivative-bounds}
Let $L \in \mathcal{L}_s^1(\lambda,\Lambda)$, and $b \in L^1_{2s}(\R^n)$ be the half-space solution from \autoref{lemma:1d-barrier}. Then, $b \in C^{2s+\alpha}_{loc}(\R)$ for any $\alpha \in (0,1+s)$, and it holds
\begin{align*}
c_1 (x_+)^{s-1} \le b'(x) \le c_2 (x_+)^{s-1} ~~ \text{ in } \R.
\end{align*}
\end{lemma}

\begin{proof}
By application of \autoref{lemma:1d-reg-half-space-higher} to $b$, using also the upper bounds for $b$, we immediately obtain $b \in C^{2s+\alpha}_{loc}(\R)$, as well as 
\begin{align}
\label{eq:application-interior-b}
[b]_{C^{2s+\alpha}(x_+,3x_+)} \le C (x_+)^{-2s-\alpha} \left( \dashint_0^{x_+} y^s \d y + (x_+)^{2s} \int_{x_+}^{\infty} y^{-1-s} \d y \right) \le C (x_+)^{-s-\alpha} ~~ \forall x \in \R.
\end{align}
In particular, by setting $\alpha = 1-2s$, we deduce the desired upper bound for $b'$. To prove the lower bound, we proceed as follows. First, we claim that $b' \ge 0$. To see that, let us consider for any $h \in (0,1)$ the function $w(x) = b(x+h) - b(x)$. Clearly, as a consequence of all the properties of $b$, it holds
\begin{align*}
\begin{cases}
L w &= 0~~ \text{ in } (0,\infty),\\
w & > 0 ~~ \text{ in } (-h,0],\\
w &= 0 ~~ \text{ in } (-\infty,-h].
\end{cases}
\end{align*}
Moreover, by the upper bound for $b'$ we have that $|w(x)| \le 2c_2 (x_+)^{s-1}$, so in particular, $w \in L^{\infty}(\R)$, and $|w(x)| \to 0$, as $x \to \infty$. Let us now assume that $w$ is not nonnegative. Then, it must attain its infimum at some $x_0 \in (0,\infty)$, i.e., $0 > \inf_{\R} w = w(x_0)$. However, as a consequence, it must be $Lw(x_0) > 0 = Lw(x_0)$, which is a contradiction. Thus, it must be $w(x) \ge 0$, and since $h > 0$ was arbitrary, and $b \in C^1_{loc}((0,\infty))$, we deduce that $b' \ge 0$ in $(0,\infty)$, as desired.\\
Having this information in mind, we can now finish the proof of the lower bound. First, note that $L b' = 0$ in $(0,\infty)$. Hence, by the Harnack inequality it must be $b' > 0$ in $(0,\infty)$. We claim that there exists a uniform constant $c_0 > 0$, depending only on $n,s,\lambda,\Lambda$, such that 
\begin{align}
\label{eq:b-prime-bounded}
b'(1) \ge c_0.
\end{align}
Note that once this is proved, the desired result follows immediately by scaling. Indeed, given $x > 0$ note that $b_x(r) := x^{-s} b(xr)$, which is the half-space solution associated with the operator $L_x \in \mathcal{L}_s^1(\lambda,\Lambda)$ with kernel $K_x(h) = x^{n+2s}K(xh)$, satisfies $b'_x(r) = x^{-s} (b(x \cdot))'(r) = x^{1-s} b'(xr)$, and thus evaluating at $r = 1$, it holds
\begin{align*}
c_0 x^{s-1} \le x^{s-1} b'_x(1) \le b'(x),
\end{align*}
as desired. So it remains to prove \eqref{eq:b-prime-bounded}. We will do so by contradiction. Assume that there exists a sequence $(L_k) \subset \mathcal{L}_s^1(\lambda,\Lambda)$ such that the associated half-space solutions $b_k$ satisfy $b_k'(1) \searrow 0$. Note that by \eqref{eq:application-interior-b} we have $b_k \in C^{1+\eps}_{loc}((0,\infty))$ with a uniform bound for some $\eps > 0$. Moreover, by \autoref{lemma:1d-barrier}, we have $b_k \in C^s((-1,1))$ with a uniform bound. Hence, by the Arzel\`a-Ascoli theorem, we have that $b_k \to b_{\infty} \in C^s(-1,1) \cap C^{1+\eps}_{loc}((0,\infty))$ uniformly in $\R$, and moreover also $b_k' \to b_{\infty}'$ locally uniformly in $(0,\infty)$, so in particular 
\begin{align}
\label{eq:b-infty-prime-properties}
b_{\infty}'(x) \ge 0, \qquad b_{\infty}'(1) = 0, \qquad b_{\infty}'(x) \le c (x_+)^{s-1} ~~ \forall x \in \R.
\end{align}
Moreover, by the uniform bounds on $b_k$ the convergence also holds in $L^1_{2s}(\R)$, so by the stability for distributional solutions (see \cite[Proposition 2.2.36]{FeRo24}) there exists $L_{\infty} \in \mathcal{L}_s^1(\lambda,\Lambda)$ such that $L_{\infty} b_{\infty} = 0$ in $(0,\infty)$. Since $b_{\infty}'\in L^1_{2s}(\R)$ by \eqref{eq:b-infty-prime-properties}, also $L_{\infty} b_{\infty}' = 0$ in $(0,\infty)$. Thus, by the Harnack inequality, it must be $b_{\infty}' > 0$ in $(0,\infty)$, which contradicts $b_{\infty}'(1) = 0$. Thus, we have shown \eqref{eq:b-prime-bounded}, and the proof is complete.
\end{proof}

As a consequence, we get a higher order Liouville theorem in the half-space in 1D:

\begin{lemma}
\label{lemma:Liouville-1d-2}
Let $L \in \mathcal{L}_s^1(\lambda,\Lambda)$, $\eps \in (0,s)$, and $u \in C(\R) \cap L^1_{2s}(\R^n)$ be a distributional solution of 
\begin{align*}
\begin{cases}
L u &= 0 ~~~~ \qquad\qquad \text{ in } \{ x > 0 \},\\
u &= 0 ~~~~ \qquad\qquad \text{ in } \{ x \le 0 \},\\
|u(x)| &\le C (1 + |x|^{s+\eps})  ~~ \forall x \in \R.
\end{cases}
\end{align*}
Then, $u \equiv \kappa b$ for some $\kappa \in \R$, where $b$ is the half-space solution from \autoref{lemma:1d-barrier} corresponding to $L$.
\end{lemma}

\begin{proof}
By application of \eqref{eq:bdry-Harnack} from \autoref{lemma:1d-reg-half-space-higher}, we deduce that for any $R > 1$:
\begin{align*}
\left[ \frac{u}{b} \right]_{C^{\alpha}((0,R))} &\le C R^{-s-\alpha} \left(  \dashint_{0}^R |u(x)| \d x + R^{2s}\int_{R}^{\infty} |u(x)||x|^{-1-2s} \d x \right) \\
&\le C R^{-s-\alpha} \left( 1 + R^{s+\eps} + R^{2s} \int_R^{\infty} |x|^{-1-s+\eps} \d x \right) \le c R^{\eps-\alpha}.
\end{align*}
Hence, in case $\eps < \min \{ \alpha, s\}$, where $\alpha \in (0,1)$ denotes the constant from \autoref{lemma:1d-reg-half-space-higher}, we have that 
\begin{align*}
\left[ \frac{u}{b} \right]_{C^{\alpha}((0,R))} \le R^{\eps - \alpha} \to 0, ~~ \text{ as } R \to \infty,
\end{align*}
so, $u/b$ must be constant. This implies the desired result.\\
Let us now focus on the case $\eps \in [\alpha,s)$. Then, using that $u/b \in C^{\alpha}([0,1])$, we find that $\kappa := \lim_{x \searrow 0} \frac{u(x)}{b(x)}$ exists, and that the function
\begin{align*}
v(x) := u(x) - \kappa b(x)
\end{align*} 
satisfies $|v(x)| \le C b(x)(x_+)^{\alpha} \le C (x_+)^{s+\alpha}$ for any $x \in [0,1]$, where we also used the upper bound for $b$ from \autoref{lemma:1d-barrier}. Hence, by application of \eqref{eq:interior-gradient-est} we deduce for $x \in [0,1]$, using that $\alpha \le \eps$
\begin{align*}
|v'(x)|\le C (x_+)^{-1} \left( \dashint_0^{x_+} \hspace{-0.2cm} y^{s+\alpha} \d y + (x_+)^{2s} \int_{x_+}^{\infty} \hspace{-0.2cm} y^{-1-s+\eps} \d y \right) \le C \left( (x_+)^{s+\alpha-1} + (x_+)^{s+\eps-1} \right) \le C (x_+)^{s+\alpha - 1}.
\end{align*}
Similarly, for $x \in (1,\infty)$, we have
\begin{align*}
|v'(x)| \le C (x_+)^{-1} \left( 1 + (x_+)^{-1}\int_1^{x_+} y^{s+\eps} \d y + (x_+)^{2s} \int_{x_+}^{\infty} y^{-1-s+\eps} \d y \right) \le C (x_+)^{s+\eps-1}.
\end{align*}
Next, we define for $A > 0$, to be chosen later, the function
\begin{align*}
w(x) = A(b(x) + b'(x)).
\end{align*}
Note that as a consequence of \autoref{lemma:b-derivative-bounds} and \autoref{lemma:1d-barrier}, we have that for some $c > 0$
\begin{align*}
w(x) \ge c A (x_+)^s ~~ \text{ for } x \in (1,\infty), \qquad w(x) \ge c A (x_+)^{s-1} ~~ \text{ for } x \in (0,1]. 
\end{align*}
Consequently, using that $s+\alpha-1 < s-1$, as well as $s+\eps-1 < 2s-1<s$, we deduce that it must be $w >> v'$ at zero and at infinity, so if $A$ is chosen large enough, it holds $w \ge v'$ in $\R$. Next, let us define by $A^{\ast} \ge 0$ the smallest number such that $w \ge v'$. By the previous considerations, if $A^{\ast} > 0$, in that case we still have $w \ge v'$, and there must be $x_0 \in (0,\infty)$ such that $w(x_0) = v'(x_0)$. Hence, $w-v'$ has a strict global minimum at $x_0$, and therefore $L(w-v')(x_0) < 0$. However, since $Lw = Lv' = 0$ in $(0,\infty)$ by construction, this leads to a contradiction, unless $A^{\ast} = 0$. Thus, it must be $v'\le 0$, and therefore $v \le 0$. By repeating the same argument with $-v$ instead of $v$, we deduce that $v \ge 0$, and hence $v \equiv 0$, i.e., it holds $u = \kappa b$, as desired.
\end{proof}

By combining the 1D Liouville theorems (see \autoref{lemma:Liouville-1d-1} and \autoref{lemma:Liouville-1d-2}) with \autoref{lemma:1dnd} we obtain the following Liouville theorem in the half-space.

\begin{proof}[Proof of \autoref{thm:Liouville-half-space}]
The proof follows the same arguments as in \cite[Theorem 2.7.2]{FeRo24}.  Let us assume that $e = e_n$. We take $R > 1$, $h \in \R^n$ with $h_n = 0$ and $|h| \le R/4$, and $\eps \in (0,1)$ to be the constant from \autoref{lemma:Lipschitz-Holder-prelim} (in case $g \equiv 0$). Then, we observe that $u_h(x) = \frac{u(x+h) - u(x)}{|h|^{\eps}}$ is a distributional solution to
\begin{align*}
\begin{cases}
L u_h &= 0 ~~ \text{ in } \{ x_n > 0 \},\\
u_h &= 0 ~~ \text{ in } \{ x_n \le 0 \}.
\end{cases}
\end{align*}
Thus, by the boundary regularity for inhomogeneous nonlocal operators (see \autoref{lemma:Lipschitz-Holder-prelim} with $g \equiv 0$) we have
\begin{align*}
\Vert u_h \Vert_{L^{\infty}(B_R)} = [u]_{C^{\eps}(B_R)} \le C R^{-\eps} \left( \dashint_{B_{R}} |u(x)| \d x + \tail(u;R) \right) \le C R^{\beta-\eps}. 
\end{align*}
Taking now incremental quotients of $u_h$, and repeating this procedure $k$ times, where $k \in \N$ is so large that $\beta - k \eps < 0$, we eventually deduce that $u$ must be a polynomial in $x'$ with coefficients that are 1D functions in $x_n$. Then, by the growth condition on $u$, we deduce
\begin{align*}
u(x) = A(x_n) + A'(x_n) \cdot x', ~~ \forall x \in \R,
\end{align*}
where $A : \R \to \R$ and $A'(x_n) : \R \to \R^n$. Note that $A_i'(x_n) = \partial_i u(x)$ for any $i \in \{1,\dots,n-1\}$, and therefore, by  \autoref{lemma:1dnd} we see that $\tilde{L} A'_i (x_n) = L (\partial_i u)(x) = 0$ if $x_n > 0$. Moreover, since $A'_i(x_n) = 0$ in $\{x_n < 0 \}$,   depending on whether $\beta < s$, or $s < \beta < s+\alpha$, we can apply the 1D Liouville theorems (see \autoref{lemma:Liouville-1d-1}, or \autoref{lemma:Liouville-1d-2}) to deduce that $A'_i(x_n) \equiv 0$, since otherwise $u$ would have to grow like $x \mapsto x_+^{1+s}$, which would contradict the growth-assumption. Thus, it must be $u(x) = A(x_n)$. Using again \autoref{lemma:1dnd} and the 1D Liouville theorems, we deduce that $u(x) = A(x_n) = 0$, or $u(x) = \kappa b(x_n)$, respectively. This concludes the proof.
\end{proof}

\subsection{$C^{s-\eps}$ regularity in flat Lipschitz domains}

The goal of this section is to prove that solutions to nonlocal equations with inhomogeneous kernels are $C^{s-\eps}$ up to the boundary in Lipschitz domains with a sufficiently small constant. This result matches corresponding ones in the homogeneous case (see \cite[Proposition 2.6.15]{FeRo24}, \cite[Lemma 5.5]{RoWe23}) and is optimal. Moreover, it will be a central ingredient in our proof of the $C^s$ regularity in more regular domains.

\begin{theorem}
\label{thm:C-s-eps}
Let $L \in \mathcal{L}_s^n(\lambda,\Lambda)$. Let $\eps > 0$, $\delta > 0$, and $\Omega \subset \R^n$ be a Lipschitz domain with Lipschitz constant less than $\delta$. Let $f \in L^{\infty}(B_1 \cap \Omega)$ and $u \in C(B_1) \cap L^1_{2s}(\R^n)$ be a distributional solution of 
\begin{align*}
\begin{cases}
L u &= f ~~ \text{ in } B_1 \cap \Omega,\\
u &= 0 ~~ \text{ in } B_1 \setminus \Omega.\\
\end{cases}
\end{align*}
Then, there are $C, \delta_0 > 0$, depending only on $n,s,\lambda,\Lambda,\eps$, such that if $\delta \le \delta_0$, then
\begin{align*}
\Vert u \Vert_{C^{s-\eps}(B_{1/2} \cap \overline{\Omega})} \le C \left(\Vert u \Vert_{L^1_{2s}(\R^n)} + \Vert f \Vert_{L^{\infty}(B_1 \cap \Omega)} \right).
\end{align*}
\end{theorem}

The proof goes by compactness, and crucially uses the previously established Liouville theorem in the half-space (see \autoref{thm:Liouville-half-space}) and also \autoref{lemma:Lipschitz-Holder-prelim} (with $g \equiv 0$).

\begin{proof}
After a standard truncation argument, using that $L \in \mathcal{L}_s^n(\lambda,\Lambda)$, and by normalization, we can assume that $u \in L^{\infty}(\R^n)$. Note that in the light of the existing interior regularity results (see \cite[Theorem 2.4.3]{FeRo24}), it suffices to prove that there exists a constant $C > 0$, depending only on $n,s,\lambda,\Lambda,\eps$, such that for any $z \in \partial \Omega \cap B_{1/2}$ it holds
\begin{align}
\label{eq:Cs-eps-help-1}
|u(x)| \le C |x-z|^{s-\eps} \left(\Vert u \Vert_{L^{\infty}(\R^n)} + \Vert f \Vert_{L^{\infty}(B_1 \cap \Omega)} \right) ~~ \forall x \in B_{1/2} \cap \Omega.
\end{align}
We assume that the desired result does not hold, i.e., that there exist sequences $(\delta_k)_k$ with $\delta_k \searrow 0$, $(\Omega_k)_k$ Lipschitz domains with Lipschitz constant less than $\delta_k$, $0 \in \partial \Omega_k$, $(L_k)_k \subset \mathcal{L}_s^n(\lambda,\Lambda)$, $f_k \in L^{\infty}(B_1 \cap \Omega_k)$, $u_k \in C(\R) \cap L^{\infty}(\R^n)$ such that
\begin{align}
\label{eq:blowup-1-unif-bd}
\Vert u_k \Vert_{L^{\infty}(\R^n)} + \Vert f_k \Vert_{L^{\infty}(B_1 \cap \Omega_k)} \le 1,
\end{align} 
where $u_k$ are solutions to
\begin{align*}
\begin{cases}
L_k u_k &= f_k ~~\text{ in } B_1 \cap \Omega_k,\\
u_k &= 0 ~~~~ \text{ in } B_1 \setminus \Omega_k,\\
\end{cases}
\end{align*}
but it holds
\begin{align*}
\sup_{k \in \N} \sup_{r > 0} \frac{\Vert u_k \Vert_{L^{\infty}(B_{r/2})}}{r^{s-\eps}} = \infty.
\end{align*}
Then, by \cite[Lemma 4.4.11]{FeRo24}, there exist subsequences $r_m \searrow 0$ and $(u_{k_m})_m$ such that,
\begin{align}
\label{eq:blow-up-1-nondeg}
\Vert u_{k_m} \Vert_{L^{\infty}(B_{r_m})} \ge C r_m^{s-\eps},
\end{align}
and moreover the functions
\begin{align}
\label{eq:blow-up-1-vm-growth}
v_m(x) := \frac{u_{k_m}(r_m x)}{\Vert u_{k_m} \Vert_{L^{\infty}(B_{r_m})}} \quad \text{ satisfy } \quad |v_m(x)| \le 2(1 + |x|^{s-\eps}) ~~ \forall x \in \R^n, \qquad \Vert v_m \Vert_{L^{\infty}(B_1)} = 1.
\end{align}
Note that 
\begin{align*}
\begin{cases}
\tilde{L}_{k_m} v_m &= \tilde{f}_m ~~ \text{ in } B_{r_m^{-1}} \cap r_m^{-1} \Omega_{k_m},\\
v_m &= 0 ~~~~ \text{ in } B_{r_m^{-1}} \setminus r_m^{-1} \Omega_{k_m},
\end{cases}
\end{align*}
where $\tilde{L}_{k_m} \in \mathcal{L}_n^s(\lambda,\Lambda)$ is the operator with kernel $\tilde{K}_{k_m}(h) = r_m^{n+2s} K_{k_m}(r_m h)$, and
\begin{align}
\label{eq:blow-up-1-fm-bounded}
\tilde{f}_m := \frac{r_m^{2s} f_{k_m}(r_m \cdot)}{\Vert u_{k_m} \Vert_{L^{\infty}(B_{r_m})}}, \quad \Vert \tilde{f}_m \Vert_{L^{\infty} (B_{r_m^{-1}} \cap r_m^{-1} \Omega_{k_m})} \le \Vert f_{k_m} \Vert_{L^{\infty}(B_1 \cap \Omega_{k_m})} C^{-1} r_m^{s+\eps} \to 0, ~~ \text{ as } m \to \infty,
\end{align}
where we used \eqref{eq:blow-up-1-nondeg} and \eqref{eq:blowup-1-unif-bd}. Altogether, by \eqref{eq:blow-up-1-vm-growth} and \eqref{eq:blow-up-1-fm-bounded}, together with \autoref{lemma:Lipschitz-Holder-prelim} (with $g \equiv 0$), the sequence $(v_m)_m$ is uniformly bounded in $C^{\alpha}(B_{r_m^{-1}/2} \cap r_m^{-1} \Omega_{k_m})$, and therefore by the Arzel\`a-Ascoli theorem, it converges locally uniformly to some $v \in C(\R^n) \cap L^{1}_{2s}(\R^n)$ satisfying
\begin{align}
\label{eq:blow-up-1-limit-properties}
|v(x)| \le 2(1 + |x|^{s-\eps}) ~~ \forall x \in \R^n, \qquad \Vert v \Vert_{L^{\infty}(B_1)} = 1.
\end{align}
Moreover, using that by $\delta \searrow 0$,  we have $B_{r_m^{-1}} \cap r_m^{-1} \Omega_{k_m} \to \{x \cdot e > 0\}$ for some $e \in \mathbb{S}^{n-1}$, and again \eqref{eq:blow-up-1-fm-bounded}, and that due to \eqref{eq:blow-up-1-vm-growth} it also holds $v_m \to v$ in $L^1_{2s}(\R^n)$, by the stability for distributional solutions (see \cite[Proposition 2.2.36]{FeRo24}), there exists $L_{\infty} \in \mathcal{L}_s^n(\lambda,\Lambda)$ such that
\begin{align*}
\begin{cases}
L_{\infty} v &= 0 ~~ \text{ in } \{ x \cdot e > 0 \},\\
v &= 0 ~~ \text{ in } \{ x \cdot e \le 0 \}.
\end{cases}
\end{align*}
Thus, by the Liouville theorem in the half-space (see \autoref{thm:Liouville-half-space}(i)), and the first property in \eqref{eq:blow-up-1-limit-properties}, we have that $v \equiv 0$, which contradicts the last property in \eqref{eq:blow-up-1-limit-properties}. Therefore, it must hold \eqref{eq:Cs-eps-help-1}, and the proof is complete.
\end{proof}

\subsection{$C^s$ regularity in $C^{1,\alpha}$ domains}

In this section we prove the $C^s$ boundary regularity in $C^{1,\alpha}$ domains, which is contained in the following theorem. Note that it immediately implies \autoref{thm-expansion-intro} and the first part of \autoref{thm-Cs-intro}.

\begin{theorem}
\label{thm:inhom-Cs}
Let $L \in \mathcal{L}_s^n(\lambda,\Lambda)$. Let $\Omega \subset \R^n$ be a bounded domain with $\partial \Omega \in C^{1,\alpha}$ for some $\alpha \in (0,s)$. Let $f \in L^{\infty}(B_1 \cap \Omega)$ and $u \in C(B_1) \cap L^1_{2s}(\R^n)$ be a distributional solution of 
\begin{align*}
\begin{cases}
L u &= f ~~ \text{ in } B_1 \cap \Omega,\\
u &= 0 ~~ \text{ in } B_1 \setminus \Omega.\\
\end{cases}
\end{align*}
Then, there is $C_1 > 0$, depending only on $n,s,\lambda,\Lambda,\Omega,\alpha$, such that
\begin{align}
\label{eq:inhom-Cs-estimate}
\Vert u \Vert_{C^{s}(B_{1/4} \cap \overline{\Omega})} \le C_1 \left(\Vert u \Vert_{L^1_{2s}(\R^n)} + \Vert f \Vert_{L^{\infty}(B_1 \cap \Omega)} \right).
\end{align}
Moreover, for any $\eps \in (0,\alpha s)$, there exist $C_2, C_3$, depending only on, $n,s,\lambda,\Lambda,\Omega,\alpha,\eps$, such that for any $z \in B_{1/2} \cap \partial \Omega$ there exists $q_z \in \R$ with $|q_z| \le C_2 \left(\Vert u \Vert_{L^1_{2s}(\R^n)} + \Vert f \Vert_{L^{\infty}(B_1 \cap \Omega)} \right)$ such that
\begin{align}
\label{eq:inhom-Cs-expansion}
|u(x) - q_z b_{\nu_z}((x-z) \cdot \nu_z)| \le C_3 |x-z|^{s+\eps} \left(\Vert u \Vert_{L^1_{2s}(\R^n)} + \Vert f \Vert_{L^{\infty}(B_1 \cap \Omega)} \right) ~~ \forall x \in B_{1/2},
\end{align}
where $\nu_z \in \mathbb{S}^{n-1}$ denotes the inner normal vector at $z$, and $b_{\nu_z}$ is the half-space solution from \autoref{thm-1D} corresponding to $L$ and $\nu_z$.
\end{theorem}

Having at hand the higher order Liouville theorem in the half-space (see \autoref{thm:Liouville-half-space}(ii)) and the 1D barriers from \autoref{thm-1D}, we can prove this result by adapting the arguments in \cite{RoSe16b} to our setting. Note that some adjustments are also required since we consider $C^{1,\alpha}$ domains instead of $C^{1,1}$ domains as in \cite{RoSe16b}. For instance, we need to make use of the almost optimal regularity obtained in \autoref{thm:C-s-eps}.
Moreover, we will need \autoref{lemma:Lipschitz-Holder-prelim} with nonzero exterior data $g$.

\begin{proof}[Proof of \autoref{thm:inhom-Cs}]
First, we prove that the expansion \eqref{eq:inhom-Cs-expansion} implies the $C^s$ estimate \eqref{eq:inhom-Cs-estimate}. First, note that by a standard truncation argument, we can assume that $u \in L^{\infty}(\R^n)$. Moreover, after a normalization, we might assume $\Vert u \Vert_{L^{\infty}(\R^n)} + \Vert f \Vert_{L^{\infty}(B_1 \cap \Omega)} \le 1$. Then, \eqref{eq:inhom-Cs-estimate} follows, once we show 
\begin{align}
\label{eq:inhom-Cs-help-1}
|u(x)| \le C d^s(x) ~~ \forall x \in B_{1/4} \cap \Omega.
\end{align}
Let us take $x \in B_{1/4} \cap \Omega$ and let $z \in B_{1/2} \cap \partial \Omega$ be such that $|x-z| = d(x)$. Clearly, when such $z$ does not exist, it must be $d(x) \ge 1/4$, in which case \eqref{eq:inhom-Cs-help-1} is trivially satisfied for $x$. Thus, we can assume that such $z$ exists. Then, by application of \eqref{eq:inhom-Cs-expansion}, we deduce that there exists $q_z \in \R$ such that
\begin{align*}
|u(x)| &\le |q_z b_{\nu_z}((x-z) \cdot \nu_z)| + |u(x) - q_z b_{\nu_z}((x-z) \cdot \nu_z)| \\
&\le C \Vert q_{\cdot} \Vert_{L^{\infty}(B_{1/2} \cap \partial \Omega)} |x-z|^s + C |x-z|^{s+\eps} \le C |x-z|^s,
\end{align*}
where we used the pointwise estimates for $b_{\nu_z}$ from \autoref{thm-1D} and the fact that $|q_z| \le C_2$. Let us explain, how to deduce the boundedness of $(z \mapsto q_z)$ from \eqref{eq:inhom-Cs-expansion}. In fact, to see this, note that there exists $c_0 > 0$, depending only on the $C^{1,\alpha}$ radius of $\Omega$ such that for any $z \in B_{1/2} \cap \partial \Omega$, we can find $x \in B_{1/2} \cap \Omega$ such that $d(x) = |x-z| = (x-z) \cdot \nu_z \ge c_0$. Then, using \eqref{eq:inhom-Cs-expansion}, and again the pointwise estimates for $b_{\nu_z}$ from \autoref{thm-1D}, we deduce
\begin{align*}
|q_z| &\le c_0^{-s} |q_z ((x-z)\cdot \nu_z)_+^s| \le c |q_z b_{\nu_z}((x-z)\cdot \nu_z)| \\
&\le c |u(x)| + c |u(x) - q_z b_{\nu_z}((x-z)\cdot \nu_z)| \le c + c |x-z|^{s+\eps} \le c.
\end{align*}
This yields boundedness of $(z \mapsto q_z)$, and the proof of \eqref{eq:inhom-Cs-estimate} is complete.

Now, we turn to the proof of the expansion \eqref{eq:inhom-Cs-expansion}.  We prove it by a contradiction compactness argument, as in \cite{RoSe16b}. Let us assume that there exist sequences $(\Omega_k)_k$ with $\partial \Omega_k \in C^{1,\alpha}$, and $C^{1,\alpha}$ radius bounded by one, $(L_k)_k \subset \mathcal{L}_s^n(\lambda,\Lambda)$, $(f_k)_k$, and $(u_k)_k \subset C(\R^n) \cap L^{\infty}(\R^n)$ with 
\begin{align*}
\Vert f_k \Vert_{L^{\infty}(B_1 \cap \Omega_k)} + \Vert u_k \Vert_{L^{\infty}(\R^n)} = 1,
\end{align*}
such that
\begin{align*}
\begin{cases}
L_k u_k &= f_k ~~ \text{ in } B_1 \cap \Omega_k,\\
u_k &= 0 ~~~~ \text{ in } B_1 \setminus \Omega_k,
\end{cases}
\end{align*}
but assume by contradiction that 
\begin{align*}
\sup_{k \in \N} \sup_{z \in B_{1/2} \cap \partial \Omega_k} \sup_{r > 0} r^{-s-\eps} \left\Vert u_k - Q b^{(k)}_{\nu^{(k)}_z} ( (x - z) \cdot \nu^{(k)}_z ) \right\Vert_{L^{\infty}(B_r(z))} = \infty ~~ \forall Q \in \R.
\end{align*}
Here, $\nu^{(k)}_z \in \mathbb{S}^{n-1}$ denotes the inner unit normal vector to $\partial \Omega_k$ at $z$, and $b^{(k)}_{\nu^{(k)}_z}$ denotes the 1D barrier in the half-space from \autoref{thm-1D} corresponding to $L_k$ and $\nu^{(k)}_{z_k}$.
Let us define
\begin{align*}
\theta(r) := \sup_{k \in \N} \sup_{z \in B_{1/2} \cap \partial \Omega_k} \sup_{\rho \ge r} \rho^{-s-\eps} \Vert u_k - \phi_{k,z,\rho} \Vert_{L^{\infty}(B_{\rho}(z))},
\end{align*}
where we set
\begin{align*}
\phi_{k,z,r}(x) = Q_{k,z,r} ~ b^{(k)}_{\nu^{(k)}_z} ( (x - z) \cdot \nu^{(k)}_z ), \qquad Q_{k,z,r} = \frac{\int_{B_r(z)} \left[ u_k(x) b^{(k)}_{\nu^{(k)}_z} ( (x-z) \cdot \nu^{(k)}_z ) \right] \d x }{\int_{B_r(z)} b^{(k)}_{\nu^{(k)}_z} ( (x-z) \cdot \nu^{(k)}_z )^2 \d x}.
\end{align*}
Note that $Q_{k,z,r}$ is the $L^2(B_r(z))$-projection of $u_k$ over $b^{(k)}_{\nu^{(k)}_z}((\cdot - z) \cdot \nu^{(k)}_z)\R$. Clearly, by the normalization assumption on $u_k$ we have $Q_{k,r,z} \le C r^n$. It follows by the exact same arguments as in \cite[Lemma 5.3]{RoSe16b} that $\theta(r) \nearrow \infty$, as $r \searrow 0$. Moreover, there are sequences $r_m \searrow 0$, $(k_m)_m$, and $z_m \to z \in \overline{B_{1/2}}$, $\nu_{z_m}^{(k_m)} \to \nu \in \mathbb{S}^{n-1}$, such that
\begin{align*}
r_m^{-s-\eps} \Vert u_{k_m} - \phi_{k_m,z_m,r_m} \Vert_{L^{\infty}(B_{r_m}(z_m))}  \ge \theta(r_m)/2 ~~ \forall m \in \N.
\end{align*}
To keep the notation simple, let us denote from now on $\phi_m := \phi_{k_m,z_m,r_m}$, $b_m = b^{(k_m)}_{\nu_{z_m}^{(k_m)}}$, and $\nu_m := \nu_{z_m}^{(k_m)}$. We define
\begin{align*}
v_m(x) = \frac{u_{k_m}(z_m + r_m x) - \phi_m(z_m + r_m x)}{r_m^{s+\eps} \theta(r_m)},
\end{align*}
and observe that by construction, we have
\begin{align}
\label{eq:blow-up-2-nondeg}
\int_{B_1} v_m(x) b_m(x \cdot \nu_m) \d x = 0, \qquad \Vert v_m \Vert_{L^{\infty}(B_{1/2})} \ge \frac{1}{2} ~~ \forall m \in \N.
\end{align}
Now, using the pointwise bounds on the 1D barriers $b_m$ from \autoref{thm-1D}, we get for any $k,z,r$:
\begin{align*}
|Q_{k,z,2r} - Q_{k,z,r}| &\le c r^{-s} \Vert \phi_{k,z,r} - \phi_{k,z,2r} \Vert_{L^{\infty}\big(B_r(z) \cap \big\{b_{\nu_z^{(k)}} \ge \frac{r}{2} \big\} \big)} \\
&\le c r^{-s} \Vert u_k - \phi_{k,z,r} \Vert_{L^{\infty}(B_r(z))} + c r^{-s} \Vert u_k - \phi_{k,z,2r} \Vert_{L^{\infty}(B_{2r}(z))} \\
&\le c \theta(r) r^{\eps} + c \theta(2r) r^{\eps} \le c \theta(r) r^{\eps}.
\end{align*}
Thus, by the exact same arguments as in the proof of \cite[Proof of Proposition 5.2]{RoSe16b} or \cite[Proof of Proposition 4.1]{AbRo20}, using only the definition and monotonicity of $\theta(r)$, we can deduce from here
\begin{align}
\label{eq:blow-up-2-vm-growth}
\Vert v_m \Vert_{L^{\infty}(B_R)} \le C R^{s+\eps} ~~ \forall R \ge 1 ~~ \forall m \in \N.
\end{align}
Next, using this information, we investigate the equation that is satisfied by $v_m$ and intend to prove locally uniform convergence, and convergence in $L^1_{2s}(\R^n)$ of $(v_m)_m$, up to a subsequence, in order to pass this equation to the limit. First, given $R > 1$ and $m \in \N$ with $r_m R < 1/2$, let us define the half-spaces
\begin{align*}
\Sigma_m^+ := \{ (x - z_m) \cdot \nu_m > 0 \}, ~~ \Sigma_m^- := \{ (x - z_m) \cdot \nu_m < 0 \},
\end{align*}
and the set
\begin{align*}
\Omega_{R,k_m} = \left\{ x \in B_R : (z_m + r_m x) \in \Omega_{k_m}, ~~ \text{ and } x \cdot \nu_m > 0 \right\}.
\end{align*}
We observe that $0 \in \partial \Omega_{R,k_m}$, and $\Omega_{R,k_m} \to B_R \cap \{ x \cdot \nu > 0 \}$. Moreover, let us denote by $\tilde{L}_{k_m} \in \mathcal{L}_n^s(\lambda,\Lambda)$ the operator with kernel $\tilde{K}_{k_m}(h) = r_m^{n+2s} K_{k_m}(r_m h)$. Then, since $L_{k_m} \phi_m = 0$ in $\Sigma_m^{+}$, we have
\begin{align}
\label{eq:blow-up-2-PDE}
\tilde{L}_{k_m} v_m(x) = \tilde{f}_m(x) ~~ \text{ in } \Omega_{R,k_m},
\end{align}
where, since we took $\eps < s$:
\begin{align}
\label{eq:blow-up-2-RHS}
\tilde{f}_m(x) = r_m^{s-\eps}\frac{f_{k_m}(z_m + r_m x)}{\theta(r_m)} \to 0, ~~ \text{ as } m \to \infty ~~ \text{ uniformly in } \Omega_{R,k_m}.
\end{align}
Next, we claim that there exists $\delta \in (0,1)$ such that for all $R > 1$ with $r_m R < 1/4$ it holds
\begin{align}
\label{eq:blow-up-2-unif-bd}
\Vert v_m \Vert_{C^{\delta}(B_{R})} \le C(R) ~~ \forall m \in \N
\end{align}
for some constant $C(R) > 0$, depending on $n,s,\lambda,\Lambda,R$, but not on $m$.
We show such bound separately for the sets $B_R = (B_R \setminus \Omega_{R,k_m}) \cup \Omega_{R,k_m}$. 
First, we prove that for some $\delta_1 \in (0,1)$ it holds
\begin{align}
\label{eq:blow-up-2-unif-bd-1}
\Vert v_m \Vert_{C^{\delta_1}(B_{R} \setminus \Omega_{R,k_m})} \le C(R) ~~ \forall m \in \N.
\end{align}
To see it, we first deduce from the pointwise bounds on $b_m$ in \autoref{thm-1D} and the almost optimal regularity from \autoref{thm:C-s-eps} that for any $r < 1/4$ and any $\gamma \in (0,s)$:
\begin{align*}
|u_{k_m}(x)| &\le c \dist(x,\partial \Omega_{k_m})^{s-\gamma} ~~ \forall x \in B_r(z_m) \cap \Sigma_m^-,\\
|\phi_m(x)| &\le c b_m((x-z_m) \cdot \nu_z^{(k)}) = c b_m(\dist(x,\Sigma_m^-)) \le c \dist(x,\Sigma_m^-)^s ~~  \forall x \in B_r(z_m) \setminus \Omega_{k_m}.
\end{align*}
Note that since by assumption $\partial \Omega_k \in C^{1,\alpha}$, we have the following control over how $\partial \Omega_{k_m}$ separates from $\Sigma_m^-$: It holds
\begin{align*}
\dist(x,\partial \Omega_{k_m}) + \dist(x, \Sigma_m^-) \le C r^{1+\alpha} ~~ \forall x \in (B_r(z_m) \setminus \Omega_{k_m}) \cup (B_r(z_m) \cap \Sigma_m^-). 
\end{align*}
Thus, by combination of the previous three estimates, and upon choosing $\gamma \in (0,s)$ so small that $\zeta := (1+\alpha)(s-\gamma) > s + \eps$ (which is possible since $\eps < \alpha s$) we deduce
\begin{align*}
\Vert u_{k_m} - \phi_m & \Vert_{L^{\infty}((B_r(z_m) \setminus \Omega_{k_m}) \cup (B_r(z_m) \cap \Sigma_m^-))} \\
&\le \Vert u_{k_m} \Vert_{L^{\infty}((B_r(z_m) \setminus \Omega_{k_m}) \cup (B_r(z_m) \cap \Sigma_m^-))} + \Vert \phi_m \Vert_{L^{\infty}((B_r(z_m) \setminus \Omega_{k_m}) \cup (B_r(z_m) \cap \Sigma_m^-))} \\
&\le C r^{(1+\alpha)(s-\gamma)} + C r^{(1+\alpha)s} \le C r^{\zeta}.
\end{align*}
Next, we deduce from the almost optimal regularity (see \autoref{thm:C-s-eps}) and also the regularity of the barriers in \autoref{thm-1D} that for any $r < 1/4$ and $m \in \N$:
\begin{align*}
\Vert u_{k_m} - \phi_m & \Vert_{C^{s-\gamma}((B_r(z_m) \setminus \Omega_{k_m}) \cup (B_r(z_m) \cap \Sigma_m^-))} \\
&\le \Vert u_{k_m}\Vert_{C^{s-\gamma}((B_{1/4}(z_m) \setminus \Omega_{k_m}) \cup (B_{1/4}(z_m) \cap \Sigma_m^-))} + \Vert \phi_m \Vert_{C^{s-\gamma}((B_{1/4}(z_m) \setminus \Omega_{k_m}) \cup (B_{1/4}(z_m) \cap \Sigma_m^-))} \le C.
\end{align*}
Hence, interpolating the previous two estimates, using that $\zeta > s+\eps$, we can find $\delta_1 \in (0,1)$ such that
\begin{align*}
\Vert u_{k_m} - \phi_m \Vert_{C^{\delta_1}((B_r(z_m) \setminus \Omega_{k_m}) \cup (B_r(z_m) \cap \Sigma_m^-))} \le C r^{s+\eps}.
\end{align*}
After choosing $r := r_m R$ and rescaling the latter estimate, this proves \eqref{eq:blow-up-2-unif-bd-1}. Finally, in the light of \eqref{eq:blow-up-2-PDE}, \eqref{eq:blow-up-2-RHS}, and using also \eqref{eq:blow-up-2-unif-bd-1}, we immediately deduce from a rescaled version of \autoref{lemma:Lipschitz-Holder-prelim} (applied with $\Omega := B_R \setminus \Omega_{R,k_m}$, $g:= v_m$, $f = \tilde{f}$)
\begin{align}
\label{eq:blow-up-2-unif-bd-2}
\Vert v_m \Vert_{C^{\delta}(\Omega_{R,k_m})} \le C(R) ~~ \forall m \in \N
\end{align}
for some $\delta \in (0,1)$, and $C(R)$, independent of $m$. Combining \eqref{eq:blow-up-2-unif-bd-1} and \eqref{eq:blow-up-2-unif-bd-2}, we obtain \eqref{eq:blow-up-2-unif-bd}.

Altogether, by combination of \eqref{eq:blow-up-2-vm-growth} and \eqref{eq:blow-up-2-unif-bd} and the Arzel\`a-Ascoli theorem, we deduce that $v_m \to v \in C(\R) \cap L^1_{2s}(\R^n)$ locally uniformly and in $L^1_{2s}(\R^n)$, up to a subsequence. Thus, using also \eqref{eq:blow-up-2-PDE} and \eqref{eq:blow-up-2-RHS}, we can apply the stability lemma (see \cite[Proposition 2.2.36]{FeRo24}) to deduce that there exists $L_{\infty} \in \mathcal{L}_s^n(\lambda,\Lambda)$ such that
\begin{align*}
\begin{cases}
L_{\infty} v &= 0 ~~~~ \qquad\qquad \text{ in } \{ x \cdot \nu > 0 \},\\
v &= 0 ~~~~ \qquad\qquad \text{ in } \{ x \cdot \nu \le 0\},\\
|v(x)| & \le C(1 + |x|^{s+\eps}) ~~ \forall x \in \R^n.
\end{cases}
\end{align*}
Thus, by the Liouville theorem in the half-space (see \autoref{thm:Liouville-half-space}), it follows 
\begin{align*}
v(x) = \kappa b(x) ~~ \forall x \in \R^n,
\end{align*}
where $\kappa \in \R$, and $b(x) = \tilde{b}(x \cdot \nu)$ is the half-space solution from \autoref{thm-1D} corresponding to $L_{\infty}$ and $\nu$. Note that by \autoref{thm-1D}, resp. \autoref{lemma:1d-barrier}, the sequence $(b_m)_m$ is uniformly bounded in $C^s_{loc}(\R^n) \cap L^1_{2s}(\R^n)$. Note that here, we slightly abuse notation and write $b_m(x) := b_m(x \cdot \nu_m)$. Therefore, again by the Arzel\`a-Ascoli theorem, up to a subsequence, it holds $b_m \to \bar{b} \in C^s(\R^n) \cap L^1_{2s}(\R^n)$ locally uniformly and in $L^1_{2s}(\R^n)$. Moreover, by the stability lemma (see \cite[Proposition 2.2.36]{FeRo24}),
\begin{align*}
\begin{cases}
L_{\infty} \bar{b} &= 0 ~~ \text{ in } \{ x \cdot \nu > 0 \},\\
\bar{b} &= 0 ~~ \text{ in } \{ x \cdot \nu \le 0\}.
\end{cases}
\end{align*}
Moreover, there must exist $0 < c_1 < c_2 < \infty$ such that
\begin{align*}
c_1 (x \cdot \nu)_+^s \le |\bar{b}(x)| \le c_2 (x \cdot \nu)_+^s ~~ \forall x \in \R^n.
\end{align*}
Thus, applying the Liouville theorem in the half-space (see \autoref{thm:Liouville-half-space}) to $\bar{b}$, we find that there exists $\bar{\kappa} \in [c_1,c_2]$ such that
\begin{align*}
\bar{b}(x) = \bar{\kappa} b(x) = \bar{\kappa} \tilde{b}(x \cdot \nu) ~~ \forall x \in \R^n.
\end{align*} 
Hence, passing the first statement in \eqref{eq:blow-up-2-nondeg} to the limit, and plugging in the representation for $v$,
\begin{align*}
0 = \lim_{m \to \infty} \int_{B_1} v_m(x) b_m(x \cdot \nu_m) \d x = \int_{B_1} v(x) \bar{b}(x \cdot \nu) \d x = \kappa \bar{\kappa}^{-1} \int_{B_1} \bar{b}(x \cdot \nu) \d x.
\end{align*}
Therefore, it must be $\kappa = 0$, i.e., $v \equiv 0$, which contradicts the second statement from \eqref{eq:blow-up-2-nondeg}, since it implies
\begin{align*}
\Vert v \Vert_{L^{\infty}(B_{1/2})} = \lim_{m \to \infty} \Vert v_m \Vert_{L^{\infty}(B_{1/2})} \ge \frac{1}{2}.
\end{align*}
Thus, the expansion in \eqref{eq:inhom-Cs-expansion} must hold true, and the proof is complete.
\end{proof}

\begin{proof}[Proof of \autoref{thm-expansion-intro} and of the first part of \autoref{thm-Cs-intro}]
This is an immediate consequence of \autoref{thm:inhom-Cs}.
\end{proof}

\subsection{Hopf lemma}

Having at hand the expansion in \autoref{thm:inhom-Cs} (see \eqref{eq:inhom-Cs-expansion}) we can also prove a Hopf lemma for nonlocal equations with inhomogeneous kernels, proving the second part of our main result \autoref{thm-Cs-intro}.\\
The idea of the proof goes as follows: First, we observe that it suffices to prove $q_z > 0$ in order for the Hopf lemma to follow. Then, assume that $q_z = 0$ and observe that in this case \eqref{eq:inhom-Cs-expansion} implies that $u(x) = O(|x|^{s+\eps})$. We can use this information to show that $u(x) = O(|x|^{2s-\gamma})$ for any $\gamma \in (0,s)$ by a compactness argument. However, this is a contradiction with \cite[Lemma 4.3.4]{FeRo24}, where a barrier was constructed, ruling out such boundary behavior of positive solutions.

\begin{theorem}[Hopf lemma]
\label{thm:inhom-Hopf}
Let $L \in \mathcal{L}_s^n(\lambda,\Lambda)$. Let $\Omega \subset \R^n$ be a bounded domain with $\partial \Omega \in C^{1,\alpha}$ for some $\alpha \in (0,s)$. Let $f \in L^{\infty}(B_1 \cap \Omega)$ and $u \in C(\overline{B_1}) \cap L^1_{2s}(\R^n)$ be a distributional solution of 
\begin{align*}
\begin{cases}
L u &= f \ge 0  ~~ \text{ in } B_1 \cap \Omega,\\
u &\ge 0 ~~\qquad \text{ in }\R^n \setminus (B_1 \cap \Omega).\\
\end{cases}
\end{align*}
Then, either $u \equiv 0$ in $\Omega$, or for some $c > 0$
\begin{align*}
u(x) \ge c d^s(x) ~~ \text{ in } B_{1/2}.
\end{align*} 
\end{theorem}

\begin{proof}
First, by the strong maximum principle (see \cite[Theorem 2.4.15]{FeRo24}) we deduce that in case $u \not \equiv 0$ in $\Omega$, it must be $u > 0$ in $\Omega$. Let us assume this property for the remainder of the proof. 

Step 1: We prove that it suffices to show the Hopf lemma in case $u \equiv 0$ in $\R^n \setminus (B_1 \cap \Omega)$. Indeed, if $Lu = Lv = f \ge 0$ in $B_1 \cap \Omega$ and $v \ge 0 = u$ in $\R^n \setminus (B_1 \cap \Omega)$, then by the comparison principle, we have $v \ge u$ in $\R^n$, and the result for $v$ follows immediately from the result for $u$.

Step 2: By application of \autoref{thm:inhom-Cs} to $u$, we deduce that there exists $\eps > 0$ such that for any $z \in B_{1/2} \cap \Omega$ there exists $q_z \ge 0$ such that
\begin{align}
\label{eq:expansion-Hopf}
\begin{split}
|u(x) - q_z b_{\nu_z}((x-z) \cdot \nu_z)| &\le C_1 \left( \Vert u \Vert_{L^1_{2s}(\R^n)} + \Vert f \Vert_{L^{\infty}(B_1 \cap \Omega)} \right) |x-z|^{s+\eps} \\
&\le C_2 \left(\Vert u \Vert_{L^{\infty}(B_1)} + \Vert f \Vert_{L^{\infty}(B_1 \cap \Omega)} \right) |x-z|^{s+\eps} ~~ \forall x \in B_{1/2}(z),
\end{split}
\end{align}
where we used in that last step the tail estimate from \cite[Theorem 1.9]{KaWe23}. \\
We will prove that $\inf_{z \in B_{1/2} \cap \partial \Omega} q_z > 0$. We will show this property in three consecutive steps.

Step 2a: We claim that any function $u$ satisfying the assumptions of this theorem and an expansion of the form \eqref{eq:expansion-Hopf} with $q_z = 0$ must satisfy
\begin{align}
\label{eq:expansion-Hopf-2}
|u(x)| \le C_3 \left( \Vert u \Vert_{L^{\infty}(B_1)} + \Vert f \Vert_{L^{\infty}(B_1 \cap \Omega)} \right) |x-z|^{2s-\gamma} ~~ \forall x \in B_{1/2}
\end{align}
for any $\gamma \in (0,s)$, where $C_3 > 0$ depends on $n,s,\lambda,\Lambda,\alpha,\gamma$, and the $C^{1,\alpha}$ radius of $\Omega$.
To prove \eqref{eq:expansion-Hopf-2} let us assume by contradiction that there exists a sequence of sets $(\Omega_k)_k$ with $\partial \Omega_k \in C^{1,\alpha}$, $0 \in \partial\Omega_k$, operators $(L_k)_k \subset \mathcal{L}_s^n(\lambda,\Lambda)$, $(f_k)_k \subset B_1 \cap \Omega_k$, $(u_k)_k$ with $u_k \not\equiv 0$ (i.e., $u_k > 0$ in $B_1 \cap \Omega_k$) solutions to
\begin{align*}
\begin{cases}
L_k u_k &= f_k \ge 0 ~~ \text{ in } B_1 \cap \Omega_k,\\
u_k &= 0 \qquad ~~~ \text{ in }\R^n \setminus (B_1 \cap \Omega_k)\\
\end{cases}
\end{align*}
with 
\begin{align}
\label{eq:q0-uk}
\Vert u_k \Vert_{L^{\infty}(B_1)} + \Vert f_k \Vert_{L^{\infty}(B_1 \cap \Omega_k)} = 1, ~~ \text{ and } ~~ |u_k(x)| \le C_2 |x|^{s+\eps} ~~ \forall x \in B_{1/2}, 
\end{align}
but such that \eqref{eq:expansion-Hopf-2} fails, i.e.,
\begin{align*}
\sup_{k \in \N} \sup_{r > 0} \frac{\Vert u_k \Vert_{L^{\infty}(B_{r})}}{r^{2s-\gamma}} = \infty.
\end{align*}
Then, by \cite[Lemma 4.4.11]{FeRo24} there exist sequences $(k_m)_m$ and $r_m \to 0$ such that
\begin{align}
\label{eq:vm-growth-Hopf-1}
v_m(x) = \frac{u_{k_m}(r_m x)}{\Vert u_{k_m} \Vert_{L^{\infty}(B_{r_m})}} ~~ \text{ satisfies } ~~ \Vert v_m \Vert_{L^{\infty}(B_1)} = 1, ~~ |v_m(x)| \le C (1 + |x|^{2s-\gamma}) ~~ \forall x \in \R^n.
\end{align}
Moreover, let us denote by $\tilde{L}_{k_m} \in \mathcal{L}_n^s(\lambda,\Lambda)$ the operator with kernel $\tilde{K}_{k_m}(h) = r_m^{n+2s}K_{k_m}(r_m h)$, and note that 
\begin{align*}
\begin{cases}
\tilde{L}_{k_m} v_m &= \tilde{f}_{k_m} ~~ \text{ in } B_{r_m^{-1}} \cap r_m^{-1} \Omega_{k_m},\\
v_m &= 0 ~~~~~~ \text{ in } \R^n \setminus (B_{r_m^{-1}} \cap r_m^{-1} \Omega_{k_m}),\\
v_m &\ge 0 ~~~~~~ \text{ in }  \R^n,
\end{cases}
\end{align*}
where
\begin{align}
\label{eq:f-Hopf}
\tilde{f}_{k_m}(x) = r_m^{2s} \frac{f_{k_m}(r_m x)}{\Vert u_{k_m} \Vert_{L^{\infty}(B_{r_m})}}, ~~ \text{ satisfies } ~~ \Vert  \tilde{f}_{k_m} \Vert_{L^{\infty}( B_{r_m^{-1}} \cap r_m^{-1} \Omega_{k_m} )} \le c r_m^{\gamma} \Vert f_{k_m} \Vert_{L^{\infty}(B_1 \cap \Omega)} \le c r_m^{\gamma},
\end{align}
and we used that by \cite[Lemma 4.4.11]{FeRo24} there is $c > 0$ such that $\Vert u_{k_m} \Vert_{L^{\infty}(B_{r_m})} \ge c r_m^{2s-\gamma}$.
By application of \autoref{thm:inhom-Cs} (resp. \eqref{eq:expansion-Hopf}) to $v_m$, and using that $0 = q_0 = \lim_{x \to 0} u_{k_m}(x)/b_{\nu_0}(x \cdot \nu_0) = \lim_{x \to 0} v_m(x)/b_{\nu_0}(x \cdot \nu_0)$ for every $m \in \N$ due to \eqref{eq:q0-uk} and the definition of $v_m$ we deduce
\begin{align}
\label{eq:vm-growth-Hopf-2}
|v_m(x)| \le C_2 \left( \Vert v_m \Vert_{L^{\infty}(B_1)} + \Vert \tilde{f}_{k_m} \Vert_{L^{\infty}(B_1)} \right) |x|^{s+\eps} \le c C_2 (1 + r_m^{\gamma}) |x|^{s+\eps} \le C |x|^{s+\eps} ~~ \text{ in } B_{1/2}.
\end{align}
Moreover, note that by \autoref{thm:inhom-Cs} we have that $(v_m)_m$ is locally uniformly bounded in $C^s$, and therefore it converges locally uniformly and in $L^1_{2s}(\R^n)$ (by the Arzel\`a-Ascoli theorem) up to a subsequence to some $v \in C(\R^n) \cap L^1_{2s}(\R^n)$. Using the stability for distributional solutions (see \cite[Proposition 2.2.36]{FeRo24}) and \eqref{eq:f-Hopf}, \eqref{eq:vm-growth-Hopf-2}, we deduce that there are $e \in \mathbb{S}^{n-1}$ and $L_{\infty} \in \mathcal{L}_s^n(\lambda,\Lambda)$ such that
\begin{align*}
\begin{cases}
L_{\infty} v &= 0 ~~~~\qquad\qquad~~ \text{ in } \{ x \cdot e > 0 \},\\
v &= 0 ~~~~\qquad\qquad~~ \text{ in } \{ x \cdot e \le 0 \},\\
|v(x)| &\le C (1 + |x|)^{2s-\gamma} ~~ \forall x \in \R^n,\\
v &> 0 ~~~~\qquad\qquad~~ \text{ in } \{ x \cdot e > 0 \}.
\end{cases}
\end{align*}
However, such $v$ cannot exist. In fact, note that $v \not\equiv 0$ by \eqref{eq:vm-growth-Hopf-1}, but by the uniqueness of positive solutions in the half-space (see \cite[Proposition 4.4.5]{FeRo24}), it must be $v = \kappa b$ for some $\kappa > 0$, where $b$ is the half-space solution from \autoref{thm-1D} with respect to $L_{\infty}$ and $e$. Also this cannot be, since \eqref{eq:vm-growth-Hopf-2} implies that $v(x) \le C |x|^{s+\eps}$ in $B_{1/2}$. This is a contradiction, so \eqref{eq:expansion-Hopf-2} must hold true.

Step 2b: Having proved \eqref{eq:expansion-Hopf-2}, let us now show that $q_z > 0$ for every $z \in B_{1/2} \cap \partial \Omega$. By contradiction, we assume that there is $z \in B_{1/2} \cap \partial \Omega$ such that $q_z = 0$, i.e., we assume that \eqref{eq:expansion-Hopf-2} holds true at $z$. We recall the barrier function from \cite[Lemma B.3.1]{FeRo24}. Let us take $D \subset \R^n$ with $\partial D \in C^{1,\alpha}$ such that $B_{1/2} \cap \Omega \subset D \subset B_1 \cap \Omega$ and recall from  \cite[Lemma B.3.1]{FeRo24} that there exist $\gamma_0 > 0$ and $\rho > 0$ such that for any $\gamma_1 \in (0,\gamma_0)$
\begin{align*}
L(d_D^{2s-\gamma_1}) \le -1\le 0 \le L u ~~ \text{ in } D \cap \{ d_D < \rho \}.
\end{align*}
Moreover, since $u > 0$ in $B_1 \cap \Omega$, it must be $\inf_{D \cap \{ d_D \ge \rho \}} u > 0$. Hence, by the comparison principle, we deduce that
\begin{align*}
u \ge c \left(\inf_{D \cap \{ d_D \ge \rho \}} u \right) d_D^{2s-\gamma_1} ~~ \text{ in } D \supset B_{1/2} \cap \Omega
\end{align*}
for some $c > 0$. However, this is a contradiction with \eqref{eq:expansion-Hopf-2} if $\gamma_1 > \gamma$. Thus, it must hold $q_z > 0$.

Step 2c: We have shown that $q_z > 0$ for any $z \in B_{1/2} \cap \partial \Omega$. Now, we use this information to prove that $\inf_{z \in B_{1/2} \cap \partial \Omega} q_z > 0$. To see this, note that there exists $\kappa > 0$ such that for any $z \in B_{1/2} \cap \partial \Omega$, there exists $\delta_z > 0$ such that the following holds:
\begin{align}
\label{eq:boundary-cover}
\forall z' \in B_{\delta_z}(z) \cap \partial \Omega ~~ : ~~ q_{z'} \ge \kappa q_z.
\end{align}
In fact, given $z$, let us take $x \in B_{1/2}(z)$ such that $d(x) = |x-z| = (x-z) \cdot \nu_z$, where we will choose $x$ close enough to $z$, later. Clearly, for any such $x$, if $z'$ is close enough to $z$, it holds $\frac{1}{2}|x-z'| \le |x-z| \le 2 |x-z'|$. Then, as a consequence of the expansion \eqref{eq:expansion-Hopf}, we have
\begin{align*}
|q_{z} b_{\nu_z}((x-z) \cdot \nu_z) - q_{z'} b_{\nu_{z'}}((x - z') \cdot \nu_{z'})| &\le |q_{z} b_{\nu_z}((x-z) \cdot \nu_z) - u(x)| + |u(x) - q_{z'} b_{\nu_{z'}}((x - z') \cdot \nu_{z'})| \\
&\le C_0C_2(|x-z|^{s+\eps} + |x-z'|^{s+\eps}) \le 4 C_0C_2 |x-z|^{s+\eps},
\end{align*}
where we wrote $C_0 := \left(\Vert u \Vert_{L^{\infty}(B_1)} + \Vert f \Vert_{L^{\infty}(B_1 \cap \Omega)} \right)$.
Hence, by the bounds for the 1D barriers (see \autoref{lemma:1d-barrier}), and the properties of $x$, we deduce
\begin{align*}
q_{z'} &\ge c |x-z'|^{-s} q_{z'}  b_{\nu_{z'}}((x - z') \cdot \nu_{z'}) \\
&\ge c |x-z|^{-s} \Big(q_{z}  b_{\nu_{z}}((x - z) \cdot \nu_{z}) - |q_{z'}  b_{\nu_{z'}}((x - z') \cdot \nu_{z'}) - q_{z}  b_{\nu_{z}}((x - z) \cdot \nu_{z})| \Big) \\
&\ge c_1 q_z - c_2 |x-z|^{\eps}
\end{align*}
for some $c_1,c_2 > 0$. Thus, upon choosing $x$ close enough to $z$, we get \eqref{eq:boundary-cover} with $\kappa := \frac{c_1}{2c_2} > 0$.\\
Since $B_{1/2} \cap \partial \Omega$ is compact, there exist $N \in \N$ and $z_1, \dots, z_N$ such that
\begin{align*}
\overline{B_{1/2} \cap \partial \Omega} \subset \bigcup_{i = 1}^N  B_{\delta_{z_i}}(z_i).
\end{align*}
Therefore, $\inf_{z \in B_{1/2} \cap \partial \Omega} q_z = \inf_{i = 1,\dots,N} \kappa q_{z_i} > 0$, as desired.

Step 3: Having proved that it must hold $\overline{q} := \inf_{z \in B_{1/2} \cap \partial \Omega} q_z > 0$, we can now conclude the proof. 
Indeed, note that by \eqref{eq:expansion-Hopf} for any $x \in B_{1/2}$, taking $z \in B_{1/2} \cap \partial \Omega$ such that $d(x) = |x-z| = (x-z) \cdot \nu_z$, we have
\begin{align*}
u(x) \ge q_z b_{\nu_z}((x-z) \cdot \nu_z) -|u(x) - q_z b_{\nu_z}((x-z) \cdot \nu_z)| \ge q_z c_1 |x-z|^s - C_2 C_0 |x-z|^{s+\eps}.
\end{align*}
Since $\inf_{z \in B_{1/2} \cap \partial \Omega} q_z > 0$, there exists a radius $\delta > 0$ such that for any $x \in B_{1/2} \cap \{ d(x) \le \delta \}$ it holds
\begin{align*}
u(x) \ge q_z c_1 |x-z|^s - C_2 C_0 |x-z|^{s+\eps} \ge \frac{\overline{q}}{2c_1} d^s(x).
\end{align*}
Since by the weak Harnack inequality, we clearly have $\inf_{B_{1/2} \cap \{ d(x) > \delta \}} u > 0$, the proof is complete.
\end{proof}

\begin{proof}[Proof of \autoref{thm-Cs-intro}]
The first part was already shown before. The second part follows immediately from \autoref{thm:inhom-Hopf}.
\end{proof}

\section{$C^{1+s}$ regularity for nonlocal obstacle problems}
\label{sec:obstacle}

The previous $C^s$ regularity estimates, as well as the existence of the 1D barriers open the door to proving regularity estimates for the nonlocal obstacle problem, where $L \in \mathcal{L}_n^s(\lambda,\Lambda)$ is a nonlocal operator with an inhomogeneous kernel. In particular, we obtain optimal $C^{1,s}$ regularity estimates for solutions, and $C^{1,\gamma}$ regularity of the free boundary near regular points. 

The following is the main result of this section. Note that it immediately implies \autoref{thm-obst-intro}.

\begin{theorem}
\label{thm:obstacle-problem}
Let $L \in \mathcal{L}_s^n(\lambda,\Lambda)$, $\phi \in C_c^{2,\eps}(\R^n)$ with $\eps > \max\{2s-1 , 0 \}$ and let $u$ be the solution to the obstacle problem
\[\begin{split}
\min \{ L u, u - \phi \} = 0 ~~ \text{ in } \R^n.
\end{split}\]
Denote $C_0:=\Vert \phi \Vert_{C^{2,\eps}(\R^n)}$. Then, 
\begin{itemize}
\item[(i)] $u \in C^{1+s}(\R^n)$, with $C > 0$ depending only on $n,s,\lambda,\Lambda$, such that
\[\begin{split}
\Vert u \Vert_{C^{1+s}(\R^n)} \le CC_0.
\end{split}\]

\item[(ii)] There exist $C >0$ and $\gamma > 0$, depending only on $n,s,\lambda,\Lambda$, such that for any free boundary point $x_0 \in \{ u > \phi\}$ there exist $c_{x_0} \ge 0$, $B : \R \to [0,\infty)$, and $e \in \mathbb{S}^{n-1}$ such that
\[\begin{split}
\left| u(x) - \phi(x) - c_{x_0} B \big( (x - x_0) \cdot e \big) \right| \le C C_0 |x-x_0|^{1+s+\gamma} ~~ \forall x \in B_{1/2}(x_0),
\end{split}\]
where $B' = b$ and $b$ denotes the barrier from \autoref{thm-1D} corresponding to $L$ and $e$.
\item[(iii)] Moreover, if $c_{x_0} > 0$, then the free boundary $\partial \{ u > \phi \}$ is a $C^{1,\gamma}$-graph in a neighborhood of $x_0$, where $\gamma > 0$ depends only on $n,s,\lambda,\Lambda$.
\end{itemize}
\end{theorem}

\begin{remark}
We emphasize that a modification of the proofs in the same way as in \cite{RTW23} would allow to prove all the results in this section for obstacle problems in bounded domains and also to assume obstacles to be less regular under the additional assumption $|\nabla K(y)| \le C |y|^{-1} K(y)$.
\end{remark}

The main reason why the previous theorem could not have been proven before is that the classification of blow-ups near regular points heavily relies on a Liouville theorem in the half-space. Without a Liouville theorem, one can merely show that blow-ups are 1D and that solutions to the obstacle problem with a small right hand side that are almost convex are arbitrarily close to these blow-ups. However, without a Liouville theorem in the half-space, it is not clear whether blow-ups behave nice enough in order to imply smoothness of the free boundary near regular points. Since in \autoref{thm:Liouville-half-space}, we obtain a Liouville theorem in the half-space for general inhomogeneous operators, we can show the following new classification of blow-ups:

\begin{theorem}[Classification of blow-ups]
\label{thm:obstacle-classification-of-blowups}
Let $L \in \mathcal{L}_s^n(\lambda,\Lambda)$. Let $\alpha \in (0, \min \{ s , 1-s \} )$, and $u_0 \in C^{0,1}_{loc}(\R^n)$ be such that:
\begin{itemize}
\item $u_0 \ge 0$ and $D^2 u_0 \ge 0$ in $\R^n$ with $0 \in \partial \{ u_0 > 0 \}$.
\item $u_0$ solves in the distributional sense
\[\begin{split}
L(\nabla u_0) = 0  ~~ \text{ and } ~~ L(D_{h} u_0) \ge 0 ~~ \text{ in } \{ u_0 > 0 \} ~~ \forall h \in \R^n,
\end{split}\]
where $D_h u(x) = \frac{u(x+h) - u(x)}{|h|}$.
\item $u_0$ has controlled growth at infinity, i.e.,
\[\begin{split}
\Vert \nabla u_0 \Vert_{L^{\infty}(B_R)} \le R^{s+\alpha} \quad \textrm{for all}\quad R \ge 1.
\end{split}\]
\end{itemize}
Then, it holds
\[\begin{split}
u_0(x) = \kappa B(x \cdot e)
\end{split}\]
for some $\kappa \ge 0$ and $e \in \mathbb S^{n-1}$, where $B \in C^{1+s}(\R)$ is given by $B(x) = \int_{-\infty}^x \tilde{b}(t) \d t$, where $\tilde{b}$ denotes the barrier from \autoref{thm-1D} corresponding to $L$ and $e$ and satisfies
\begin{align}
\label{eq:obstacle-blowup-growth}
\frac{c_1}{1+s} (x_+)^{1+s} \le B(x) \le \frac{c_1}{1+s}(x_+)^{1+s}.
\end{align}
\end{theorem}

\begin{proof}[Proof of \autoref{thm:obstacle-classification-of-blowups}]
We closely follow the proof in \cite[Proposition 4.4.3]{FeRo24}, indicating the difference that occur due to the inhomogeneity of the kernel $K$. First, note that $\{ u_0 > 0 \}$ is convex. Then, by the same arguments as in \cite[Proof of Proposition 4.4.3]{FeRo24}, using the Liouville theorem with growth in the full space (see \cite[Corollary 2.4.14]{FeRo24}), we can prove that the set $\{ u_0 = 0 \}$ must contain a convex cone $\Sigma$ with nonempty interior and $0 \in \Sigma$. Moreover, in this case we can argue again as in \cite[Proof of Proposition 4.4.3]{FeRo24} with the help of the classification of positive solutions to nonlocal equations in cones (see \cite[Theorem 4.4.5]{FeRo24}) that there exist $e \in \mathbb{S}^{n-1}$ and $U \in C^1(\R)$ such that 
\begin{align*}
u_0(x) = U(x \cdot e), \qquad \{ u_0 > 0 \} = \{ x \cdot e > 0 \}.
\end{align*}
In particular, recalling the assumptions on $u_0$, we see that the function $w(x) := \partial_e u_0(x)$ solves in the distributional sense
\begin{align*}
\begin{cases}
L w &= 0~~ \text{ in } \{ x \cdot e > 0 \},\\
w &= 0 ~~ \text{ in } \{ x \cdot e \le 0 \},\\
|w(x)| &\le C |x|^{s+\alpha}, ~~ \forall x \in \R^n,
\end{cases}
\end{align*}
where $\alpha < \eps_0$. 
Thus, we can apply the Liouville theorem in the half-space for the nonlocal operator $L$ (see \autoref{thm:Liouville-half-space}) and deduce that there exists $\kappa \in \R$ such that $w = \kappa b$, where $b(x) = \tilde{b}(x \cdot e)$ is the half-space solution from \autoref{thm-1D} corresponding to $L$ and $e$. From here, the desired result follows immediately, after integrating $w$ in $e$. The desired properties of $B$ can be read off immediately from \autoref{thm-1D}.
\end{proof}

As a direct corollary of the classification of blow-ups, we obtain the following quantitative estimate on the closeness of a solution to the obstacle problem to the blow-up (see \cite[Proposition 4.4.14]{FeRo24}).

\begin{corollary}[Quantitative closeness to the blow-up]
\label{lemma:prop4.4.14}
Let $L \in \mathcal{L}_s^n(\lambda,\Lambda)$. let $\alpha \in (0,\min\{s,1-s\})$ and let $\delta > 0$ and $R_0 > 1$. Then, there is $\eta > 0$, depending only on $n,s,\lambda,\Lambda,\alpha,\delta_0,R_0$, such that the following holds true:\\
Let $u \in C^{0,1}(\R^n)$ such that
\begin{itemize}
\item[(i)] $\min \{ Lu - f , u \} = 0$ in $\R^n$ in the distributional sense, and $|\nabla f| \le \eta$,
\item[(ii)] $u \ge 0$ and $D^2 u \ge - \eta \mathrm{Id}$ in $\R^n$, and $0 \in \partial\{ u > 0 \}$,
\item[(iii)] $\Vert \nabla u \Vert_{L^{\infty}(B_R)} \le R^{s+\alpha}$ for every $R \ge 1$.
\end{itemize}
Then, we have
\begin{align}
\label{eq:close-to-blowup}
\Vert u - \kappa B(x \cdot e) \Vert_{C^{0,1}(B_{R_0})} \le \delta
\end{align}
for some $e \in \mathbb{S}^{n-1}$ and $\kappa \ge 0$, where $B$ is as in \autoref{thm:obstacle-classification-of-blowups}.
\end{corollary}

\begin{proof}
First, we observe that by the exact same arguments as in \cite[Proposition 4.4.10]{FeRo24}, we can show that
\begin{align}
\label{eq:obstacle-ao-reg}
\Vert u \Vert_{C^{1+s-\eps}(B_R)} \le C_{\eps,R} ~~ \forall R > 0, ~~ \forall \eps > 0,
\end{align}
where $C_{\eps,R} > 0$ only depends on $n,s,\lambda,\Lambda,\alpha,\eps,R$. Indeed, the proof is a contradiction compactness argument, using the classification of blow-ups (see \autoref{thm:obstacle-classification-of-blowups}), and in particular their growth (see \eqref{eq:obstacle-blowup-growth}). Having this result at hand, we can prove \autoref{lemma:prop4.4.14} by following the lines of \cite[Proposition 4.4.14]{FeRo24}. We assume that there exists $\eta_k \searrow 0$, $(L_k)_k \subset \mathcal{L}_s^n(\lambda,\Lambda)$, $(u_k)_k$ satisfying the assumptions (i), (ii), and (iii) such that \eqref{eq:close-to-blowup} does not hold for any $e \in \mathbb{S}^{n-1}$ and $\kappa \ge 0$. Then, by \eqref{eq:obstacle-ao-reg}, the sequence $(u_k)_k$ converges locally uniformly in the $C^{1+s-\eps}$-norm (up to a subsequence) to a function $u \in C^{1+s-\eps}$ satisfying the assumptions (i), (ii), and (iii) for some operator $L \in \mathcal{L}_s^n(\lambda,\Lambda)$ with $\eta = 0$. In particular, we can argue as in \cite{FeRo24} that $u$ satisfies all the assumptions of \autoref{thm:obstacle-classification-of-blowups}, which yields a contradiction.
\end{proof}

The second main ingredient in the proof of \autoref{thm:obstacle-problem} (in addition to the classification of blow-ups) is the following result, stating that flatness of the free boundary implies that it is $C^{1,\gamma}$. Here, flatness is measured by closeness of the solution to the blow-up $B$ in the half-space from \autoref{thm:obstacle-classification-of-blowups}.

\begin{theorem}[Flatness implies $C^{1,\gamma}$]
\label{prop:4.4.15}
Let $L \in \mathcal{L}_s^n(\lambda,\Lambda)$. 
Let $\alpha \in (0,\min\{s,1-s\})$ and $\kappa_0 > 0$. 
Then, there are $\eps > 0$, $R_0 > 1$, and $\eta > 0$, depending only on $n,s,\lambda,\Lambda,\alpha,\kappa_0$, such that the following holds true:

Let $u \in C^{0,1}_{\rm loc}(\R^n)$ be such that
\begin{itemize}
\item[(i)] $\min \{ Lu - f , u \} = 0$ in $B_{1}$ in the distributional sense, and $|\nabla f| \le \eta$,
\item[(ii)] $u \ge 0$ and $D^2 u \ge - \eta \mathrm{Id}$ in $\R^n$, and $0 \in \partial\{ u > 0 \}$,
\item[(iii)] $\Vert \nabla u \Vert_{L^{\infty}(B_R)} \le R^{s+\alpha}$ for every $R\ge 1$,
\item[(iv)] $\Vert u - \kappa B_e(x \cdot e) \Vert_{C^{0,1}(B_{R_0})} \le \eps$ for some $\kappa \ge \kappa_0$ and $e \in \mathbb{S}^{n-1}$, where $B_e$ is as in \autoref{thm:obstacle-classification-of-blowups}.
\end{itemize}
Then, the free boundary $\partial \{ u > 0 \}$ is a $C^{1,\gamma}$-graph in $B_{1/2}$, and moreover $u\in C^{1+s}(B_{1/2})$ with
\[\begin{split}
\Vert \nabla u \Vert_{C^s(B_{1/2})} \le C, \quad |\nabla u(x) - \kappa' b_{e'}(x \cdot e')| \le C |x|^{s+\gamma} ~~ \forall x \in B_{1/2},
\end{split}\]
where $\kappa' \in \R$, $e'\in \mathbb{S}^{n-1}$ denotes the normal vector of $\partial \{ u > 0 \}$ at $0$, and $b_{e'}$ is as in \autoref{thm-1D}.
The constants $C > 0$ and $\gamma > 0$ depend only on $n,s,\lambda,\Lambda,\alpha,\kappa_0$.
\end{theorem}

\begin{proof}[Proof of \autoref{prop:4.4.15}]
We define $u_0(x) = \kappa B(x \cdot e)$. By assumption (iv), for any $e' \in \mathbb{S}^{n-1}$ with $e' \cdot e \ge \frac{1}{\sqrt{2}}$ we have
\begin{align*}
|\partial_{e'} u - \partial_{e'} u_0| \le \eps_0 ~~ \text{ in } B_{R_0}.
\end{align*}
Moreover, by the definition of $u_0$, we have $\partial_{e'}u_0(x) = \kappa B'(x \cdot e) (e \cdot e') = \kappa \tilde{b}(x \cdot e) (e \cdot e') \ge 2^{-1/2} \kappa \tilde{b}(x \cdot e)$, where $\tilde{b}$ is the 1D barrier from \autoref{lemma:1d-barrier}. From the properties of $\tilde{b}$, we deduce:
\begin{align*} 
\partial_{e'} u_0 \ge 0 ~~ \text{ in } \R^n, \qquad \partial_{e'} u_0 \ge c_1 \kappa ~~ \text{ in } \{ x \cdot e \ge 2^{-1/2} \}.
\end{align*}
Thus, given any $\rho_0 > 0$, we can apply \cite[Lemma 4.4.13]{FeRo24} to $\partial_{e'}u$, and deduce that (upon choosing $R_0$ large, and $\eps, \eta > 0$ small enough, respectively, depending on $\rho_0$) 
\begin{align*}
\partial_{e'} u \ge 0 ~~ \text{ in } B_{\rho_0} \qquad \forall e' \in \mathbb{S}^{n-1} ~~ \text{ with } e \cdot e' \ge 2^{-1/2},
\end{align*}  
which implies that $\partial \{ u > 0 \} \cap B_{\rho_0}$ is a Lipschitz graph with Lipschitz constant bounded by $1$. Arguing as in \cite[Proposition 4.4.15]{FeRo24}, by the boundary Harnack principle (see \cite[Theorem 4.3.1]{FeRo24}), and upon choosing $\rho_0 > 1$ large enough, we deduce that the free boundary $\partial \{ u > 0 \} \cap B_{1/2}$ is a $C^{1,\gamma}$-graph, as desired. Finally, the $C^s$ estimate and the expansion for $\nabla u$ follows from the boundary regularity in $C^{1,\gamma}$ domains that we established in \autoref{thm:inhom-Cs}.
\end{proof}

Now, we are in position to give a proof of the main result (see \autoref{thm:obstacle-problem}).

\begin{proof}[Proof of \autoref{thm:obstacle-problem}]
The proof of (i) follows exactly as in \cite[Theorem 4.5.1]{FeRo24}, combining \autoref{lemma:prop4.4.14} and \autoref{prop:4.4.15} together with interior regularity estimates. To prove (ii) and (iii), we proceed as in \cite[Proposition 4.5.2]{FeRo24}. We can assume without loss of generality that $x_0 = 0$ and $C_0 \le 1$, and upon defining $w = \eta (u-\phi)$ as in \cite[Proposition 4.5.1]{FeRo24}, we find that $w$ satisfies the assumptions of \autoref{lemma:prop4.4.14}. First, we assume that
\begin{align*}
\Vert \nabla w \Vert_{L^{\infty}(B_r)} \le r^{s+\alpha} ~~ \forall r \in (0,1).
\end{align*}
In that case, $|u(x)| \le C |x|^{1+s+\alpha}$, and the conclusions in (ii) and (iii) hold true with $c_0 = 0$. Otherwise, we define
\begin{align*}
r_1 := \sup \{ r > 0 : \Vert \nabla w \Vert_{L^{\infty}(B_r)} \ge r^{s+\alpha} \} \in (0,1),
\end{align*}
and observe that then the rescaled function $w_{r_1}(x) = w(r_1 x) r_1^{-(1+s+\alpha)}$ satisfies
\begin{align*}
\Vert \nabla w_{r_1}  \Vert_{L^{\infty}(B_{1})} =1, \qquad \Vert \nabla w_{r_1} \Vert_{L^{\infty}(B_R)} \le R^{s+\alpha} ~~ \forall R > 1.
\end{align*}
Hence, by \autoref{lemma:prop4.4.14} and \autoref{prop:4.4.15} we deduce that the free boundary $\partial \{ u > 0 \} \cap B_{1/2}$ is in $C^{1,\gamma}$ for some $\gamma \in (0,\alpha]$, and moreover
\begin{align*}
|\nabla w_{r_1}(x) - \kappa b^{(r_1)}(x \cdot e)| \le C |x|^{s+\gamma} ~~ \forall x \in B_{1/2},
\end{align*}
where $b^{(r_1)}$ denotes the barrier from \autoref{thm:inhom-Cs} associated with the operator $L^{(r_1)}$ with kernel $K^{(r_1)}(h) = r_1^{n+2s} K(r_1 h)$ and $e \in \mathbb{S}^{n-1}$.
Integrating this inequality, we immediately obtain
\begin{align*}
|w_{r_1}(x) - \kappa B^{(r_1)}(x \cdot e)| \le C |x|^{1+s+\gamma} ~~ \forall x \in B_{1/2},
\end{align*}
By rescaling from $w_{r_1}$ back to $w$ and using that $\alpha \ge \gamma$, we deduce
\begin{align*}
|w(x) - \kappa r_1^{\alpha} [r_1^{1+s} B^{(r_1)}(r_1^{-1} x \cdot e)]| \le C |x|^{1+s+\gamma} r_1^{\alpha-\gamma} \le C |x|^{1+s+\gamma} ~~ \forall x \in B_{r_1/2}.
\end{align*}
Moreover, recall from the note that by the growth condition on $B^{(r_1)}$ (see \eqref{eq:obstacle-blowup-growth}), we have
\begin{align*}
|r_1^{\alpha}[r_1^{1+s}B^{(r_1)}(r_1^{-1}x \cdot e)]| \le C r_1^{\alpha}|x|^{1+s} \le C |x|^{1+s+\alpha} \le C |x|^{1+s+\gamma} ~~ \forall x \in B_{1/2} \setminus B_{r_1/2}.
\end{align*}
Since by the definition of $w_{r_1}$, we also have $\Vert w_{r_1} \Vert_{L^{\infty}(B_R)} \le C R^{1+s+\alpha}$ for any $R > 0$ and thus
\begin{align*}
|w(x)| \le C |x|^{1+s+\alpha} r_1^{1+s+\alpha} \le C |x|^{1+s+\gamma}  ~~ \forall x \in B_{1/2} \setminus B_{r_1/2},
\end{align*}
we deduce that
\begin{align*}
|w(x) - \kappa r_1^{\alpha} [r_1^{1+s} B^{(r_1)}(r_1^{-1} x \cdot e)]| \le C |x|^{1+s+\gamma} ~~ \forall x \in B_{1/2} \setminus B_{r_1/2}.
\end{align*}
Finally, we observe that by \autoref{rem:unique-barrier}, it holds $r_1^s b^{(r_1)}(r_1^{-1}x) = \kappa_1 b(x)$ for some $\kappa_1 > 0$, where $b$ corresponds to $K$. Hence, $r_1^{1+s}B^{(r_1)}(r_1^{-1}x) = \kappa_1 B(x)$. This implies the desired expansion in (ii) with $c_0 = \kappa r_1^{\alpha} \kappa_1 > 0$, and thus also (iii) is proved. Hence, the proof is finished.
\end{proof}

\begin{proof}[Proof of \autoref{thm-obst-intro}]
This follows immediately from \autoref{thm:obstacle-problem}.
\end{proof}

\end{document}